\documentclass[12pt]{amsart} 
\usepackage{amssymb}
\usepackage{amsmath}
\usepackage{latexsym}
\usepackage{amscd}
\usepackage{amsfonts}
\usepackage{pb-diagram}
\usepackage[all]{xypic}
\usepackage[mathcal,mathscr]{eucal}
\usepackage{amsthm}
\usepackage{graphicx}
\usepackage{url}
\usepackage{pdfsync}
\usepackage{bbm}
\usepackage{xypic}
\usepackage[pdftex,colorlinks]{hyperref}
\usepackage{pifont}
\usepackage{amsbsy}
\usepackage{epstopdf}
\usepackage{times}
\usepackage{epic,eepic}

\newtheorem{theorem}{Theorem}[section]
\newtheorem{prop}[theorem]{Proposition}
\newtheorem{lemma}[theorem]{Lemma}
\newtheorem{corollary}[theorem]{Corollary}

\theoremstyle{remark}
\newtheorem{remark}[theorem]{\bf {Remark}}
\newtheorem{example}[theorem]{\bf {Example}}

\numberwithin{equation}{section}

\DeclareMathOperator{\Crit}{Crit}

\DeclareMathOperator{\Hom}{Hom}

\DeclareMathOperator{\End}{End}

\DeclareMathOperator{\Ker}{Ker}

\DeclareMathOperator{\Tc}{Tc}

\DeclareMathOperator{\TP}{\mathrm{TP}}

\DeclareMathOperator{\Tr}{Tr}

\DeclareMathOperator{\Sch}{Sch}

\DeclareMathOperator{\grad}{grad}

\DeclareMathOperator{\ch}{ch}
\DeclareMathOperator{\Tch}{Tch}

\DeclareMathOperator{\tr}{tr}

\DeclareMathOperator{\id}{id}

\DeclareMathOperator{\codim}{codim}

\DeclareMathOperator{\wstr}{w-str}

\DeclareMathOperator{\Real}{Re}

\DeclareMathOperator{\Imag}{Im}

\newcommand\bC{{\mathbb C}}

\newcommand{\bN}{{{\mathbb N}}}

\newcommand{\bP}{{{\mathbb P}}}

\newcommand\bR{{\mathbb R}}

\newcommand\bZ{{\mathbb Z}}

\newcommand{\ul}{\underline}

\newcommand{\ra}{\rightarrow}

\newcommand{\ve}{\varepsilon}
\newcommand{\vre}{\varepsilon}

\begin{document}

\title{Non-tame Morse-Smale flows and odd Chern-Weil Theory}

\author{Daniel Cibotaru}
\address{Universidade Federal do Cear\'a, Fortaleza, CE, Brazil}
\email{daniel@mat.ufc.br}
\thanks{Partially supported by the CNPq Universal Project}
\author{Wanderley Pereira}
\address{Universidade Estadual do Cear\'a, Limoeiro do Norte, CE, Brazil}
\email{wanderley.pereira@uece.br}

\subjclass[2010]{Primary 58A25,  49Q15; Secondary 53C05.}
\begin{abstract} Using a certain well-posed ODE problem introduced by Shilnikov in the sixties, G. Minervini proved  in his PhD thesis \cite{M}, among other things,  the Harvey-Lawson Diagonal Theorem but without the restrictive tameness condition for Morse flows.  Here we combine the same techniques with the ideas of Latschev in order to construct local resolutions for the flow of the graph of a section of a fiber bundle. This is endowed with a vertical vector field which is horizontally constant and Morse-Smale in every fiber. The resolution allows the removal of the tameness hypothesis from the homotopy formula in \cite{Ci2}. We give one finite and one infinite dimensional application. For that end, we introduce  closed smooth forms of odd degree associated to any triple $(E,U,\nabla)$ composed of a hermitian vector bundle, unitary endomorphism and metric compatible connection.  

\end{abstract}
\maketitle
\tableofcontents
\section{Introduction}

The well-known Morse Lemma gives the canonical form of a Morse function $f$ on a compact, Riemannian manifold $(M,g)$ around a critical point but does not provide information about the gradient flow. On the other hand, the Hartman-Grobman Theorem gives the \emph{topological} conjugacy class of the gradient flow around the critical point. However, there are important situations where both of these classical results are insufficient to answer the relevant questions. We have in mind the following context. Suppose  one is interested in taking a  smooth submanifold $S$  and "flow it through" the critical point. Let us think that $S$ lies within a regular level $c-\epsilon$ of the Morse function right "before" a critical level and we look at its "trace" at a regular level $c+\epsilon$, meaning the intersection of the (forward) flow lines determined by $S$ with level $c+\epsilon$.  Obviously, this "trace" can be empty if $S$ is contained in the stable manifold of the critical point. So a transversality condition with the stable manifold is naturally imposed. The natural question is whether  one say anything about the  structure of the \emph{closure} of the "trace" at level $c+\epsilon$? One expects to get   at least a rectifiable set because of transversality.  One might even hope to prove something stronger, namely the existence manifold with corners of the same dimension as the submanifold and a  proper "projection"  which  maps to the closure of the trace and is one-to-one almost everywhere.  It turns out that in order to make this rigorous  a tameness condition on the triple $(M,g,f)$ is  helpful. One such condition was introduced in \cite{HL1}.  A Morse function $f$ is called \emph{tame} if around each critical point one can find coordinates for which two requirements are met, the metric is flat and the Morse function has the canonical form of the Morse Lemma. An immediate consequence of tameness is that the eigenvalues of the Hessian are $\pm 1$. This gives an idea of how restrictive tameness is. On the positive side, the flow  has the simplest form possible and one can prove quite easily, by performing a blow-up of the intersection of the submanifold $S$ with the stable manifold  of the critical point that a resolution of the closure of the "trace" is available. In fact, one can prove that such a manifold with corners resolution is available also for the closure of the entire "flow-out" of the submanifold between levels $c-\epsilon$ and $c+\epsilon$.  More general situations are contemplated in \cite{Ci2,La}. 
 
 The existence of such resolutions have important consequences. The Harvey-Lawson Diagonal  Theorem says that for the  gradient flow $\varphi:\bR\times M\ra M$ induced by a tame   $f$, which additionally  satisfies Smale's transversality condition  there exists a rectifiable current $T$ on $M\times M$ such that
 \[ dT= \Delta-\sum_{p\in\Crit(f)}U_p\times S_p,
 \]
 where $\Delta$ is the diagonal and $U_p$ and $S_p$ are the unstable, resp. the stable manifold of the critical point $p$.  To be a bit more precise  the submanifold in this case is the diagonal and    the flow is on $M\times M$ via $\varphi$ in the first component of the product and keeping fixed the second component.  In his PhD thesis, J. Latschev \cite{La} also used the resolution idea and extended the Harvey-Lawson Diagonal  Theorem to Morse-Bott-Smale flows. The first author developed this point of view further  in \cite{Ci2} in order to extend the results to sections of fiber bundles, satisfying adequate transversality conditions. Even with the tameness condition in place, the rigorous details of the construction of the  resolution are quite involved. Moreover, special care needs  to be taken for those points mapped by the section to the critical points, e.g. the points $(p,p)\in \Delta$ when $p\in \Crit(f)$. 
 
  Completely new ideas are necessary in order to deal with the \emph{non-tame} case. In his PhD thesis, Minervini \cite{M} used a combination of results of 
 Shilnikov \cite{Sh} on a certain type of ODE problems together with objects he introduced, called horned stratified spaces in order to prove the Harvey-Lawson Theorem without the tameness condition. Applications to Morse-Novikov theory were given by Harvey and Minervini in \cite{HM}. In this article, we take the next natural step and remove the tameness condition from the currential homotopy formula of \cite{Ci2}. With one caveat, the (model)  flows in each fiber are assumed here Morse as opposed to Morse-Bott in \cite{Ci2}.  We plan to return to the Morse-Bott case somewhere else. 
 
We implement a combination of the two main ideas from  \cite{La} and \cite{M}  in our present approach. On one hand we  use  Shilnikov-Minervini local analysis of  the closure of the  graph of the flow which gives a local resolution (see Theorem \ref{teo.subvde}), but use induction on the critical levels ala Latschev for the proof of the next homotopy formula.  
\begin{theorem}\label{tlimxi00}
Let $\pi: P\longrightarrow B$ be a fiber bundle with compact fiber.  Let $X$ be a horizontally constant Morse-Smale	vertical vector field and denote by $\Phi:\bR\times P\ra P$ the flow induced by $X$.  Let $s:B\longrightarrow P$ be a section transverse to all the stable manifolds $\mathrm{S}(F)$ associated to the critical manifolds $F$ of $X$ and let $\xi_t(b):=(\Phi_t(s(b)),s(b))$, $b\in B$. Then $$ T=\xi([0,+\infty)\times B)$$ defines a $ (n + 1) $-dimensional  rectifiable current of locally finite mass  and if $B$ is compact then $T$ is of finite mass.
	
	Moreover, the following equality of currents holds in $P\times_BP$: 
	\begin{eqnarray}\label{bordo.T}
	\mathrm{d} T=\sum_{F}\mathrm{U}(F)\times_{F}s(s^{-1}(\mathrm{S}(F)))-(\xi_0)_*(B).
	\end{eqnarray}	
where  $U(F)$ are the  unstable manifolds of $X$.
\end{theorem}
Recall that a vertical vector field $X$ on the total space of a fiber bundle $P\ra B$  is called horizontally constant if there exist local trivializations of the fiber bundle such that $X$ has a zero horizontal component in this trivialization. As a consequence, the flow induce by $X$ is, up  to diffeomorphism, the same  in every fiber. It also means that if the flow in the fiber is Morse-Bott, then the critical sets $F$ of $X$ are manifolds and so are the sets $S(F)$ and $U(F)$.

An immediate consequence (see Corollary \ref{c.principal}) of  Theorem \ref{tlimxi00} is the explicit computation of  limits in the weak sense
\[ \lim_{t\ra \infty} s_t^*\omega\]
 where $s_t:=\Phi_t\circ s$ and $\omega\in \Omega^*(P)$ while also justifying a transgression formula for closed forms $\omega$
\begin{equation}\label{0eq0}\lim_{t\ra \infty} s_t^*\omega-s^*\omega=dT(\omega).
\end{equation}

This Poincar\'e duality type of result  is a source of many applications (see \cite{Ci4}) even in the tame case. In an early pre-print of \cite{Ci2} posted on arXiv  an application  to (\ref{0eq0}) concerning certain odd degree forms on the unitary group was included. The flow used however did not satisfy the tameness hypothesis and the application was removed from the published version.   We present it here in a more general context, but not before revisiting a classical topic and introducing some new objects which seem of independent interest.
\vspace{0.2cm}

Chern-Weil theory is an important source of closed forms arising from geometric data. To any  complex vector bundle $E\ra B$ of rank $n$ endowed with a connection $\nabla$ and a $GL(n)$ invariant polynomial $P$ in the entries of an $n\times n$ matrix one has an associated closed form $P(F(\nabla))$. For homogeneous $P$ one gets that $P(F(\nabla))$ is of \emph{even} degree, more precisely twice the degree of $P$. The deRham cohomology class of $P(F(\nabla))$ does not depend on $\nabla$.  

 In order to get odd degree forms we endow $E\ra B$ with an automorphism  $A:E\ra E$. Then we associate to the quadruple $ (E,A,\nabla,P)$ a closed form $\TP(E,A,\nabla)$ which satisfies the following properties: it is natural with respect to pull-back, the cohomology class determined by $\TP(E,A,\nabla)$ does not depend on  the connection $\nabla$, the same cohomology class does not change under deformations of $A$ in the same homotopy class. We prove all these properties in Section \ref{OCW} for hermitian vector bundles but the interested reader can adapt the results without difficulty to other structure groups.
 
 Let $P=c_k$ be the invariant polynomial induced by the $k$-th elementary symmetric polynomial. The following statement, which generalizes a result of Nicolaescu (\cite{Ni}, Prop. 57) also gives a description of the Poincar\'e duals to $\Tc_k(E,g,\nabla)$.
 
 \begin{theorem}\label{Nico0} Let $E\ra B$ be a trivializable hermitian vector bundle of rank $n$ over an oriented manifold with corners $B$ endowed with a compatible connection. Let $g:E\ra E$ be a smooth gauge transform.  Suppose that  a complete flag $E=W_0\supset W_1\supset \ldots \supset W_n=\{0\}$  (equivalently a trivialization of $E$) has been fixed such that $g$ as a section of $\mathcal{U}(E)$ is completely transverse to certain (see (\ref{DefS}))  submanifolds $S(U_{I})$ determined by the flag. Then, for each $1\leq k\leq n$ there exists a  flat current $T_k$ such that the following equality of  currents of degree $2k-1$ holds:
\begin{equation}\label{TCkE0} \Tc_k(E, g,\nabla)-g^{-1}(S(U_{\{k\}}))=dT_k.
\end{equation}
where 
\begin{eqnarray*}g^{-1}(S(U_{\{k\}}))=\{b\in B~|~\dim{\Ker(1+g_b)}=\dim{\Ker{(1+g_b)\cap (W_{k-1})_b}}=1,\qquad\\
 \dim{\Ker{(1+g_b)}\cap (W_{k})_b}=0\}.\end{eqnarray*}
In particular, when $B$ is compact without boundary, then $\Tc_k(E, g,\nabla)$ and $g^{-1}(S(U_{\{k\}}))$ are Poincar\'e duals to each other and (\ref{TCkE0}) is a spark equation (\cite{HLZ4,CS}).
\end{theorem} 

The condition that $E\ra B$ be trivializable is related to the non-tame flow used in the proof which requires the existence of a complete flag $E=W_0\supset\ldots \supset W_n=\{0\}$ of vector subbundles. It is an interesting question of how one can describe the Poincar\'e duals to $\Tc_k(E,g,\nabla)$ for a general $E$. 

The next application is to families of self-adjoint Fredholm operators. Fix $H$ a Hilbert space. The space of unitary operators $U\in \mathcal{U}(H)$ such that $1+U$ is Fredholm is a classifying space for odd $K$-theory. This space is a Banach manifold but is "too big" to build smooth differential forms. Restricting the attention to the Palais classifying spaces $\mathcal{U}^p$ which are unitary operators of type $1+S$ where $S$ belongs to some Schatten ideal, e.g. trace class or Hilbert-Schmidt operators then Quillen \cite{Qu} was able to construct several families of smooth forms all representing the components of the odd Chern character.  When one has  a smooth family of Dirac operators $\mathcal{D}_{b\in B}$ parametrized by a smooth and finite dimensional manifold $B$ then by taking the Cayley transforms one gets a smooth map $\varphi:B\ra \mathcal{U}^p$.  The pull-backs of the Quillen forms   compute the cohomological analytic index determined of the family.

Let us remark that in the finite dimensional case, the Quillen forms on $U(n)$ have explicit formulas  in terms of the odd Chern-Weil forms  arising from the trivial vector bundle $\bC^n$ over $U(n)$ endowed with the tautological unitary endomorphism and trivial connection, i.e. in terms of the standard deRham generators of the cohomology ring of $U(n)$.

On the other hand, in \cite{Ci1} we produced explicit representatives for the Poincar\'e duals of these classes using the infinite dimensional analogues of the stable manifolds $S(U_{\{k\}})$ which appear in Theorem \ref{Nico0}. We used sheaf theory in \cite{Ci1} in order to be able to define cohomology classes arising from certain stratified spaces, called quasi-manifolds on an infinite dimensional Banach manifold. Here we exchange the sheaf theoretical approach from \cite{Ci1}   with  the currential approach and show that under the expected transversality hypothesis one can produce a transgression formula, strengthening thus the results from \cite{Ci1}.

\begin{theorem}\label{0thm71} Let $\varphi:B\ra \mathcal{U}^p$ be a smooth map from a compact, oriented manifold $B$, possibly with corners such that $\varphi\pitchfork Z_I^p$ for every $I$. Let $\Omega_k$ be a Quillen form of degree $2k-1$ that makes sense on $\mathcal{U}^p$. Then for every such $\Omega_k$, there exists a  flat current $T_k$ such that:
  \begin{equation}\label{lasteq} \varphi^{-1}Z_{\{k\}}-(-1)^{k-1}(k-1)!\varphi^*\Omega_k=dT_k.\end{equation}
  In particular, when $B$ has no boundary, $\frac{(-1)^{k-1}}{(k-1)!}\varphi^{-1}Z_{\{k\}}^p$ represents the Poincar\'e dual of $\ch_{2k-1}([\varphi])$, where $[\varphi]\in K^{-1}(B)$ is the natural odd $K$ theory class determined by $\varphi$.
\end{theorem}
The proof of this result reduces to Theorem \ref{Nico0} via symplectic reduction. \vspace{0.3cm}

A few more comments about the structure of the article are in order. Section \ref{s.BVP}  which revisits Shilnikov theory, also adds  some details  to Minervini's  presentation in \cite{M}. In particular, Theorem \ref{teo.vizinhanca} introduces some flow-convex neighborhoods that are  fundamental later on.      The main technical part  of the proof of the main Theorem \ref{tlimxi00} is contained in the rather long Section \ref{tec.tools}.  We felt it necessary to present many complete arguments. The proof is by induction and the amount of data one has to carry from one step to another is quite substantial. That is why we paid special care in proving properties like properness  or injectivity of the flow-resolution map. To get a feel for the level of technicality the reader can take a quick glance at Proposition \ref{model.prop} which is the key step in the induction. In essence, the main idea of the proof of the main Theorem is to follow the same steps as the induction proof presented in the Appendix of \cite{Ci2} but to substitute the oriented blow-up technique which takes care of the local picture in \cite{Ci2} with Minervini's Theorem 1.3.21 which appears here as Theorem \ref{teo.subvde}. The advantage of the presentation in \cite{Ci2} via blow-ups is that several maps are explicit and several properties come for free (e.g. a blow-down map is proper). This, of course is a consequence of tameness. On the negative side, one works hard in \cite{Ci2} to show that the relevant maps have regularity $C^1$ while here the regularity is $C^{\infty}$ and it is a consequence of the Minervini-Shilnikov theory.

 The models in the fiber are classical Morse-Smale flows associated to gradients of Morse functions. The results ought to hold  also for the Morse-Smale quasi-gradients as defined in \cite{LM}. One point that made us cautious  is contained in Remark \ref{Xtrans} and is related to the properties of the flow-convex neighborhoods of  Theorem \ref{teo.vizinhanca}. 

A proof of the main Theorem appeared in the PhD thesis \cite{Oli} of the second author. Some arguments have been simplified in this presentation.

\section{Minervini-Shilnikov theory}\label{s.BVP}
We   review some results about  certain  well-posed ODE problems studied by Shilnikov in the 60's.  We borrowed the terminology that gives the title of this section from the main reference \cite{M}. Where the complete proofs were skipped, the reader will find the details in Chapter 1 of \cite{M}.
  
 Let $(x,y)$ be coordinates in $\mathbb{R}^s\times\mathbb{R}^u$. With respect to this decomposition, let $L=\begin{bmatrix}
  L^- & 0 \\
  0 & L^+
  \end{bmatrix}$   be a constant, real coefficients matrix, in which the real parts of the eigenvalues of $L^-$ are strictly negative, say $-\lambda_s\leq\ldots\leq-\lambda_1<0$, and those of $L^+$ are strictly positive, say $0<\mu_1\leq \ldots \leq \mu_u$.  For the situation we are interested in, $L$ is symmetric. 
 
 Consider the ODE system in $\bR^s\times \bR^u$:
  \begin{equation}\label{sist.original}
  \left\{\begin{array}{lll}
  \dot{x}=L^-x+f(x,y)\\
  \dot{y}=L^+y+g(x,y)\\
  \end{array}
  \right.\end{equation}
  where 
  $F=(f,g):\mathbb{R}^s\times\mathbb{R}^u\longrightarrow\mathbb{R}^s\times\mathbb{R}^u$ is a differentiable function satisfying
  $$F(0,0)=(0,0) \; \; \mbox{and} \; \; dF(0,0)=(0,0).$$

  Given a triple $(x_0,y_1,\tau)\in \bR^s\times \bR^u \times [0,+\infty)$   a Boundary Value Problem (BVP) for the ODE \eqref{sist.original}  has the following form
  \begin{equation}\label{PVB}
\left\{\begin{array}{lll}
\dot{x}=L^-x+f(x,y)\\
\dot{y}=L^+y+g(x,y)\\
x^{*}(0)=x_0\\
 y^*(\tau)=y_1.
\end{array}
\right.
\end{equation}
  where the solution  $(x^*(t),y^*(t))$ is defined  in the interval $[0,\tau]$.  The solution at time $t$ to the BVP \eqref{PVB} with data
$(x_0,y_1,\tau)$ is denoted
$$(x^*(t,x_0,y_1,\tau),y^*(t,x_0,y_1,\tau)).$$
The "end point" $(x^*_1, y_0^*)$ for the BVP solution is
\begin{equation}\label{aplic.final}
\begin{array}{lll}
x^{*}_1(x_0,y_1,\tau)=x^*(\tau,x_0,y_1,\tau)\\
y^*_0(x_0,y_1,\tau)=y^*(0,x_0,y_1,\tau).
\end{array}
\end{equation}
  
  We compare this with the solution at time $t$ of the Initial Value Problem with data \linebreak $(x_0; y_0, t=0)$ for which the following notation is used
$$(x(t,x_0,y_0),y(t,x_0,y_0)).$$ Notice that
\begin{equation}\label{Identidades}
\begin{array}{lll}
x(t,x_0,y_0) = x^*(t,x_0,y(\tau,x_0,y_0),\tau) \\
y(t,x_0,y_0) = y^*(t,x_0,y(\tau,x_0,y_0),\tau) \\

 x^*(t,x_0,y_1,\tau) = x(t,x_0,y^*(x_0,y_1,\tau))\\
y^*(t,x_0,y_1,\tau) = y(t,x_0,y_0^*(x_0,y_1,\tau)).
\end{array}
\end{equation}
  
Let
\begin{eqnarray}
\delta_{\varepsilon}^k:=\sup_{|x,y|\leq \varepsilon}\sum_{|m|\leq k}\left|\frac{\partial^{|m|}F}{\partial (x,y)^m}\right|<+\infty,
\end{eqnarray}
where $|x,y|:=\mathrm{max}\{|x|,|y|\}$ and $| \; \cdot\; |$ denotes the euclidian norm. 

Quite similarly to the Cauchy problem for ODE and proceeding in the standard way, i.e. writing the BVP as a system of integral equations and using Banach Fixed Point Theorem,  the following general result holds:

\begin{theorem}\label{teoremasolucao}Suppose $\varepsilon>0$ is such that the the estimate $\delta^1_{2\varepsilon}<\mathrm{min}\{\lambda_1, \mu_1\}$ holds. Then the BVP for the system \eqref{PVB} is solvable for any data $(x_0,y_1,\tau)$ in the  "ball" $|x_0,y_1|<\varepsilon$.  The solution is unique, it depends smoothly on  all its arguments and satisfies:
	$$|x^*(t),y^*(t)|\leq 2|x_0,y_1|, \; \; \forall t\in [0,\tau].$$
\end{theorem}
\begin{figure}[h]
	\centering
	%	\captionsetup{width=12.3cm,skip=0pt}
	\caption{{Boundary Value Problem}} \label{fig}
	\includegraphics[scale=0.3]{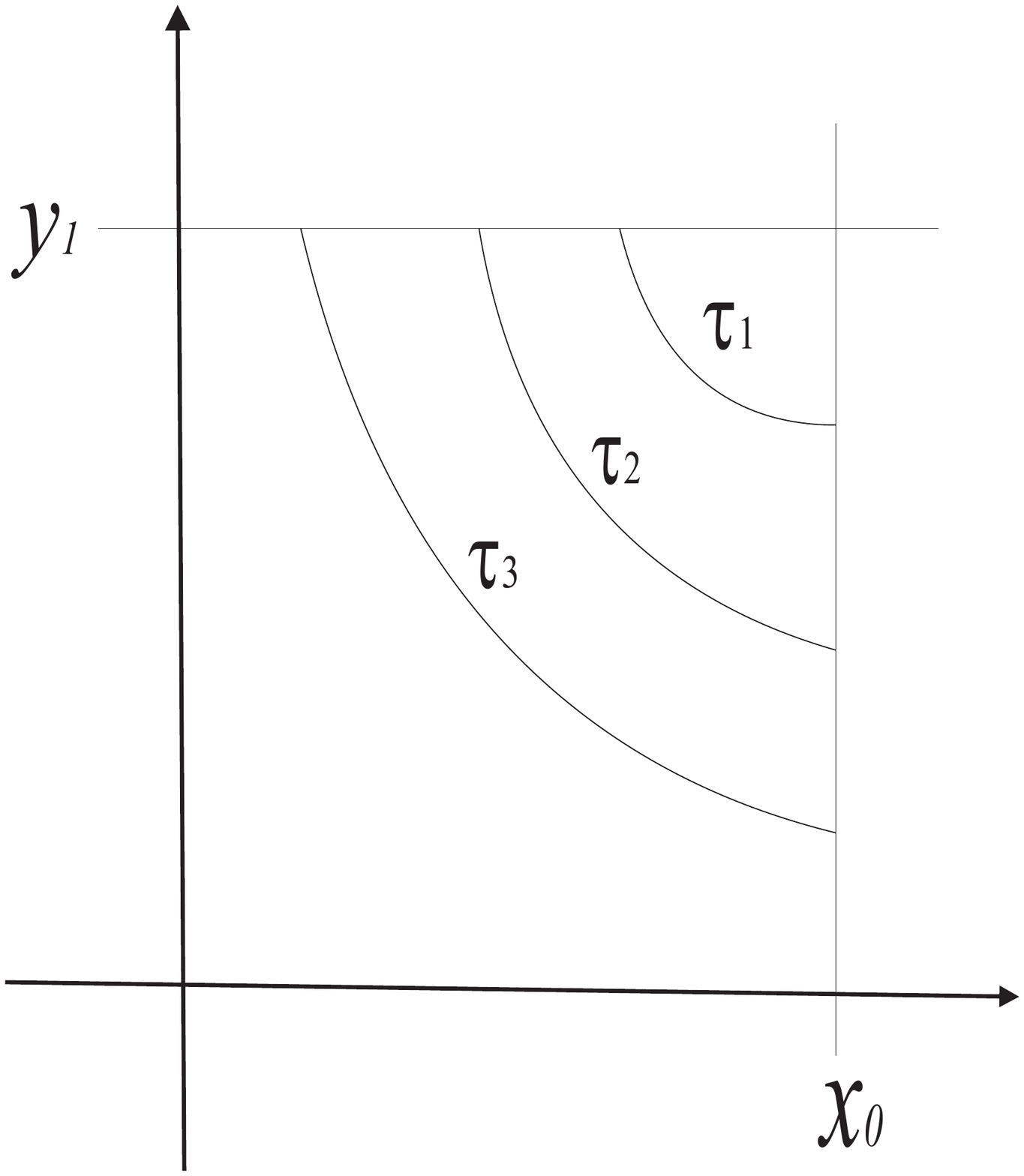}
%		\captionsetup{font=small,position=below,skip=-1pt}
	%	\caption*{Fonte: Imagem do autor.}
\end{figure}

It turns out that the integral equations  equivalent to the BVP \eqref{PVB} make sense also for $\tau=\infty$ in which the only given spatial coordinate is $x^*(0):=x_0$. The correspondence $x_0\ra y^*(0)$ is smooth and its graph is an invariant manifold of the flow, tangent to $y=0$ at the origin. This is in fact the stable manifold of the origin and using the obvious change in coordinates that takes the graph diffeomorphically to the domain of definition one notices that the original vector field
\[X:=(X_1,X_2)=(L^-x+f(x,y),L^+y+g(x,y))
\] 
gets conjugated to one for which the stable and unstable manifolds coincide with the $x$ and the $y$ axes at least locally.
\begin{theorem}\label{coord.straighten}  There are smooth coordinates centered at the origin such that \eqref{sist.original} can be written as
	\begin{equation}\label{sist.dmcoord}
	\left\{\begin{array}{lll}
	\dot{x}=L^-x+\tilde{f}(x,y)x\\
	\dot{y}=L^+y+\tilde{g}(x,y)y\\
	\end{array}
	\right.\end{equation}
	where $\tilde{f}:\mathbb{R}^{s+u}\longrightarrow \End(\mathbb{R}^s)$ and $\tilde{g}:\mathbb{R}^{s+u}\longrightarrow \End(\mathbb{R}^u)$ are square matrices of functions  that vanish at the origin. Moreover, in a neighborhood of the origin in the new coordinates,  the stable and unstable manifolds are given by  $\mathrm{S}_{0}=\{y=0\}$ and  $\mathrm{U}_{0}=\{x=0\}$.
\end{theorem}

 Gronwall Lemma is used to prove some useful estimates for solutions of BVP in straighten coordinates.

 \begin{theorem}
	\label{teo.estimativa} Let $\varepsilon, \alpha >0$ be such that  $\delta:=\delta^{1}_{2\varepsilon}< \alpha<\mbox{max}\{\lambda_1,\mu_1\}$. Then the solution of the  BVP defined by system \eqref{sist.dmcoord} with spatial data $|x_0,y_1|\leq \varepsilon$ satisfies for any $\tau\in[0,\infty)$ and $t\leq \tau$  the following inequality
	\begin{equation}\label{desigual.sol}
	\left\{\begin{array}{lll}
	|x^{*}(t,x_0,y_1,\tau)|\leq |x_0|\mathrm{e}^{-(\alpha-\delta)t}\\
	|y^{*}(t,x_0,y_1,\tau)|\leq|y_1|\mathrm{e}^{(\alpha-\delta)(t-\tau)}
	\end{array}
	\right.\end{equation}
	In particular,
	\begin{equation}\label{desigual.ext}
	\left\{\begin{array}{lll}
	|x^{*}_1(x_0,y_1,\tau)|\leq |x_0|\mathrm{e}^{-(\alpha-\delta)\tau}\\
	|y^{*}_0(x_0,y_1,\tau)|\leq|y_1|\mathrm{e}^{-(\alpha-\delta)\tau}.
	\end{array}
	\right.\end{equation}
	Estimates are available also for any partial derivative $\frac{\partial^kx_1^*}{\partial (x_0,y_1,\tau)}$ or $\frac{\partial^ky_0^*}{\partial (x_0,y_1,\tau)}$ of order $k$ in the form:
	\begin{equation}\label{desigual.ext1}\left|\frac{\partial^kx_1^*}{\partial (x_0,y_1,\tau)}\bigr|_{(x_0,y_1,\tau)}\right|\leq C_k e^{-(\alpha-k\delta)\tau},\qquad \left|\frac{\partial^ky_0^*}{\partial (x_0,y_1,\tau)}\bigr|_{(x_0,y_1,\tau)}\right|\leq C_k e^{-(\alpha-k\delta)\tau}
	\end{equation}
	for some constant $C_k$ which does not depend on $(x_0,y_1,\tau)$.
\end{theorem}

 Let $\Omega$ be a neighborhood around the origin  and suppose  the vector field  $X$ has the form \eqref{sist.dmcoord}. 

We will choose  a "cube" ${C}_{\varepsilon}=\{(x,y)\in \mathbb{R}^s\times \mathbb{R}^u;\;  |x,y|\leq\varepsilon \}\subset \Omega$ of radius $\varepsilon>0$  which satisfies the  hypothesis of  Theorem \ref{teoremasolucao}.  The cube has the following boundary pieces:
	
$$	\partial^{+}{C}_{\varepsilon}=\{(x,y)\in \mathbb{R}^s\times \mathbb{R}^u;\; |x|=\varepsilon,|y|\leq \varepsilon\} $$ and $$ \partial^{-}{C}_{\varepsilon}=\{(x,y)\in \mathbb{R}^s\times \mathbb{R}^u;\; |x|\leq \varepsilon, |y|=\varepsilon\}.$$

Denote
$$V_0^{\varepsilon}=\{(x,y)\in {C}_{\varepsilon}~|~ |x|\cdot |y|=0\}=\mathrm{S}_{0}\cup\mathrm{U}_{0}.$$

In order to state the next result the partial order notation for the flow determined by $X$ is useful, i.e. 
\[ p_1\prec p_2
\]
will say that there exists a \emph{forward} flow line from $p_1$ to $p_2$.

\begin{theorem}\label{teo.vizinhanca} Suppose $L^-$ and $L^+$ are symmetric.

 Then for $\vre$ small enough $X$ is transverse to $\partial^{+}{C}_{\varepsilon}$ and to $\partial^{-}{C}_{\varepsilon}$. In addition, the following properties  hold for $C_{\varepsilon}$ with respect to $X$:
	\begin{enumerate}
		\item \textbf{Flow-convexity}:  for every pair  $q_1\prec q_2$ with $q_{1}, q_2\in {C}_{\varepsilon}$    and every   $q_1\prec p\prec q_2$ one has $p\in C_{\varepsilon}$;
		\item \textbf {Dulac map}: there exists a "first encounter" diffeomorphism $$\mu=(\mu_1, \mu_2):\partial^+{C}_{\varepsilon}\setminus \mathrm{S}_{0}\longrightarrow \partial^-{C}_{\varepsilon}\setminus \mathrm{U}_{0}$$ induced by the flow that satisfies
		$$\mu(x,y)=(x,y), \; \;\qquad \forall (x,y)\in \partial^+{C}_{\varepsilon}\cap \partial^- C_{\varepsilon}$$
		\item \textbf{Continuity of $\mu_1$ close to $S_{0}$}: for every $0<\gamma\leq \varepsilon$, there exists $0<\gamma_0\leq \varepsilon$ such that
		$$\forall (x,y)\in \partial^+ C_{\varepsilon}\;  \mbox{with}\;  |y|\leq \gamma_0\;  \mbox {one has}  \; |\mu_1(x,y)|\leq\gamma;$$
		%\item considere o conjunto $V_0=\{(x,y,b)\in \tilde{C}_{\varepsilon}; |x|.|y|=0\}$. Existe $M_0>0$, tal que para todo $M\geq M_0$, o conjunto $V_M$ definido por todos  os   pontos  sobre  as  trajet\'orC_{\varepsilon}ias  que  s\~ao  solu\c c\~oes de  PVB's  com  dados $(x_0,y_1,\tau)$,  em que  $|x_0|=|y_1|=\varepsilon$  e  $\tau\geq M$, unido com $V_0$ \'e uma vizinhan\c ca de $V_0$.
		\item \textbf{Fundamental neighborhoods}:  let $0<\gamma\leq \varepsilon$ and 
		\[ V_{\gamma}^{\varepsilon}:=\{p\in C_{\varepsilon}~|~\exists\; q=(x_1,y_1) \in \partial^-C_{\varepsilon},\; |x_1|<\gamma,\; p\prec q \}\cup
		V_0^{\varepsilon}\] 
		Then $V_{\gamma}^{\varepsilon}$ is a flow-convex neighborhood of $V_0^{\varepsilon}$ in $C_{\varepsilon}$ such that
		\[ V^{\varepsilon}_{\gamma}\ra V^{\varepsilon}_0,
		\]
		i.e. for every neighbohood $U$ of $V^{\varepsilon}_0$ there exists $\gamma_0>0$ such that $V^{\varepsilon}_{\gamma_0}\subset U$.

			\end{enumerate}
\end{theorem}

The following figures illustrate the properties of the Theorem \ref{teo.vizinhanca}.
\begin{figure}[h]
	\centering
	%	\captionsetup{width=12.3cm,skip=0pt}
	\caption{{Dulac map and the neighborhoods $V_{\gamma}^{\varepsilon}$}} \label{fig02}
	\includegraphics[scale=0.5]{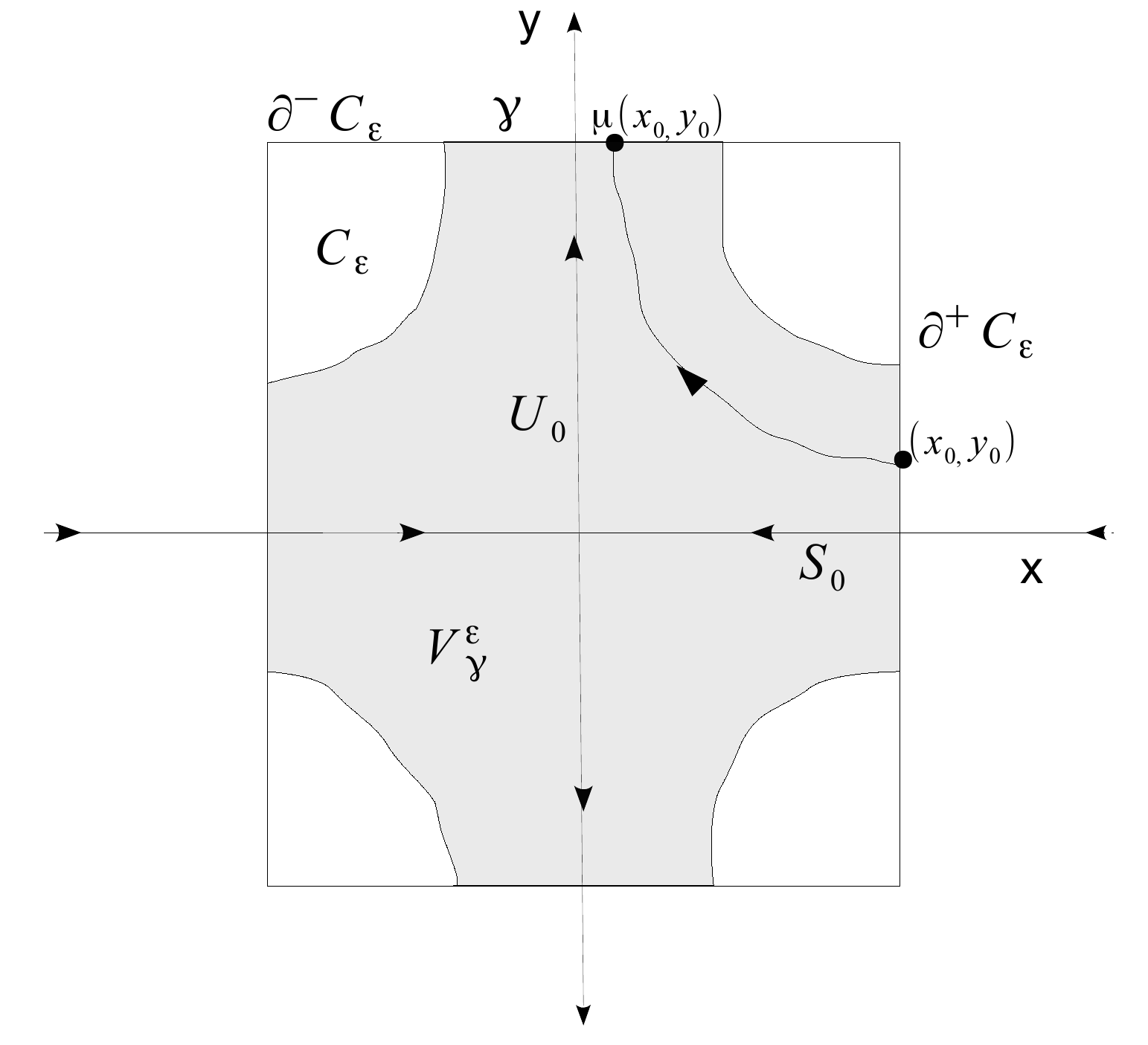}
	%\captionsetup{font=small,position=below,skip=-1pt}
	%	\caption*{Fonte: Imagem do autor.}
\end{figure}

\begin{remark}\label{Xtrans} Without the symmetry property of $L=L^-\oplus L^+$ the claim about transversality  of $X$ with $\partial^+C_{\varepsilon}$, $\partial^-C_{\varepsilon}$ fails even in the linear case. This seems to have been overlooked in \cite{M}. Take $s=0$, $u=2$, $L^+:=\left(\begin{array}{cc} 1& 2\\ 0&1\end{array}\right)$, $\tilde{g}=0$. Then
\[ \langle L^+(y_1,y_2),(y_1,y_2)\rangle=(y_1+y_2)^2.
\] 
Hence $X=L^+$ on $\bR^2$ is hyperbolic but  has points of tangency along the anti-diagonal with any coordinate sphere. 
\end{remark}

\begin{proof} The tangent spaces to the cylinders $\partial^{+}C_{\vre}$ and $\partial^{-}C_{\vre}$ at  points $(x,y)$ are described via:
\[ T_{(x,y)}\partial^{+}C_{\vre}= \{(v_1,v_2)\in \bR^n~|~\;\langle v_1,x\rangle=0\}
\]
\[T_{(x,y)}\partial^{-}C_{\vre}= \{(v_1,v_2)\in \bR^n~|~\;\langle v_2,y\rangle=0\}
\]
Hence we need to look at 
\[\langle X_1(x,y),x\rangle=\langle L^-x,x\rangle+\langle \tilde{f}(x,y)x,x\rangle\] and 
\[\langle X_2(x,y),y\rangle=\langle L^+y,y\rangle+\langle \tilde{g}(x,y)y,y\rangle.\]
We have that $\langle \tilde{f}(x,y)x,x\rangle= o(|x|^2)$ uniformly in $y$ and $\langle \tilde{g}(x,y)y,y\rangle=o(|y|^2)$ uniformly in $y$ since $\tilde{f}$ and $\tilde{g}$ are continuous and vanish at the origin.

The symmetry and definiteness of $L^-$ and $L^+$  imply now that there exists $\vre_1$ and $\vre_2$ such that
\[ \langle L^-x,x\rangle+\langle \tilde{f}(x,y)x,x\rangle <0 \,\qquad \forall\;0<|x|\leq \vre_1,\forall\; |y|\leq \vre_1
\]
and
\[\langle L^+y,y\rangle+\langle \tilde{g}(x,y)y,y\rangle>0\,\qquad \forall\; 0<|y|\leq \vre_2, \forall\; |x|\leq \vre_2
\]
For $\vre\leq \min\{\vre_1,\vre_2\}$ we get the transversality of $X$ with both $\partial^{\pm}C_{\varepsilon}$.

\textbf{Proof of (1)}. Let $q_1=(x_1,y_1)\prec q_2=(x_2,y_2),$ $q_1,q_2\in {C}_{\varepsilon}$. Let $\tau>0$ be the time needed to "travel" from $q_1$ to  $q_2$, i.e. $(x_2,y_2)=(x(\tau,x_1,y_1),y(\tau,x_1,y_1))$.  Now, consider the BVP defined by $X$ with data $(x_1, y_2,\tau)$. Since $|x_1,y_2|\leq\varepsilon$, it follows from the  Theorem \ref{teo.estimativa} that for all $0\leq t\leq\tau$, the following holds:
	\begin{eqnarray*}
		|x^{*}(t,x_1,y_2,\tau)| &\leq& |x_1|\mathrm{e}^{-(\alpha-\delta)t}<|x_1|<\varepsilon\\
		|y^{*}(t,x_1,y_2,\tau)|&\leq& |y_2|\mathrm{e}^{(\alpha-\delta)(t-\tau)}<|y_2|<\varepsilon.
	\end{eqnarray*}
	Therefore, the portion of the trajectory comprised between $q_1$ and $ q_2$ is contained in ${C}_{\varepsilon}$.
	
\textbf{Proof of (2)}. Consider $(x_0,y_0)\in \partial^{+}C_{\varepsilon}$ with $|y_0|<\varepsilon$. The normal vectors $v$ at $(x_0,y_0)$ to $\partial^{+}C_{\varepsilon}$ that point to the interior of $C_{\varepsilon}$ are described by the inequality $\langle v_1,x_0\rangle<0$ and we already chose $\vre$ so that the vector field $X$ satisfies such an inequality. Hence either  the trajectory $(x(t,x_0,y_0),y(t,x_0,y_0))$ will belong to the interior of $C_{\vre}$ for small $t>0$ or $(x_0,y_0)\in \partial^+C_{\vre}\cap \partial^-C_{\vre}$ i.e. $|y_0|=\varepsilon$, which is not the case.

Suppose $t>0$. For $y_0\neq0$ we have $\displaystyle \lim_{t\rightarrow+\infty}(x(t),y(t))\neq(0,0)$ and the trajectory cuts again the boundary of $ C_{\varepsilon}$. This happens because in the cube $C_{\vre}$ the function $t\ra |x(t)|$ is decreasing while $t\ra |y(t)|$ is increasing. Indeed the derivatives of the square of these functions are equal to 
\[ \langle X_1(x,y),x\rangle \;\;\mbox{and}\;\; \langle X_2(x,y),y\rangle \;\;\mbox{respectively}
\] 
and by our choice of $\vre$ above the first one is negative while the second one is positive as long as they are not $0$ which would happen for a stable or unstable flow line.

  Let $(x_1, y_1)\in \partial^+C_{\vre}\cup\partial^-C_{\vre}$ be the first point where this trajectory hits the boundary again and $ \tau>0$ the time needed to get from $(x_0,y_0)$ to $(x_1,y_1)$.

 Since $\langle X_1(x,y),x\rangle$ is a continuous function we see that   $(x_1,y_1)\notin \partial^{+}C_{\varepsilon}\setminus\mathrm{S}_{0}$ and moreover $|x_1|<\varepsilon$, because 
\begin{equation}\label{esteq}|x_1|=|x^{*}(\tau,x_0,y_1,\tau)| \leq |x_0|\mathrm{e}^{-(\alpha-\delta)\tau}<|x_0|\leq\varepsilon.
\end{equation}

Define $\mu$, for the time being, as the map that associates to every  $(x_0,y_0)\in \partial^{+}C_{\varepsilon}\setminus(\mathrm{S}_{0}\cup \partial ^-C_{\vre})$  the point of "first encounter" $(x_1,y_1)\in \partial^{-}C_{\varepsilon}\setminus(\mathrm{U}_{0}\cup \partial^+C_{\vre})$ which lies on the same trajectory.

Suppose now that  $|x_0|=|y_0|=\varepsilon$.  If there exist $\tau_0>0$ such that \linebreak $(x(\tau_0,x_0,y_0),y(\tau_0,x_0,y_0))=(x_0',y_0')\in C_{\varepsilon}$, define the BVP with data $(x_0,y_0', \tau_0)$. Then we get a contradiction from the estimates
\begin{eqnarray*}
	\varepsilon= |y_0|=|y^*(0,x_0,y_0',\tau_0)|\leq |y_0'|\mathrm{e}^{-(\alpha-\delta)\tau_0}<|y_0'|<\varepsilon.
\end{eqnarray*}
Hence for $|x_0|=|y_0|=\varepsilon$ the trajectory determined by $(x_0,y_0)$ only intersects $C_{\vre}$ in $(x_0,y_0)$. Therefore the natural extension of $\mu$ to $\partial^+C_{\vre}\cap \partial^-C_{\vre}$ is equal to the identity on this set.

The map $\mu$ is clearly bijective.  The differentiability of $\mu$ is standard and proved along the following lines. The essential part is to prove the differentiability of the time-function that associates to each $(x_0,y_0)$ the time $t(x_0,y_0)$ it takes to get to $\mu(x_0,y_0)$. Fix one such $(x_0,y_0)$ and use the diffeomorphism of the flow that corresponds to time $t(x_0,y_0)$ to flow a small open neighborhood of $(x_0,y_0)$ inside $\partial^+C_{\varepsilon}$ which does not contain points in the stable manifold to an $n-1$ dimensional manifold $H$ which contains $\mu(x_0,y_0)$ and is still transverse to $X$. Since $\mu(x_0,y_0)$ is not a critical point put coordinates in order to turn $X$ into the generator of the first coordinate-translation by unit-time. One arrives thus at the  problem of having to prove differentiability of the time-function obtained by going from one hypersurface  to another hypersurface  through the origin and both transverse to the first coordinate. That is simply the difference in height functions where the height is the first coordinate, hence a differentiable function.

\textbf{Proof of (3)}. Suppose that there exists $0<\gamma\leq\varepsilon$ and a sequence  $(x_{0n},y_{0n})$ such that $|x_{0n}|=\varepsilon$, $|y_{0n}|\longrightarrow 0$ ($y_{0n}\neq 0$) when $n\rightarrow\infty$ and $|x_{1n}|\geq\gamma$ for all $n$ where $\mu(x_{0n},y_{0n})=(x_{1n},y_{1n})\in \partial^-{C}_{\varepsilon}$. Let $\tau_n>0$ be the sequence of moments such that
$$ (x(\tau_n,x_{0n},y_{0n}),y(\tau_n,x_{0n},y_{0n}))=(x_{1n},y_{1n}).$$

It follows from the estimates (see \eqref{esteq})

\begin{eqnarray}\label{est.x1}
\displaystyle\gamma\leq|x_{1n}|\leq |x_{0n}|\mathrm{e}^{-(\alpha-\delta)\tau_n}\leq\varepsilon\mathrm{e}^{-(\alpha-\delta)\tau_n}
\end{eqnarray}
that $\displaystyle\tau_n\leq \ln \left(\frac{\varepsilon}{\gamma}\right)\frac{1}{\alpha-\delta}$.

 We can therefore select a convergent subsequence  $(x_{0n_k}, \tau_{n_k})\longrightarrow (x_0',\tau')$. Note that $\tau'>0$, because if $\tau'=0$ then $|y_{0n_k}|\longrightarrow \varepsilon$. Indeed, the only points for which it takes $0$-time to go from $\partial^+C_{\vre}$ to $\partial^-C_{\vre}$ are the points on $\partial^+C_{\vre}\cap \partial^-C_{\vre}$ and it is not hard to see that, due to the continuity of the time-function,  the shorter the time it takes to get to $\partial^-C_{\vre}$ the closer to $\partial^+C_{\vre}\cap \partial^-C_{\varepsilon}$ the starting point has to be.

Since $|y(\tau_{n_k},x_{0n_k},y_{0n_k})|=|y_{1n_k}|=\vre$  we get that  a contradiction with
$$y(\tau_{n_k},x_{0n_k},y_{0n_k})\ra y(\tau',x',0)=0.$$
 The later holds because the trajectory determined by $(x',0)$ is stable.

\textbf{Proof of (4)}. The flow-convex property is immediate from the analogous property of $C_{\epsilon}$.

 Suppose the $V^{\epsilon}_{\gamma}$ is not a neighborhood of $V_0$.  Hence, there exists a sequence $(x_n,y_n)\in C^{\epsilon}\setminus V_{\gamma}^\varepsilon$  with $(x_n,y_n)\longrightarrow (v_1,v_2)\in V_0^{\varepsilon}$. Thus either $x_n\longrightarrow 0$ or $y_n\longrightarrow 0$. 

As none of the points $(x_n,y_n)$ is on $\mathrm{S}_0\cup \mathrm{U}_0$ there exist  points of "first encounter" $(x_{0n},y_{0n})\in\partial^+C_{\vre}$ and $(x_{1n},y_{1n})\in \partial^-C_{\vre}$ such that
\[(x_{0n},y_{0n})\prec (x_n,y_n)\prec(x_{1n},y_{1n}).
\]
Since $(x_n,y_n)\notin V^{\epsilon}_{\gamma}$ it follows that $|x_{1n}|\geq \gamma$.

The case $x_n\ra 0$ is disposed immediately by considering the BVP  $(x_n,y_{1n},\tau_n)$ where $\tau_n$ is the time it takes to go from $(x_n,y_n)$ to $(x_{1n},y_{1n})$. Then we get a contradiction in
$$\gamma\leq|x_{1n}|\leq|x_{n}|e^{-(\alpha-\delta)\tau_n}\leq |x_{n}|\ra 0.$$

For the case $y_n\ra 0$ consider the BVP with data $(x_{0n},y_{1n},\tau_n')$ where $\tau_n'$ now is the time it takes to get from $(x_{0n},y_{0n})$ to $(x_{1n},y_{1n})$. Recall that  $|x_{0n}|=|y_{1n}|=\varepsilon$. We obtain
$$\gamma\leq|x_{1n}|\leq|x_{0n}|e^{-(\alpha-\delta)\tau_n'} \; \Rightarrow \; \tau_n\leq \ln{\left(\frac{\varepsilon}{\delta}\right)}\frac{1}{\alpha-\delta}.$$

$$|y_{0n}|\leq |y_n|\mathrm{e}^{-(\alpha-\delta)s_n}\leq |y_n|,$$
where $0<s_n\leq \tau_n'$. Since $|y_n|\longrightarrow 0$, we have $y_{0n}\longrightarrow 0$.

Now, since  $|x_{0n}|=\varepsilon$ and $\tau_n'$ are bounded and $y_{0n}\ra 0$,  we are in the same scenario that led to a contradiction in the proof of item (3). 

In order to prove that $\lim_{\gamma \ra 0} V^{\vre}_{\gamma}=V_0^{\vre}$ as defined in the statement it is enough to prove that
\begin{equation} \label{limvre}\bigcap V^{\vre}_{\gamma_n}=V_0^{\vre}
\end{equation}
for any decreasing sequence $\gamma_{n}\ra 0$. Indeed if (\ref{limvre}) holds then fix one such sequence $\gamma_n$. Suppose $\exists u_n\in  V^{\vre}_{\gamma_n}\setminus U$ for all $n$. Since $\overline{V^{\vre}_{\gamma_{n+1}}}\subset V^{\vre}_{\gamma_n}$  and $\overline{V^{\vre}_{\gamma_{n+1}}}$ (the closure in $C_{\vre}$) is compact  we get that we can extract a convergent subsequence $u_{n_k}$ whose limit necessarily belongs to $\bigcap V^{\vre}_{\gamma_n}=V_0^{\vre}$. But this is a contradiction with $U$ being a neighborhood of  $V_0^{\vre}$. 

In order to prove (\ref{limvre}) take a point $a=(x,y)\in \bigcap V^{\vre}_{\gamma_n}$ such that $x\neq 0\neq y$. Since $y\neq 0$ it follows that $a\notin S_0$ and therefore $a$ lies on a trajectory which hits $\partial ^-C_{\vre}$ in a point $b=(x',y')\notin U_0$, i.e. $x'\neq 0$. But  $a\in \bigcap V^{\vre}_{\gamma_n}$ implies that $|x'|<\gamma_n$ for all $n$, hence $x'=0$. Contradiction.
\end{proof}

The properties of these neighborhoods will enable us in Section \ref{Sec3} to control the deformation of subsets by the flow of a  vector field. 

The fundamental local tool we will use in the next sections is Theorem 1.3.21 in \cite{M} which we now state.

\begin{theorem}\label{teo.subvde} Let   $C_{\varepsilon}\subset \Omega$ be a cube  as in Theorem \ref{teo.vizinhanca} and let  $\mathring{C}_{\vre}$ be its interior. If $\psi_t$ denotes the flow of the vector field $X=(L^-+f,L^++g)$, then the closure of the submanifold	
	\begin{eqnarray}\label{subvde}
\hspace{1cm}	W=\{(t,\psi_{\frac{t}{1-t}}(x,y),x,y);\; (x,y)\in \bR^{s}\times \bR^u, \; 0<t<1\}\cap(\bR\times\mathring{C}_{\vre}\times\mathring{C}_{\vre})
	\end{eqnarray}
 inside $(\mathring{C}_{\vre}\times\mathring{C}_{\vre}\times \bR)$ is a smooth submanifold with boundary
	\begin{eqnarray}\label{b.subvde}
	\partial \overline{W}= \{1\}\times (\mathrm{U}_{0}\cap \mathring{C}_{\vre})\times (\mathrm{S}_{0}\cap\mathring{C}_{\vre}) \bigcup \{0\}\times \triangle_{\mathring{C}_{\varepsilon}},
	\end{eqnarray}
	where $\triangle_{\mathring{C}_{\varepsilon}}=\{(p,p)~|~ p\in \mathring{C}_{\varepsilon}\}$ denotes the diagonal in $\mathring{C}_{\varepsilon}\times \mathring{C}_{\varepsilon}$.
\end{theorem}
\begin{proof}(Sketch) The idea is to turn to BVP coordinates via \eqref{Identidades} in
\[ W=\left\{\left(\tau, x_1^*\left(x_0,y_1,\frac{\tau}{1-\tau}\right),y_1,x_0,y_0^*\left(x_0,y_1,\frac{\tau}{1-\tau}\right)\right)~\biggr|~|x_0,y_1|<\vre, \tau\in(0,1)\right\}.
\]
Use now \eqref{desigual.ext1} to conclude that $x_1^*$ and $y_0^*$ converge uniformly to $0$ when $\tau\ra 1$ together with their derivatives, obtaining thus that $W$ is the graph of a smooth function over $\mathring{C}_{\varepsilon}\times [0,1]$.
\end{proof}

\begin{remark}\label{rem12} It is easy to see that the projection of $\overline{W}$ onto  the $\mathring{C_{\vre}}\times \mathring{C_{\vre}}$ components coincides with the intersection $\overline{\bigcup_{t\geq 0}\tilde{\psi}_t(\Delta_{\mathring{C}_{\varepsilon}})}\cap(\mathring{C_{\vre}}\times \mathring{C_{\vre}})$ where $\tilde{\psi}$ is the flow on $\bR^{s+u}\times\bR^{s+u}$ that is equal to $\psi$ in the first component and leaves the points fixed in the second component. Take $(a,b)\in \overline{\bigcup_{t\geq 0}\tilde{\psi}_t(\Delta_{\mathring{C}_{\varepsilon}})}\cap(\mathring{C_{\vre}}\times \mathring{C_{\vre}})$. Then there exist $t_n\geq 0$ and $u_n\in \bR^{s+u}$ such that $(\psi_{t_n}(u_n),u_n)\ra (a,b)$. Since $a,b\in \mathring{C_{\vre}}$ one has  that $u_n, \psi_{t_n}(u_n)\in \mathring{C_{\vre}}$ for $n$ big enough. In fact, by passing to a subsequence either $t_n$ converges or $t_n\ra \infty$. Let $s_n\in[0,1)$ be such that $t_n=\frac{s_n}{1-s_n}$ we get that $(a,b)=\lim z_n$ where $(z_n,s_n)\in W$ converges in $[0,1]\times\mathring{C_{\vre}}\times \mathring{C_{\vre}}$. The other inclusion is also obvious.

\end{remark}

\begin{remark} \label{rem001}
We will need a slight extension of this results in the simplest situation when there exists a central manifold. If the vector field $Y$ on $\bR^s\times \bR^u\times \bR^c$ is of type $(X,0)$ with $X$ as before, not depending on $z\in \bR^c$ then rather than taking the graph of the flow $\psi^Y$ of $Y$ in $\bR^s\times \bR^u\times \bR^c\times \bR^s\times \bR^u\times \bR^c$, it makes sense to forget about one stationary variable $\bR^c$ and consider the corresponding $W^Y$ of Theorem \ref{teo.subvde} to  be
 \[W^Y:=\{(t,\psi^X_{\frac{t}{1-t}}(x,y),x,y,z)~|~t\in (0,1)\}\cap [0,1]\times\mathring{C}_{\vre}\times \mathring{C}_{\vre}\times \bR^c.\]
 The closure will be a manifold with boundary 
\[ \partial \overline{W^Y}=\{1\}\times (\mathrm{U}_{0}\cap \mathring{C}_{\vre})\times (\mathrm{S}_{0}\cap\mathring{C}_{\vre})\times \bR^c \bigcup \{0\}\times\triangle_{\mathring{C}_{\varepsilon}}\times \bR^c \]
We do this rather than considering the graph of $\psi^Y$ in the full ambient space with an eye to keep the bookkeeping simpler and avoid using fiber products later on.
 \end{remark}

We have a useful Corollary of Theorem \ref{teo.vizinhanca}.
\begin{corollary}\label{cortv} Let $(q,0), (0,r)\in C_{\ve}$ be two points in the stable, respectively the unstable manifold of the origin. Let $B^u(q,\epsilon'):=\{(q,y)~|~|y|\leq \epsilon'\}$ be a small transverse  slice  to $S_0$ that passes through $(q,0)$. Then there exists two sequences of points $(q,y_n'), (x_n,y_n)\in C_{\ve}$ such that
\begin{itemize} 
\item[(i)] $(q,y_n')\in B^u(q,\epsilon')$ and $y_n\ra 0$;
\item[(ii)] $(x_n,y_n)\ra (0,r)$;
\item[(iii)] $(q,y_n')\prec (x_n,y_n)$.
\end{itemize}
\end{corollary}

Finally, we will need  the following consequence of the Flowout Theorem \cite{Le} which is proved with the same ideas as the differentiability of the Dulac map $\mu$ from Theorem \ref{teo.vizinhanca}. Let $M$ be a compact smooth Riemannian manifold and let $\psi$ denote the gradient flow of a function $f$.
\begin{prop}\label{FlB} Let $N\subset M$ be a submanifold which does not contain any critical points of $f$ and is transverse to the gradient vector field $\nabla f$. Let $f^{-1}(\theta)$ be a regular level set of $f$ such that $N\cap f^{-1}(\theta)=\emptyset$ and suppose that for every $n\in N$ there exists a time $t_n\geq 0$ such that $\psi_{t_n}(n)\in f^{-1}(\theta)$. Then the function $n\ra t_n$ is smooth and moreover the set $\{\psi_{t_n}(n)~|~n\in N\}$ is a smooth submanifold of $f^{-1}(\theta)$ diffeomorphic to $N$.
\end{prop}

\section{A general homotopy formula}\label{Sec3}

Consider $\pi:P\longrightarrow B$ a locally trivial fiber bundle with  compact fiber $M$ over  a $n$-dimensional, oriented manifold $B$. Let $X:P\longrightarrow VP$ be a vertical vector field on $P$ where the vertical tangent space $VP:=\Ker d\pi$ represents the collection of all the tangent spaces to the fibers. Suppose $X$ a horizontally constant Morse-Smale vector field. Recall that this means that for every $b\in B$ there exists an open set $b\in B_0\subset B$ and  a local trivialization $\alpha:P|_{B_0}\longrightarrow M\times B_0$ such that 
\begin{eqnarray}\label{s3eq1}\alpha_*(X_p)=(\mathrm{grad}f_{\alpha_{1}(p)},0), \; \forall p\in P|_{B_0},
\end{eqnarray}
where $f:M\longrightarrow \bR$ is a Morse-Smale function for same Riemannian metric on $M$. This in particular implies that the critical set of $X$ is a fiber bundle over $B$ (with various components) and the same thing stays true about the stable and unstable manifolds.

In fact the critical manifolds of $X$ in the local trivialization $\alpha$ are $F=\{p\}\times B_0$ with $p$ satisfying $\nabla_pf=0$ and the stable and unstable manifolds are 
\[S(F)=S(p)\times B_0,\qquad U(F)=U(p)\times B_0
\] 

Let $s: B\longrightarrow P$ be a transversal section to all the stable bundles $\mathrm{S}(F)$ relative to the critical manifold $F\subset P$. 

We follow the ideas in \cite{HL1} and \cite{Ci1}. Consider the fiber bundle $P\times_{B}P\longrightarrow B$ and the vertical vector field $\tilde{X}=(X,0)$ on $P\times_{B}P$. It is not difficult to verify that the vector field $ \tilde{X} $ is horizontally constant Morse-Bott-Smale and that its flow is $\Theta_t(v_1,v_2)=(\Phi_t(v_1),v_2)$, where $\Phi_t:P\longrightarrow P$ denotes the flow of the vector field $X$. Thus, the stable and unstable manifolds  relative to a critical manifold $ \tilde{F} $ of $\tilde{X} $ are
\begin{eqnarray}
 \mathrm{S} (\tilde{F}):=\mathrm{S}(F)\times_B P\; \; \mathrm{and}\; \;  \mathrm{U} (\tilde{F}):=\mathrm{U}(F)\times_B P.
\end{eqnarray}

Now define the family of sections $\xi:[0,\infty)\times B\ra P\times_BP:$
$$\xi_t(b):= \Theta_t(s(b),s(b))=(\Phi_t(s(b)),s(b))$$
 transverse to all stable manifolds of $\tilde{X}$.
 
 \begin{remark}\label{rem123} The transversality of $s$ and $S(F)$ translates into the transversality of $\xi_0$ and $\mathrm{S} (\tilde{F})$
 \end{remark}

 The question we will be concerned with in this section is whether the following family of currents in $P\times_BP$ has a limit:

$$\displaystyle \lim_{t \rightarrow + \infty} (\xi_t)_*(B)? $$

For each $t>0$,  Stokes Theorem and the commutativity of $d$ with push-forward implies:
\begin{eqnarray}\label{eq.homo.xi}
(\xi_t)_*(B)-(\xi_0)_*(B)=d[\xi_*([0,t]\times B)].
\end{eqnarray}

From \eqref{eq.homo.xi} and of the continuity of the (exterior) differential  operator of currents, we have reduced our analysis to the existence of the limit $\displaystyle \lim_{t\rightarrow+\infty} \xi_*([0, t] \times B)$. The following result presents a positive response to the existence of this limit without the tameness condition of \cite{Ci2}.

\begin{theorem}\label{tlimxi}
Let $\pi: P\longrightarrow B$ be a fiber bundle with compact fiber and let $X$ be a horizontally constant Morse-Smale	vertical vector field. If $s:B\longrightarrow P$ is a section transversal to the stable manifolds $\mathrm{S}(F)$, then  $$ T=\xi([0,+\infty)\times B)$$ defines a $ (n + 1) $-dimensional current of locally finite mass and if $B$ is compact then $T$ is of finite mass.
Moreover, the following equality of kernels holds: 
	\begin{eqnarray}\label{bordo.T}
	\mathrm{d} T=\sum_{F}\mathrm{U}(F)\times_{F}s(s^{-1}(\mathrm{S}(F)))-\xi_*(B).
	\end{eqnarray}	
\end{theorem}

We deduce from  Theorem \ref{tlimxi} that 
\begin{eqnarray}
\displaystyle \xi_{\infty}(B):=\lim_{t\rightarrow+\infty}(\xi_t)_*(B)= \sum_{F}\mathrm{U}(F)\times_{F}s(s^{-1}(\mathrm{S}(F))).
\end{eqnarray}

Another consequence of the Theorem \ref{tlimxi} is expressed in terms of operators:

\begin{corollary}\label{c.principal} Let $\pi:P\longrightarrow B$ be an fiber bundle with compact oriented  fiber  on a $n$-dimensional smooth,  oriented $B$. Let $X: P\longrightarrow VP$ be a horizontally constant Morse-Smale vertical vector field and let $s: B\longrightarrow P$ be a section  transversal to all stable manifolds $\mathrm{S}(F)$ of  $ X $. Assume that for each critical manifold $F\subset P$ the bundle $\mathrm{U}(F)\longrightarrow F$ is oriented. Let $s_t=\Phi_t\circ s:B\longrightarrow P$ be the induced family of sections. Then for each closed form $\omega$ on $P$ of degree $k\leq n$, the following identity of flat currents in $ B $ is true:
	\begin{eqnarray}
	\displaystyle  \lim_{t\rightarrow+\infty}s_t^{*}\omega=\sum_{\mathrm{codim}\mathrm{S}(F)\leq k}\mathrm{Res}^{u}_F(\omega)[s^{-1}(\mathrm{S}(F))],
	\end{eqnarray}
	where $\displaystyle \mathrm{Res}^{u}_F(\omega)=\tau^{*}_F\left(\int_{\mathrm{U}(F)/F}\omega\right)$ and $\tau_F: s^{-1}(\mathrm{S}(F))\rightarrow F$ is the composition of $\pi^{s}_F:\mathrm{S}(F)\rightarrow F$ with  $s:s^{-1}(\mathrm{S}(F))\longrightarrow S(F)$ . Moreover, there is a flat current $\mathcal{T}_{\infty}(\omega)$ such that
	\begin{eqnarray}\label{for.transgressao}
	\sum_{\mathrm{codim} \mathrm{S}(F)\leq k}\mathrm{Res}^{u}_F(\omega)[s^{-1}(\mathrm{S}(F))]-s^{*}\omega= d[\mathcal{T}_{\infty}(\omega)].
	\end{eqnarray}
\end{corollary}

\begin{remark}  One can extend Theorem \ref{tlimxi} and Corollary \ref{c.principal} without extra effort to the case when $B$ is a manifold with corners and $s:B\ra P$ is a smooth section completely transverse to all $S(F)$. This happens because any smooth map on a manifold with corners can, by definition, be extended locally in a neighborhood of the corner a smooth map defined on an open set in $\bR^n$. Since the same applies to the trivialization maps, one  can extend smoothly the whole fiber bundle structure and the section to be defined on an open set in $\bR^n$. The transversality condition being open, it will hold in  an open subset. Then one uses the corresponding results of Theorem \ref{tlimxi} and Corollary \ref{c.principal} in this open set, only to restrict it afterwards.
\end{remark}

In order to prove Theorem \ref{tlimxi} we notice that due to the sheaf property of currents it is enough to prove it for a convenient open covering of $B$. Namely if there exists an open covering $B=\cup B_i$ such that $T_i:=T\bigr|_{P_i}$ exists, where $P_i:=P\bigr|_{B_i}\times_{B_i}P\bigr|_{B_i}$, it is of locally finite mass and \eqref{bordo.T} holds then the same thing is true over $B$. This happens first because on the overlap $P_i\cap P_j$ the restriction of the  limits $\xi_*([0,t]\times B_{i/j})$ are the same. This allows the patching of $(T_i)_{i\in I}$ to a single current. A similar argument works for the right hand side of \eqref{bordo.T}.
\begin{remark} One has to worry always if an embedded oriented submanifold really determines a current especially when it is not \emph{properly} embedded, i.e. the inclusion is not a proper map. This is the case of each $\mathrm{U}(F)\times_{F}s(s^{-1}(\mathrm{S}(F)))$ and is not clear apriori that they exist globally. But if \eqref{bordo.T} holds  on an open covering $P_i$  then the patching of $T_i$ to a single current implies the patching of $\mathrm{U}(F)\times_{F}s(s^{-1}(\mathrm{S}(F)))\bigr|_{P_i}$ to a single current globally and \eqref{bordo.T} stays true everywhere.
\end{remark}

\begin{remark} \label{Fmax} Notice that if $F_{\max}$ represents a "maximal" critical manifold, in other words the critical manifold for which $S(F_{\max})$ is an open subset of $P$ the transversality condition  of $s\pitchfork S(F_{\max})$ is automatically satisfied for all points $b\in s^{-1}S(F_{\max})$. In the open set $s^{-1}(S(F_{\max}))\subset B$ relation \eqref{bordo.T} is trivial to prove and the limit 
\[ \lim_{t\ra \infty}\xi_t(s^{-1}(S(F_{\max})))=F_{\max}\times_{F_{\max}} s(s^{-1}(S(F_{\max}))
\]
holds even in the $C^{\infty}$ sense. The open sets $s^{-1}(S(F_{\max}))$ will always be part of the open coverings of $B$ and we will not mention them again.
\end{remark}

We will therefore work with a convenient covering of $B$, each of its members being contained into a trivializing neighborhood for $X$, by which we mean a neighborhood $U$ where \eqref{s3eq1} holds.

The strategy is now is \emph{roughly} the following.  We show that around each point $b_0\in B$ there exists a trivializing neighborhood $B_0$ and a finite open covering $M_j$ of the fiber $M:=P_{b_0}$ such that Theorem \ref{tlimxi} holds for the restrictions of all the currents to the open sets:
\[  M_j\times M\times B_0\simeq(M_j\times B_0)\times_{B_0}(M\times B_0)
\]
where we have already used a trivializing diffeomorphism $P\bigr|_{B_0}\simeq M\times B_0$.

In order to achieve this on each of the open sets $M_j\times M\times B_0$ we construct a "resolution of the flow", namely a \emph{proper} map
\[ \Psi: N\ra M_j\times M\times B_0
\]
from a manifolds with corners $N$ of dimension $n+1$ such that
\begin{itemize}
 \item[(a)]$\Imag \Psi =\overline{\xi([0,\infty)\times B_0)}\cap (M_j\times M\times B_0)$
 \item[(b)] $\Psi$ is a diffeomorphism from an open subset of full measure in $N$ to an open subset of full $\mathcal{H}^{n+1}$-measure in $\xi([0,\infty)\times B_0)\cap ( M_j\times M\times B_0)$.
 \end{itemize}
 
 Point (b) makes sense in view of the fact that $\xi([0,\infty)\times B_0\setminus Z)$ is a smooth submanifold of $P\times_B P$ of dimension $n+1$ where $Z:=\displaystyle\bigcup_F s^{-1}(F)$ are the fixed points of the section with respect to the flow. 
 \begin{remark} We assume of course that $B_0$ is a small neighborhood of a point $b_0\notin F_{\max}$ and therefore $Z$ will have zero $\mathcal{H}^{n}$ measure.
 \end{remark}

 We then use the following 
 \begin{lemma}\label{lema.ideia}Let $N^{n+1}$ be a manifold with corners  and let $\Psi: N\longrightarrow X$ be a smooth (Lipschitz is enough)  map to a smooth manifold $X$. If $\Psi$ is a proper map, then $\Psi(M)$ has locally finite $n+1$-dimensional Hausdorff measure  and
	$$d(\Psi_{*}(N))=\Psi_*(\partial N).$$
\end{lemma}

This Lemma   will allow  not only to conclude that the restriction of $T\bigr|_{B_0\times M_j\times M}$ is of locally finite mass, but also to compute its boundary as the push-forward of $\partial N $. In order to do that we will need a full understanding of $\partial N$ and of the map $\Psi$.

\vspace{0.5cm}

\section{The local flow resolution}\label{tec.tools}

This section contains the heart of the proof of Theorem \ref{tlimxi}. 

As we discussed in the previous section it is enough to localize around each point $b_0$ in the base space $B$. We will therefore consider first an open neighborhood $B_0$ of $b_0$  where the fiber bundle $P\bigr|_{B_0}$ is trivial and the vector field $X$ is horizontally constant  and given by the gradient of a function $f:P_{b_0}\ra \bR$. With the notation $M:=P_{b_0}$ already introduced we have the following data
\begin{itemize}
\item[(a)] a flow  on $M$ induced by the gradient of $f$ and denoted $\psi$.
\item[(b)] a family of (local) sections $s_t:B_0\ra  M\times B_0$ 
 \[s_t(b)=(\psi_t(s_0(b)),b)\] originating from $s\bigr|_{B_0}=(s_0,\id_{B_0}):B_0\ra P\bigr |_{B_0}=M\times B_0$. 
 \end{itemize}
 
 We are interested in the closure of the forward flow-out of the graph of $s$.  In other words let $\xi_t:B_0\ra M\times M\times B_0$ be the family of sections:
 \[ \xi_t(b)=(\psi_t(s_0(b)),s_0(b),b).
 \]
 The function $\tilde{f}:M\times M\times B_0\ra \bR$, $\tilde{f}(m_1,m_2,b):=f(m_1)$ will be used occasionally.  
 
We will aim to construct a "manifold with corners resolution" of a piece of  $\overline{\bigcup_{t\geq 0} \xi_t(B_0)}$ to be described momentarily.

 Let 
\[ {\bf{p}}_1:=\lim_{t\ra \infty} \psi_t(s(b_0))\in M\]
be a critical point.  We will assume next that ${\bf{p}_1}$ is not a point of  local maximum for $f$, i.e.  $\dim{U_{{\bf{p}_1}}}>0$.  The case $\dim{U}_{{\bf{p}_1}}=0$ can be treated quite easily as locally everything flows when $t\ra \infty$ to the critical manifold determined by $\{{\bf{p}_1}\}$ as already noticed in Remark \ref{Fmax}.

 For simplicity, we will also assume that $f({\bf{p}_1})=0$.  
 
 We allow the situation ${\bf{p}_1}=s_0(b_0)$.  At the other extreme, $s_0(b_0)$ might be "far" from ${\bf{p}_1}$. If that is the case, we notice that nothing interesting happens with the flow-out of $\xi_0(B_0)$ before we get close to $\{{\bf{p}_1}\}\times M\times B_0$. So we might assume without restriction of the generality that $s_0(b_0)$ is in a neighborhood $D$ of ${\bf{p}_1}$ where the Straighten Coordinates Theorem is valid.
Then we will also assume that $B_0$ was first chosen so that $s_0(B_0)\in \mathring{C_{\varepsilon}}$ for some fixed $\varepsilon$, where $C_{\varepsilon}\subset \bR^s\times \bR^{u}$ satisfies the hypothesis of Theorem \ref{teo.vizinhanca}. Hence $\xi_0(B_0)\subset \mathring{C_{\varepsilon}}\times M\times B_0$.

 We will need to work with certain particular neighborhoods of $\{|x|\cdot |y|=0\} \times M\times B_0$ of type $V_{\gamma}^{\epsilon}\times M\times B_0$ where $V_{\gamma}^{\epsilon}$ is as in Theorem  \ref{teo.vizinhanca}.

The next technical statement prepares the field for the next step of the induction, namely when we will go from the first critical level (that of ${\bf{p}_1}$) to the second critical level in the direction of the flow.
\begin{prop}\label{prop.triplo} 
\begin{itemize} \item[(a)] There exist   $\gamma\leq \varepsilon$ and $\theta>0$ a regular value of $f$  such that  the trajectory determined by any $p\in V^{\varepsilon}_{\gamma}\cap f^{-1}(-\infty,\theta]$ that intersects $f^{-1}\{\theta\}$ does so before intersecting $\partial C_{\varepsilon}^{-}$ when $t\ra \infty$.

\item[(b)] If ${\bf{p}_1}$ is not a point of minimum then we can take $\theta$ and $\gamma$ such that the trajectory determined by any $p\in V^{\varepsilon}_{\gamma}\cap f^{-1}[-\theta,\theta]$ that intersects $f^{-1}\{-\theta\}$ does so before intersecting $\partial C_{\varepsilon}^{+}$ when $t\ra -\infty$.
\end{itemize}

 Moreover, after shrinking $B_0$ the following holds: 
	
	\begin{itemize}
	\item[(i)] $s_0(B_0)\subset V_{\gamma}^{\varepsilon}\cap f^{-1}(-\infty,\theta')$ for some $\theta'$ with $0<\theta'<\theta$ and
	\item[(ii)] $\left(\bigcup_{t\geq0}\xi_t(B_0)\right)\cap \tilde{f}^{-1}(\theta)\subset V_{\gamma}^{\varepsilon}\times M\times B_0$
	\end{itemize}	
\end{prop}
\begin{proof} In the cube $C_{\varepsilon}$ of Theorem   \ref{teo.vizinhanca} there exists  $0<\gamma<\varepsilon$ such that
	$$\inf_{u\in T_{\gamma}}{f}(u)>0,$$
where $T_{\gamma}=\{(x,y)\in \partial^{-}{C}_{\varepsilon}; |x|<\gamma, |y|=\varepsilon\}$ since $f>0$ on the compact $\partial^{-}{C}_{\varepsilon}\cap U_{{\bf{p}_1}}$. Fix such a $\gamma$.

Choose now $\theta>0$ with
\begin{equation}\label{defthe}  0<\theta< \inf_{u\in T_{\gamma}}f(u)
\end{equation}
Notice that each trajectory that starts inside $V^{\varepsilon}_{\gamma}$ either ends up at the critical point or leaves  $V^{\varepsilon}_{\gamma}$ through $T_{\gamma}$. Let $p\in V^{\varepsilon}_{\gamma}\cap f^{-1}(-\infty,\theta]$. On one hand $f$ is increasing and continuous along the trajectories and on the other hand $V^{\epsilon}_{\gamma}\cap\gamma_p$  \footnote{the trajectory $\gamma_p$ is determined by $p$}is connected by the flow-convex property of $V^{\varepsilon}_{\gamma}$. It follows that $\gamma_p$ meets $f^{-1}(\theta)$ before reaching $T_{\gamma}$ by (\ref{defthe}).

Part (b) is analogous.

Since $s_0(b_0)\in S_{{\bf{p}_1}}$ we have that $f(s(b_0))\leq 0$ and therefore one can choose $B_0$ with $s_1(B_0)\subset V_{\gamma}^{\epsilon}\cap f^{-1}(-\infty,\theta')$. 

Part (ii) is an immediate consequence of (i) and the first part of the proof.
\end{proof}

\vspace{0.5cm} 

From now on the neighborhood $B_0$ of $b_0$ will satisfy the properties of Proposition \ref{prop.triplo} for a certain $\theta$ and $\gamma$.

We define now the first piece of the transverse intersection we will use later.

Let $\tilde {C_{\varepsilon}}:=C_{\varepsilon}\times C_{\ve}\times B_0$ and $\mathring{\tilde{C_{\varepsilon}}}:=\mathring{C_{\varepsilon}}\times \mathring{C_{\epsilon}}\times B_0$ and consider $W_1\subset  [0,1]\times\mathring{\tilde {C_{\varepsilon}}} $  be the analogue of $W$ from Theorem \ref{teo.subvde} for this context:
\begin{eqnarray}\label{W1def}
W_1 
&=&\left\{\left(t,\psi_\frac{t}{1-t}(x,y),x,y,b\right); \; 0<t<1\right\}\cap \mathbb{R}\times\mathring{{C_{\varepsilon}}}\times\mathring{{C_{\varepsilon}}}\times B_0.\nonumber
\end{eqnarray}
In other words, modulo a permutation of the last two  coordinates we have:
\[ W_1= W\times B_0\subset ([0,1]\times\mathring{C_{\varepsilon}}\times \mathring{C_{\varepsilon}})\times  B_0
\]
where  $W$ is as in Theorem \ref{teo.subvde}. 

It  follows then from Theorem \ref{teo.subvde}  that $\overline{W_1}$, the closure inside $[0,1]\times\mathring{ {C_{\varepsilon}}}\times \mathring{ {C_{\varepsilon}}}\times B_0$ is a smooth $(m + 1+n)$-dimensional with boundary:
\begin{eqnarray*}
\partial\overline{W_1}=\underbrace{\{1\}\times({\mathrm{U}}_{{\bf{p}_1}}\cap \mathring{{C_{\varepsilon}}}) \times ({\mathrm{S}}_{{\bf{p}_1}}\cap\mathring{{C_{\varepsilon}}})\times  B_0}_{\partial_1{\overline{W_1}}}\bigcup\underbrace{\{0\}\times \Delta_{\mathring{{C_{\varepsilon}}}}\times B_0}_{\partial_0\overline{W_1}}.
\end{eqnarray*}

For future reference,  let
\begin{equation}\label{Wp1}\overline{W}_{{\bf{p}_1}}:=[0,1]\times\{{\bf{p}_1}\}\times\{{\bf{p}_1}\}\times  B_0\subset \overline{W_1}.
\end{equation}
be the set of fixed points in $\overline{W}_1$.
%\begin{remark}\emph{ Is this necessary now?} Remark \ref{rem12} leads to the conclusion that the projection of the closure $\overline{W_1}$ onto the components $\mathring{{{C_{\varepsilon}}}}\times \mathring{{{C_{\varepsilon}}}}$ can also be seen as the intersection $\overline{\displaystyle\bigcup_{t\geq 0}\tilde{\Theta}_t\left(\Delta_{\mathring{\tilde{{C_{\varepsilon}}}}}\right)}\cap (\mathring{\tilde{{C_{\varepsilon}}}}\times \mathring{\tilde{{C_{\varepsilon}}}})$, where the flow $\tilde{\Theta}$ is on $(M\times M\times B_0)_{\mathrm{I}}\times (M\times M\times B_0)_{\mathrm{II}}$ and coincides with $\Theta$ in the first, $(\cdot)_{\mathrm{I}}$ component , while leaving the second, $(\cdot)_{\mathrm{II}}$ component fixed. In other words, this projection is the part of the closure of the flow-out of the diagonal of $\mathring{\tilde{{C_{\varepsilon}}}}$ via $\tilde{\Theta}$ that lies in $(\mathring{\tilde{{C_{\varepsilon}}}}\times \mathring{\tilde{{C_{\varepsilon}}}})$. 
%\end{remark}

We look now at the second piece of transverse intersection mentioned before. Let 
\[V_{\theta}:= \mathring{V}^{\epsilon}_{\gamma}\cap f^{-1}((-\infty,\theta])\subset \mathring{C_{\varepsilon}}
\] where $\epsilon$, $\gamma$ and $\theta$ are as in Proposition \ref{prop.triplo}. 
\begin{figure}[h]
	\centering
	%	\captionsetup{width=12.3cm,skip=0pt}
	\caption{{The neighborhoods $V_{\theta}$}} \label{fig02}
	\includegraphics[scale=0.5]{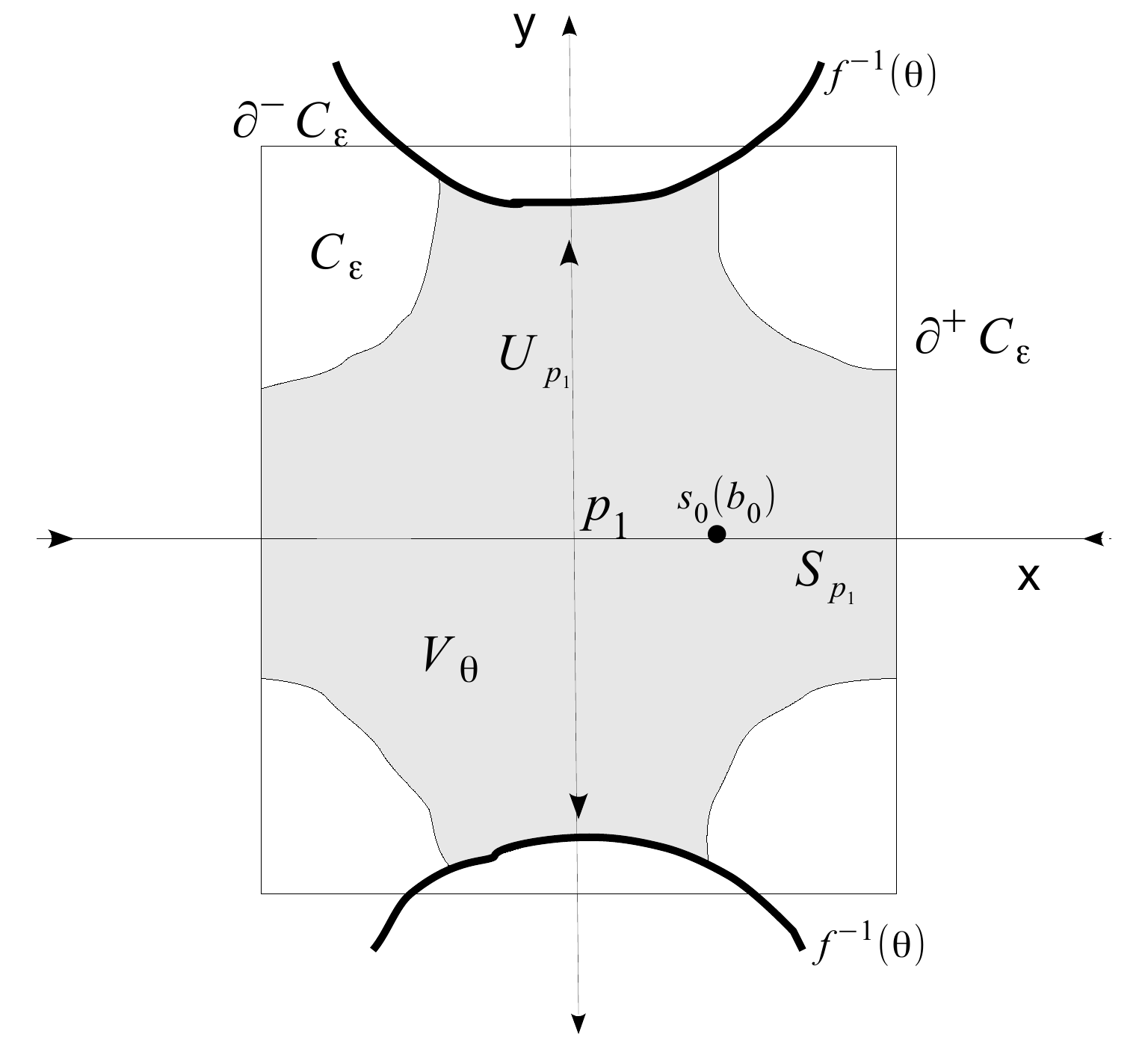}
	%\captionsetup{font=small,position=below,skip=-1pt}
	%	\caption*{Fonte: Imagem do autor.}
\end{figure}

Consider the following set: 
 \[ Z_1:=\bR\times {V}_{\theta}\times s(B_0) \subset \bR\times \mathring{{{C}_{\varepsilon}}}\times (\mathring{{{C}_{\varepsilon}}}\times B_0).\] 
  
 \begin{lemma}\label{lem.Zbordo}  The set $Z_1$ is a manifold of dimension $m+n+1$ with boundary 
 \[\partial Z_1= \bR\times(\mathring{V}^{\epsilon}_{\gamma}\cap {f}^{-1}(\theta))\times s(B_0).\]
 \end{lemma}
 \begin{proof} For any regular value $\theta$ the intersection $U\cap {f}^{-1}(-\infty, \theta]$ is a manifold with boundary for any open $U\subset M$ such that $f^{-1}(\theta)\cap U\neq \emptyset$.
 \end{proof}
 \begin{lemma} \label{s4l1}
The manifolds with boundary $Z_1$ and $\overline{W_1}$ are completely transverse inside $\bR\times\mathring{{{C}_{\varepsilon}}}\times \mathring{{{C}_{\varepsilon}}}\times B_0$, meaning that the different strata are all transverse. This implies that $Z_1\cap \overline{W_1}$ is a manifold with corners of dimension $\dim{B}+1$.
 \end{lemma}
 \begin{proof} Notice that $\mathring{\overline{W}_1}=W_1$.
 	
We start  with the intersection ${W}_1\cap \mathring{Z}_{\theta}$. The transversality is immediate from the fact that $W_1$ is a graph over the first plus the last two variables, i.e. $\bR\times\mathring{{{C}_{\varepsilon}}}\times B_0$ while the first component of $\mathring{Z}_{\theta}$ is an open subset of $\mathring{{{C}_{\varepsilon}}}$.

A similar reasoning applies for $q\in \partial_0 \overline{W_1}\cap \mathring{Z}_{\theta}$. In fact due to the transversality of the flow to a regular level set $f^{-1}(\theta)$ this also proves the transversality of ${W_1}$ with $\partial {Z}_{\theta}$.

Notice that $\partial_0{\overline{W_1}}\cap \partial {Z}_{\theta}=\emptyset$ since a point  $q\in\xi_0(B_0)$ cannot satisfy $\tilde{f}(q)=\theta$ due to property (i) of Proposition \ref{prop.triplo}.

 	The transversality of $\partial_1\overline{W_1}$ with $\mathring{Z_1}$ follows from the transversality of $s(B_0)$ and $S_{{\bf{p}_1}}\times B_0$.
	
 	Finally the transversality of $\partial_1\overline{W_1}$ with $\partial Z_1$ follows from the transversality of $\mathrm{U}_{{\bf{p}_1}}$ and $\tilde{f}^{-1}(\theta)$ together with the transversality of $s(B_0)$ and $S_{{\bf{p}_1}}\times B_0$.
 \end{proof}
 
 \begin{lemma}\label{lema.transv} Let $N_1$ and $N_2$ be submanifolds with corners of type $k$ and $l$ respectively, inside a manifold (with no corners) $N$. If they are completely transverse then $N_1\cap N_2$ is a manifold with corners of type at most $k+l$.
 \end{lemma}
 \begin{proof} We use a standard trick. Clearly, $N_1\times N_2$ is a manifold with corners of type $k+l$. Then the complete transversality of $N_1$ and $N_2$ is equivalent with the complete transversality of $N_1\times N_2$ and the diagonal submanifold $\Delta$ in $N\times N$. The conclusion then is that $(N_1\times N_2)\cap \Delta$ is a manifold with corners (for this see Proposition A.3 in \cite{Ci3}).
 \end{proof}
 
Following Lemma \ref{lem.Zbordo} denote
 \begin{equation}
 \mathcal{A}_1:= Z_1\cap \overline{W_1}
 \end{equation}
  The codimension $1$ boundary has the following decomposition in components:
\begin{eqnarray}
\partial^1\mathcal{A}_1&=&\underbrace{\{1\}\times(\mathrm{U}_{{\bf{p}_1}}\cap f^{-1}(-\infty,\theta])\times  s(s^{-1}(\mathrm{S}_{{\bf{p}_1}}\times B_0))}_{\partial^1_1\mathcal{A}_1}\bigcup \underbrace{\{0\}\times {\xi_0(B_0)}}_{\partial^1_0\mathcal{A}_1}\bigcup\nonumber\\
&&\bigcup \underbrace{\overline{W_1}\cap ([0,1]\times{f}^{-1}(\theta)\times s(B_0))}_{\partial^1_2\mathcal{A}_1}\label{bordo.A}.
\end{eqnarray}
where
 \[ \partial^1_1\mathcal{A}_1:=\partial_1\overline{W_1}\cap Z_1, \quad  \partial^1_0\mathcal{A}_1:=\partial_0\overline{W_1}\cap Z_1, \quad \partial^1_2\mathcal{A}_1:=\overline{W_1}\cap \partial Z_1.\]

The codimension  $2$ stratum is given by
\begin{equation}\label{bordo.A3}
 \partial^2\mathcal{A}_1: = \partial^1_1\mathcal{A}_1\cap\partial^1_2\mathcal{A}_1=\{1\}\times(\mathrm{U}_{{\bf{p}_1}}\cap {f}^{-1}(\theta))\times s(s^{-1}(\mathrm{S}_{{\bf{p}_1}}\times B_0)).\end{equation} 
\vspace{0.5cm}

We define the resolution map now. Let
 \[ \mathcal{R}:\mathcal{A}_1\ra V_{\theta}\times M\times B_0\] be the restriction to $\mathcal{A}_1$ of the projection onto the  three  spacial coordinates. Notice that in fact the image of $\mathcal{R}$ is contained in $V_{\theta}\times V_{\theta}\times B_0$.
 
 \begin{prop}\label{prop.propria} The map $\mathcal{R}$ is proper.
\end{prop}
\begin{proof} Recall that for locally compact metric spaces $X,Y$, a map $F:X\ra Y$ is proper if and only if for any sequence $(x_n)_{n\in \bN}\in X$  such that $\displaystyle\lim_{n\ra \infty}x_n=\infty$ one has $\displaystyle\lim_{n\ra\infty} F(x_n)=\infty$. By definition, 
\[ \lim_{n\ra \infty} x_n=\infty
\]
if for every $K\subset X$ compact there exists $n_0\in \bN$ such that $x_n\in X\setminus K$ for all $n\geq n_0$. Notice that such a sequence does not have any convergent subsequence in $X$ and the converse is also true. It follows easily then that a map $F:X\ra Y$
is \ul{not} proper if and only if there exists $(x_n)_{n\in \bN}\in X,$ $x_n\ra \infty$ such that $F(x_n)\ra y\in Y$.

Assume therefore that $(u_n)_{n\in \bN}\in \mathcal{A}_1$ satisfies $u_{n}\ra \infty$ while $\mathcal{R}(u_{n})\ra \tilde{u}$.  Now $u_n$ has $4$ components:
\begin{equation}\label{equn} u_n=(t_n,a_n',a_n,b_n)\in [0,1]\times V_{\theta}\times  V_{\theta}\times B_0
\end{equation}
 By passing to a subsequence of $u_n$ we can assume $t_n$ converges. There are two possibilities.

Either $t_n\ra t'<1$ or $t_n\ra 1$. We analyze them separately. 

First, the triple $\mathcal{R}(u_{n})=(a_n',a_n,b_n)$ converges in $V_{\theta}\times V_{\theta}\times B_0$. From $u_n\in Z_1$ we get  $(a_n,b_n)=(s_0(\beta_n),\beta_n)$ for some $\beta_n\in B_0$.  Hence $\beta_n$ converges to $\beta\in B_0$ and $s_0(\beta_n)\ra s_0(\beta)$.

If $t_n\ra t'<1$ then for $n$ big enough we have that $u_n\in \overline{W_1}\setminus \{t=1\}$ and therefore 
\[a_n'=\psi_{\frac{t_n}{1-t_n}}(a_n)
\]
Hence $a_n=s_0(\beta_n)\ra s_0(\beta)$ and $a_n'\ra \psi_{\frac{t'}{1-t'}}(s_0(\beta))$ and by hypothesis this is in $V_{\theta}$. We conclude that

\[ u_n\ra \left(t', \psi_{\frac{t'}{1-t'}}(s_0(\beta)),s_0(\beta),\beta\right)
\]
and this limit belongs  to $Z_1\cap \overline{W_1}=\mathcal{A}_1$. Contradiction with $u_n\ra \infty$.

If $t_n\ra 1$ we have that $u_n=(t_n,a_n',s_0(\beta_n),\beta_n)$ and the convergence of $(a_n',s_0(\beta_n),\beta_n)$ to a point in $V_{\theta}\times V_{\theta}\times B_0$ implies again that $\beta_n\ra \beta\in B_0$. Since $s_0(B_0)\subset V_{\theta}$ we have that $s_0(\beta_n)\ra s_0(\beta)\in V_{\theta}$.  We get therefore that $u_n$ converges in $[0,1]\times V_{\theta}\times V_{\theta}\times  B_0 $ since all its coordinates converge. In order to reach a contradiction we need only check that it converges to some element of $Z_1\cap \overline{W_1}$. We have that
\[u_n\ra u\in  \bR\times\mathring{{{C}_{\varepsilon}}}\times \mathring{{{C}_{\varepsilon}}}\times B_0.\]

 On the other hand, since $u_n\in\overline{W_1}$ and the closure of $W_1$ is taken within $ \bR\times\mathring{{{C}_{\varepsilon}}}\times \mathring{{{C}_{\varepsilon}}}\times B_0$ we get that $u\in \overline{W_1}$. One sees easily that $u\in Z_1$ since $(a_n',s_0(\beta_n),\beta_n)=\mathcal{R}(u_n)$ converges to a point $(a',s_0(\beta),\beta)\in V_{\theta}\times V_{\theta}\times B_0$.

\end{proof}

Let $\tilde{V}_{\theta}:=V_{\theta}\times M\times B_0$ be the codomain of $\mathcal{R}$. We show two things:
\begin{itemize}
\item[(i)] the currential formula  \eqref{bordo.T} holds on $\mathring{\tilde{V}}_{\theta}$, the interior of $\tilde{V}_{\theta}$. 
\item[(ii)] there exists a map  from a manifold with corners $N$ to  $\partial\tilde{V}_{\theta}=V^{\epsilon}_{\gamma}\cap f^{-1}(\theta)$ that allows us to continue the process.
\end{itemize}

First we list some set-theoretic and differential properties of $\mathcal{R}$.

\begin{prop}\label{prop.1} The map $\mathcal{R}:\mathcal{A}_1\longrightarrow \tilde{V}_{\theta}$ satisfies:
	\begin{enumerate}
		\item $\mathcal{R}(\mathcal{A}_1)=\overline{\displaystyle\bigcup_{t\geq 0}\xi_t(B_0)}\bigcap  \tilde{V}_{\theta}=:A_{\theta}$ where the closure is taken within  $M\times M\times B_0$.
		\item $\mathcal{R}(\partial^{1}_{1}\mathcal{A}_1)=(\mathrm{U}_{{\bf{p}_1}}\cap f^{-1}(-\infty,\theta])\times s(s^{-1}(\mathrm{S}_{{\bf{p}_1}}\times B_0))$;
		\item $\mathcal{R}(\partial_0^1\mathcal{A}_1)=\xi_0(B_0)$;
		\item $\mathcal{R}(\partial_2^1\mathcal{A}_1)=A_{\theta}\cap \tilde{f}^{-1}(\theta)$;
		\item the restriction of $\mathcal{R}$ is a  bijection from $\mathcal{A}_1\setminus \overline{W}_{{\bf{p}_1}} $ (see (\ref{Wp1})) onto its image.	
		\item the restriction of $\mathcal{R}$ to $\partial^1_2\mathcal{A}_1$ is a bijection onto the image.	
	\end{enumerate}
\end{prop}
\begin{proof} Let $q\in \overline{W_1}\cap Z_1$. On one hand
\[ q=\lim_{n\ra \infty}\left (t_n,\psi_{\frac{t_n}{1-t_n}}\left(x_n,y_n\right),x_n,y_n,b_n\right)
\]
where $b_n\in B$, $(x_n,y_n)\in \mathring{C}_{\ve}$, $t_n\in[0,1]$. From $q\in Z_1$ it follows that $(x_n,y_n,b_n)\ra (s_0(\beta),\beta)$. Suppose $t_n\ra t'<1$. Then 
\[\psi_{\frac{t_n}{1-t_n}}(x_n,y_n)\ra \psi_{\frac{t'}{1-t}}(s_0(\beta))
\]
Hence, in this case $\mathcal{R}(q)=\left (\psi_{\frac{t'}{1-t'}}(s_0(\beta)),s_0(\beta),\beta\right)\in \xi_{\frac{t'}{1-t'}}(B_0)$. 

When $t_n\ra 1$, $q\in \overline{W_1}$ implies that $(x_n,y_n,b_n)\ra (s_0(\beta),\beta)\in S_{{\bf{p}_1}}\times B_0$, i.e. $\beta\in s^{-1}(S_{{\bf{p}_1}}\times B_0)$ and $\psi_{\frac{t_n}{1-t_n}}\left(x_n,y_n\right)\ra q_1\in U_{{\bf{p}_1}}\cap V_{\theta}$.

We argue that due to the transversality of $s$ with $S_{{\bf{p}_1}}\times B_0$ all points in $U_{{\bf{p}_1}}\cap (V_{\theta}\times s(s^{-1}(S_{{\bf{p}_1}}\times B_0)))$ are limits of type $\xi_{t_n}(s_0(\beta_n),s_0(\beta_n),\beta_n)$, $t_n\geq 0$.  Fix first $q_2=(x_2,0,b_2)\in s(s^{-1}(S_{{\bf{p}_1}}\times B_0))$. Since $s\pitchfork B_0$ we can provide  a submanifold $B_0'\subset B_0$  of dimension equal to $\dim{U}_{{\bf{p}_1}}$ such that $s_0\bigr|_{B_0'}\pitchfork S_{{\bf{p}_1}}$ and $q_2\in s(B_0')$.  Take then a transverse small disk $D$ in $q_2+T_{q_2}s(B_0')$ centered at $q_2$. The trajectories originating in this disk will cut $s(B_0')$ exactly once due to the transversality of $s(B_0')$ to the flow. This stays true even if $q_2=(0,0,b_2)$ is  critical. By Corollary  \ref{cortv} which can be applied also to "slanted" disks, given any point $(0,y_2)\in U_{{\bf{p}_1}}$ there exists a sequence of points $u_n\in D$ with $u_n\ra q_2$ and a corresponding sequence of points on trajectories determined by $u_n$ that converges to $(0,y_2)$.

 This finishes the inclusion $\mathcal{R}(\mathcal{A}_1)\subset \overline{\bigcup_{t\geq 0}\xi_t(B_0)}$. 

 Conversely, let $ (a,b,c)\in \overline{\bigcup_{t\geq 0}(\xi_t(B_0))}\bigcap  \tilde{V}_{\theta}$. Then there exist $t_n\geq 0$ and $b_n\in B_0$ such that $(\psi_{t_n}(s_1(b_n)),s_1(b_n), b_n)\ra (a,b,c)$ with $c\in B_0$. By passing to a subsequence one can assume that $t_n\ra t_0$ or $t_n\ra \infty$.  Since $a\in V_{\theta}\subset \mathring{V}^{\vre}_{\gamma}$ and the latter is open in $M$ we can consider    $\psi_{t_n}(s_1(b_n))\in \mathring{V}^{\vre}_{\gamma}$       for $n$ big enough. However $\psi_{t_n}(s_1(b_n))$ might not be in $V_{\theta}$ for infinitely many $n$, since it could happen that $f(\psi_{t_n}(s_1(b_n)))>\theta$ for a subsequence.
  
 Let $r_n:=\frac{t_n}{1+t_n}$. We have that
\[ u_n:=\left(r_n,\psi_{t_n}(s_1(b_n)),s_1(b_n), b_n\right)\in W_1
\]
Since $r_n$, $b_n$ and $\psi_{t_n}(s_1(b_n))$ all converge we  have that in fact $u_n$ converges to a point $u\in [0,1]\times{V}_{\theta}\times\mathring{{C_{\vre}}}$ that necessarily lies  in $\overline{W_1}$. The limit $u=(r,a,b,c)$ will also be a point in $Z_1$ since $(b,c)\in s(B_0)$ and $a\in V_{\theta}$. Hence $(a,b,c)\in \mathcal{R}(\mathcal{A}_1)$.

\vspace{0.3cm}

Statements (2) and (3) are trivial from the description of $\partial^{1}_{1}\mathcal{A}_1$ and $\partial_0^1\mathcal{A}_1$.

For (4) the inclusion $\subset$ is straightforward.  For the inclusion $\supset$ notice that if $(a,b,c)\in \mathcal{A}_1$ then either ${f}(a)<\theta$ in which case $a\in \mathring{{V_{\theta}}}$ and so $(a,b,c)\in \mathring{Z_1}\cap \overline{W_1}$ or ${f}(a)=\theta$ in which case $(a,b,c)\in \partial  {Z_1}\cap \overline{W_1}=:\partial^1_2\mathcal{A}$.

\vspace{0.3cm}

For (5) one notices that   for $t\neq 1$ the (restriction of the) map $\mathcal{R}$ is injective away from the points corresponding to $s^{-1}(\{{\bf{p}_1}\}\times B_0)$.  Moreover $\mathcal{R}$ is injective when $t=1$. For $p\neq q$ with $t_p\neq 1$ and $t_q=1$, $\mathcal{R}(p)\neq \mathcal{R}(q)$ unless $p,q\in \overline{W}_{{\bf{p}_1}}$.
\vspace{0.5cm}

For (6) one notices that $\overline{W}_{{\bf{p}_1}}\cap \partial^1_2\mathcal{A}_1=\emptyset$.
	
	\end{proof}
	
	\begin{corollary}\label{coreqfund} The equality of currents (\ref{bordo.T}) holds on the open set $\mathring{V}_{\theta}\times M\times B_0$.
	\end{corollary}
	\begin{proof} The intersection $\mathring{\mathcal{A}}_{1}:=\mathcal{A}_1\cap (\mathring{V}_{\theta}\times M\times B_0)$ is a manifold with boundary since $\partial^1_2\mathcal{A}_1$ gets eliminated in the intersection. Since $\mathcal{R}$ is proper we can push-forward any current.  Use:
	\[ d (\mathcal{R}_*(\mathring{\mathcal{A}}_{1}))=\mathcal{R}_*(d\mathring{\mathcal{A}}_{1})=\mathcal{R}_*(\partial^1_1\mathcal{A}_1)-\mathcal{R}_*(\partial^1_0\mathcal{A}_1).
	\]
	Now $\mathcal{R}\bigr|_{\mathring{\mathcal{A}}_{1}}$ away from a set of zero measure\footnote{with respect to any measure induced by the volume form of a Riemannian metric on $\mathcal{A}_1$} is a  bijection onto 
	\[\xi([0,\infty)\times B_0)\cap \mathring{V}_{\theta}\times M\times B_0.\] It follows from the area formula that
	\[\mathcal{R}_*(\mathring{\mathcal{A}}_{1})=T\bigr|_{\mathring{V}_{\theta}\times M\times B_0}
	\]
	where $T$ is the current appearing in (\ref{bordo.T}).
		\end{proof}
		
		This is the first step. In order to proceed further we will use the restriction $\mathcal{R}:\partial^1_2\mathcal{A}_1\ra f^{-1}(\theta)\times M\times B_0=\tilde{f}^{-1}(\theta)$. We fix now another critical point ${\bf p}_2$ with $f({\bf p}_2)>f({\bf{p}_1})$. In order to implement the program we need the following.
		
		\begin{lemma}\label{otranslem} The restriction of $\mathcal{R}$  denoted $\sigma: \partial_{2}^{1}\mathcal{A}_1\longrightarrow \tilde{f}^{-1}(\theta)$ is completely transverse to $\mathrm{S}_{{\bf p}_2}\times M\times B_0$ within $\tilde{f}^{-1}(\theta)$ for all critical points ${\bf p}_2$.
\end{lemma}

First a clarification. The complete transversality for the map $\sigma$ is meant here both within the ambient space $\tilde{f}^{-1}(\theta)$ (with $\mathrm{S}_{p_2}\times M\times B_0\cap \tilde{f}^{-1}(\theta)$) and within the ambient space $M\times M\times B_0$. The two statements are clearly equivalent due to the transversality of $S_{p_2}$ to ${f}^{-1}(\theta)$ for every regular $\theta$.
\begin{proof} Since $\partial_{2}^{1}\mathcal{A}_1$ is a manifold with boundary $\partial^2\mathcal{A}_1$, as defined in  \eqref{bordo.A3}, we need to show transversality at points
	\begin{itemize}
	\item[(1)]  $q\in \sigma(\partial^2\mathcal{A}_1)\cap (\mathrm{S}_{{\bf p}_2}\times M\times B_0)$
	
	\item[(2)] $q\in \sigma(\partial_{2}^{1}\mathcal{A}_1\setminus \partial^2\mathcal{A}_1) \cap (\mathrm{S}_{{\bf p}_2}\times M\times B_0).$
	\end{itemize}
	For both situations, take the unique (by (6) of Prop.\ref{prop.1} ) $q'\in \partial_{2}^{1}\mathcal{A}_1$ such that $\sigma(q')=q$. The two situations are distinguished by $t_{q'}=1$ (for (1)) or $t_{q'}\neq 1$.
	
	From the explicit expression of $\partial^2\mathcal{A}_1$ in  (\ref{bordo.A3}) we see that transversality for (1) is implied by the Smale property of the flow. 
	
	For (2) since $t\neq 1$ we can give another description to $\overline{W_1}\setminus\{t=1\}\cap (f^{-1}(\theta)\times M\times B_0)$ as the graph of the time map defined over $\xi_0(B_0\setminus s^{-1}(S_{{\bf{p}_1}}\times B_0))$ that associates to a point $p$ the time $t_p$ it needs to reach $\tilde{f}^{-1}(\theta)$. We deduce that the intersection of this time map graph with the flow-invariant $S_{p_2}\times M\times B_0$ can be described as the flow-out of the intersection $\xi_0(B_0\setminus s^{-1}(S_{{\bf{p}_1}}\times B_0))\cap (S_{{\bf p}_2}\times M\times B_0)$ to the level set $\tilde{f}^{-1}(\theta)$. By the transversality of $s$ with $S_{p_2}$ we get that this intersection is transverse. Moreover, since the flow preserves transversality we get by Proposition \ref{FlB} the transversality condition we are after within $\tilde{f}^{-1}(\theta)$.
\end{proof}
	
	We resume what we did so far. We started with the proper submanifold  $\xi_0(B_0)$ of $V_{\theta}\times M\times B_0$ and we constructed a resolution of its flow-out in the open set $\mathring{V}_{\theta}\times M\times B_0$. The resolution took the form of a map $\mathcal{R}$ from a manifolds with corners $\mathcal{A}_1$ to $V_{\theta}\times M\times B_0$. Of course this only solves the problem of the flow-out of $\xi_0(B_0)$ for the first critical point encountered.  How do we go from here. 
	
	The map $\sigma:\partial^1_2\mathcal{A}_1\ra \tilde{f}^{-1}(\theta)$ will play now the role of $\xi_0:B_0\ra M\times M\times B_0$ and we would like to apply the same ideas to $\sigma$.

	So we would like to see what properties of $\sigma$ can be preserved when going through the next critical level. Clearly by composing with the flow-diffeomorphisms we can assume that the image of $\sigma$ is contained in a regular level of $\tilde{f}$ close to the next critical level.

	There is no harm in assuming that the next critical level lies within $\tilde{f}^{-1}(0)$ by changing $f$ to $f+c$ for some constant $c$.  Moreover we will be using certain neighborhoods of the critical sets. First, since the Shilnikov-Minervini results are local around the critical points we will use different neighborhoods around the critical sets of $\tilde{f}^{-1}(0)$. 
	
	\begin{itemize}
	\item We will denote $\tilde{C}_{\vre}:= \bigcup_{p\in \Crit(f)\cap f^{-1}(0)} C_{\vre}(p)\times M\times B_0$, where $\vre$ is chosen so that the results of Section \ref{s.BVP} hold for each cube $C_{\vre}(p)$ with $p$ in the finite set $\Crit(f)\cap f^{-1}(0)$.
	\item For each ${\bf p}\in \Crit(f)\cap f^{-1}(0)$ we will chose a $\theta_p$ and $\gamma_p$ small enough so that   Proposition \ref{prop.triplo} item [(ii)] is satisfied (observe that the points ${\bf p}\in \Crit(f)\cap f^{-1}(0)$ are not points of local minimum in our context). Then let $\theta_1:=\min {\theta_{\bf p}}$, $\gamma:=\min {\gamma_{\bf p}}$ $V_{\theta_{\bf p}}:=V^{\vre}_{\gamma}({\bf p})\cap f^{-1}[-\theta_1,\theta_1]$,  $V_{\theta_1}=\cup_{{\bf p}} V_{\theta_{\bf p}}$ and
	\[ \tilde{V}_{\theta_1}:=V_{\theta_1}\times M\times B_0.
	\]
	\end{itemize}

	%\begin{itemize}

	%\item[(i)] while $\sigma$ is a bijection onto the image this seems to be an accident\footnote{this property is lost in the next step of the induction as one will notice later}; so we might as well think that $\sigma$ is not an embedding like $\xi_0$. 
	%\item[(ii)] a  property that we used implicitly a few times (for example when proving properness) is that the projection $M\times M\times B_0\ra M\times B_0$ onto the last two coordinates is a fiber bundle and $\xi_0$ is a section of this fiber bundle. It is quite easy to see that the projection of the image of $\sigma$ onto the last two coordinates is by no means injective. For that just take a look at what happens for $t=1$. For each point $p\in s(s^{-1}(S_{p_1}\times B_0))$ the set $U_{p_1}\cap f^{-1}(\theta)\times \{p\}\times \{1\}$  is in $\partial^2\mathcal{A}_{\theta}$, hence in the domain of $\sigma$.
	%\item[(iii)] we chose $B_0$ so that the image of $\xi_0$ fit into a small enough neighborhood $V_{\theta}\times M\times B_0$ of a critical manifold $\{p_1\}\times M\times B_0$ that satisfies certain properties; when we flow-out the image of $\sigma$ from the level  $\tilde{f}^{-1}(\tilde{f}(p_1))$ away we cannot make sure we are staying within convenient neighborhoods of critical points/manifolds. How can this be dealt with?  
	%\end{itemize} 
	
	%In a sense all the points raised are negative points. But on the positive side and this will be of crucial importance is that the image of $\sigma$ is contained within a regular level set. 

Here is the model result we are after.
\begin{prop}\label{model.prop} Let $N$ be an oriented manifold with corners of dimension $n$ e type $k\geq1 $\footnote{The type of a manifold with corners gives the codimension of the smallest strata in the boundary. Some authors prefer to call this "depth".} and let $-\theta$ be a regular level for $\tilde{f}$ such that $0$ is a critical level and $\theta$ is small enough. Let $\theta'<\theta_1$ and $\sigma:N\rightarrow \tilde{f}^{-1}(-\theta' )\cap \tilde{V}_{\theta_1}$, $\sigma=(\sigma_1,\sigma_2,\sigma_3)$ be a smooth map such that
\begin{itemize}
\item[(a)] $\sigma$ is proper
\item[(b)]  $\sigma$ is completely transverse to all stable manifolds $\mathrm{S}_{\bf p}\times M\times B_0$ for ${\bf p}\in \Crit(f)$; equivalently $\sigma_1\pitchfork \mathrm{S}_{\bf p}$, for all ${\bf p}\in\Crit(f)$.
\item[(c)]  $\sigma\bigr|_{N\setminus \partial^2N}$ is injective where $\partial^2N$ is the collection of strata of codimension at least $2$.
\item[(d)] $\sigma_1\bigr|_{N^0}= \alpha((\sigma_2,\sigma_3)\bigr|_{N^0})$ for some function $\alpha$ where $N^0$ is the top stratum of $N$.
\end{itemize}
Then there exists a smooth map, called flow resolution $\mathcal{R}_{\sigma}: \mathcal{A}_\sigma\longrightarrow  \tilde{V}_{\theta_1}$ defined over an oriented manifold with corners $\mathcal{A}_{\sigma}$ of dimension $n+1$  and type $k+2$ such that
\begin{itemize}
\item[(i)] $\mathcal{R}_{\sigma}$ is proper;
\item[(ii)] $\mathcal{R}_{\sigma}(\mathcal{A}_\sigma)=\overline{\left(\bigcup_{t\geq 0}\xi_t(\sigma(N))\right)}\cap \tilde{V}_{\theta_1}$ where the closure is taken in $M\times M\times B_0$.
\item[(iii)] $\mathcal{R}_{\sigma}$ is a bijection from an open subset of $\mathcal{A}_{\sigma}$ of full measure to  an open subset of full measure of $\left(\bigcup_{t\geq 0}\xi_t(\sigma(N))\right)\cap \tilde{V}_{\theta_1}.$
\item[(iv)] $\mathcal{A}_{\sigma}$ has a  distinguished boundary $N':=\partial^1_{\theta_1}\mathcal{A}_{\sigma}$ of dimension $n$ and type $k+1$ such that $\sigma':=\mathcal{R}_{\sigma}|_{N'}: N'  \longrightarrow \tilde{f}^{-1}(\theta_1)\cap \tilde{V}_{\theta_1}$ is completely transverse to all the stable manifolds $S_{\bf p}\times M\times B_0$, is injective on  $N'\setminus\partial^2N'$ and the components of $\sigma'$ satisfy property of item $(d)$ when restricted to $(N')^0$.
\end{itemize}
\end{prop}

\begin{remark} The word distinguished is related to the fact that $\partial^1_{\theta_1}\mathcal{A}_{\sigma}$ is not the full boundary of $\mathcal{A}_{\sigma}$ but one that has a collar neighborhood.
\end{remark}

\begin{remark} It is important to specify the codomain of $\sigma$ in order to state the properness property. In the induction process we are using, the original map $\mathcal{R}\bigr|_{\partial^1_2\mathcal{A}_{\theta}}$ is  proper when the codomain is  $V_{\theta}\times M\times B_0\cap \tilde{f}^{-1}(\theta)$. The later is an open set inside $\tilde{f}^{-1}(\theta)$. Clearly the inclusion of an open set into the ambient space is not proper. 
\end{remark}
\begin{remark} The injectivity property stated in item (c) appears because  we want the current $(\mathcal{R}_{\sigma})_*(\mathcal{A}_{\sigma})$ to pe determined by the image of $\mathcal{R}_{\sigma}$. Otherwise one could have multiplicities or worse things happening. One cannot expect injectivity to hold everywhere. 
\end{remark}
\begin{remark} The seemingly strange property (d) is to insure the "replication" of the injectivity property away from the codimension $2$ stratum.  Property (d) is fulfilled for the initial $\sigma$, the restriction of $\mathcal{R}$ to $\partial^1_2\mathcal{A}$. In that case, $N^0$ is the graph of the time map $p\ra t_p$ for $p\in \xi_0(B_0\setminus \xi_0^{-1}(S_{{\bf p}_1}\times M\times B_0))$ as described in the proof of Lemma \ref{otranslem} while $\sigma$ projects $(t_p,p)$ to $p$. Since the first component of $\xi_0$ is dependent on the other two, we get the claim.
\end{remark}

In order to prove  Proposition \ref{model.prop} we  also need to deal with the fact that $\sigma(N)$ is not necessarily a subspace of $V_{\theta_1}\times M\times B_0$.  Hence rather than "flowing" $\sigma(N)$ we will consider the graph $\Gamma_{\sigma}$. It is convenient to consider first a proper embedding $N\hookrightarrow \bR^j$ and look at the closure of the set
\[ W_2:=\left\{ \left(t,\psi_{\frac{t}{1-t}}(m_1),m_1,m_2,b,n\right)~|~  0<t<1\right\}\cap \bR\times\mathring{C}_{\vre}\times \mathring{C}_{\vre}\times M\times B_0\times \bR^j
\]
inside $ \bR\times\mathring{C}_{\vre}\times \mathring{C}_{\vre}\times M\times B_0\times \bR^j$. 

By Remark \ref{rem001}, this closure is a manifold of dimension $2m+n+j+1$ with boundary
\begin{eqnarray*} \partial^1_1 \overline{W_2}:= \bigcup_{{\bf p}\in f^{-1}(0)\cap \Crit(f)}  \{1\}\times(U_{{\bf p}}\cap \mathring{C}_{\vre}({\bf p}))\times (S_{{\bf p}}\cap \mathring{C}_{\vre}({\bf p}))\times M\times B_0\times \bR^j \bigcup
  \\
         \partial^1_0 \overline{W_2}:=   \{0\}\times \triangle_{\mathring{C}_{\vre}}\times M\times B_0\times \bR^j.\qquad\qquad\qquad
\end{eqnarray*}
\vspace{0.3cm}

 Let
 \[Z_{\sigma}:= \bR\times V_{\theta_1}\times \Gamma_{\sigma}\subset \bR\times V_{\theta_1}\times V_{\theta_1}\times M\times B_0\times N.\]  
 
 Since $V_{\theta_1}$ is a manifold with boundary we get that $Z_{\sigma}$ is a manifold with corners of dimension $m+n+1$ and type $k+1$. The codimension $1$ boundary components of $Z_{\sigma}$ are
 \begin{eqnarray}\partial^1_0Z_{\sigma}:=\bR\times(V_{\theta_1}\cap f^{-1}(-\theta_1))\times  \Gamma_{\sigma} \qquad  \quad \qquad\\
  \partial^1_1Z_{\sigma}:=\bR\times(V_{\theta_1}\cap f^{-1}(\theta_1))\times  \Gamma_{\sigma} \qquad  \qquad\qquad \\
\partial^1_{j+1}Z_{\sigma}:=\bR\times V_{\theta_1}\times \Gamma_{\sigma\bigr|_{\partial^1_{j} N}},\;\; 1\leq j\leq f_{N}
 \end{eqnarray}
where $f_N$ is the number of codimension $1$ boundary components of $N$, i.e. the number of connected components of $\partial^1 N\setminus \partial^2 N$, assumed finite.

\begin{lemma} \label{Lem} The manifolds $\overline{W_2}$ and $Z_{\sigma}$ are completely transverse inside $\bR\times\mathring{C}_{\vre}\times \mathring{C}_{\vre}\times M\times B_0\times \bR^j$. This implies that $Z_{\sigma}\cap \overline{W_2}$ is a manifold with corners of dimension $n+1$ and type at most $k+2$. 

Moreover, the codimension $1$ boundary of $Z_{\sigma}\cap \overline{W_2}$ has the following components which are themselves manifolds with corners
\begin{eqnarray} \quad  \partial^1_1\overline{W_2}\cap Z_{\sigma}\quad&=&\bigcup_{{\bf p}\in \Crit(f)\cap f^{-1}(0)}\{1\}\times (U_{{\bf p}}\cap f^{-1}([0,\theta_1]))\times \Gamma_{\sigma_{|\sigma^{-1}(S_{\bf p}\times M\times B_0)}} \\
 \label{levelthetaprim} \partial^1_0\overline{W_2}\cap Z_{\sigma}\quad&= &\{0\}\times\Gamma_{\tilde{\sigma}}\\ 
 \label{boundtheta} \overline{W_2}\cap \partial^1_1Z_{\sigma}\quad&&\\
 \overline{W_2}\cap \partial^1_{j+1}Z_{\sigma}&\mbox{for}&1\leq j\leq f_N.
\end{eqnarray}
where $\tilde{\sigma}=(\sigma_1,\sigma_1,\sigma_2,\sigma_3)$ given that $\sigma=(\sigma_1,\sigma_2,\sigma_3)$.  Finally,
 \begin{equation} \label{emptyset}\overline{W_2}\cap \partial^1_0Z_{\sigma}\quad=\emptyset.\end{equation}
\end{lemma}
\begin{proof} Analogous to the proof of Lemma \ref{s4l1}. 

The reason for (\ref{emptyset}) is that the $f$-value of the second component of $\overline{W_2}$ is at least as big as the $f$-value of the third component while this is not the case for an element of $\partial^1_0Z_{\sigma}$ due to $f \circ\sigma_1=-\theta'>-\theta_1$.  
\end{proof}
 \begin{remark} The reason for which we chose $\Imag \sigma\subset \tilde{f}^{-1}(-\theta')$ is because we did not want $\Imag \sigma \subset\partial \tilde{V}_{\theta_1}$. That would render  Lemma \ref{Lem} false. An alternative approach, if $\Imag \sigma\subset \tilde{f}^{-1}(-\theta_1)$, would be to  replace $V_{\theta_1}$ in the definition of $Z_{\sigma}$ with $V^{\theta'}_{\theta_1}=f^{-1}([-\theta',\theta_1])\cap V^{\epsilon}_{\gamma}$ where $\theta'<\theta_1$. Then one has to be content with the construction of the resolution for the flow-out of $\sigma(N)$ in between levels $-\theta'$ and $\theta_1$. 
 \end{remark}

Let $\mathcal{A}_{\sigma}:=\overline{W_2}\cap Z_{\sigma}$. This is an oriented manifolds with corners because $\overline{W_2}$ and $Z_{\sigma}$ are both oriented. The convention here is that the components which correspond to graphs of smooth functions over a certain oriented base manifold $B'$ inherit the orientation of the manifold $B'$. Hence the direction of the flow gives the first vector of a positively oriented basis. Other than this, we respect the order of factors in a product. 

 Consider now
\[ \mathcal{R}_{\sigma}:\mathcal{A}_{\sigma}\ra V_{\theta_1}\times M\times B_0
\]
to be the restriction of the projection onto the second, fourth and fifth components of the product $\bR\times\mathring{C}_{\vre}\times \mathring{C}_{\vre}\times M\times B_0\times \bR^j$.

\begin{prop}\label{proper2} The map $\mathcal{R}_{\sigma}$ is proper.
\end{prop}
\begin{proof} Suppose just like in Proposition \ref{prop.propria} that there exists a sequence $u_n\in\mathcal{A}_{\sigma} $ such that $u_n\ra \infty$ and $\mathcal{R}_{\sigma}(u_n)$ converges. Now, $u_n$ has six components
\[ u_n=(t_n,a_n',a_n,m_n,b_n, z_n)
\]
with $t_n\in [0,1]$. By passing to a subsequence we can assume that $t_n\ra t'\in[0,1]$. We will show that in both cases $t'=1$ and $t'\neq 1$ one reaches a contradiction by showing that there exists a subsequence of $u_n$ which converges in $\mathcal{A}_{\sigma}$.

For the case $t'=1$ we will use the fact that $\{t=1\}\cap   \overline{W}\cap ([0,1]\times V_{\theta_1}\times V_{\theta_1})$ is compact and thus has a compact neighborhood. In fact, one notices that $\overline{W}$ is completely transverse to $V_{\theta_1}\times V_{\theta_1}\times \bR$ and their intersection is a manifold with corners, with one of the codimension $1$ boundaries being contained in $\{t=1\}$
\[ \bigcup_{{\bf p}\in\Crit(f)\cap f^{-1}(0)}\{1\}\times(U_{\bf p}\cap V_{\theta_1})\times(S_{\bf p}\cap V_{\theta_1}).
\]
This is compact. We conclude that from $(a_n',a_n)$ we can extract a subsequence denoted again $(a_n',a_n)$ that converges to a point in $V_{\theta_1}\times V_{\theta_1}$.  On the other hand, we have by  hypothesis that $(a_n',m_n,b_n)$ converges. Hence by passing to a subsequence we conclude that $(a_n,m_n,b_n)$ converges to some point $(a,m,b)$. But $(a_n,m_n,b_n)=(\sigma_1(z_n),\sigma_2(z_n),\sigma_3(z_n))$ and $\sigma$ is proper. It follows that we can extract yet another subsequence this time from $z_n$  that converges to $z\in N$ (just take $\sigma^{-1}(K)$ where $K$ is a compact neighborhood of $(a,m,b)$). But then by the continuity of $\sigma$ we get that $(a_n,m_n,b_n, z_n)$ converges to $(a,m,b,z)\in\Gamma_{\sigma}$. Since $a_n'$ converges to a point in $V_{\theta_1}$ we conclude that $u_n$ converges to a point in $\mathcal{A}_{\sigma}$, contradiction with $u_n\ra \infty$.

For $t'\neq 1$ we use that $a_n'=\psi_{\frac{t_n}{1-t_n}}(a_n)$ and since $a_n'$ converges and $\frac{t_n}{1-t_n}\ra \frac{t'}{1-t'}$ we conclude that $a_n$ also converges. We claim that $a_n$ converges to $a\in V_{\theta_1}$. First $a_n'\ra a'\in V_{\theta_1}$ by hypothesis. Now $V_{\theta_1}$ is a flow-convex neighborhood and $a'=\psi_{\frac{t'}{1-t'}}(a)$. It follows that $a\in V_{\theta_1}$. Now the contradiction is obtained as before $(a_n,m_n,b_n)=(\sigma_1(z_n),\sigma_2(z_n),\sigma_3(z_n))$, etc.
\end{proof}

We now complete the proof of Proposition \ref{model.prop}.
\begin{proof} We need only be concerned with items (ii)-(iv). The distinguished boundary is defined as:
\[ N':=\partial_{\theta_1}^1\mathcal{A}_{\sigma}:=\overline{W_2}\cap \partial^1_1Z_{\sigma}=\overline{W_2}\cap( [0,1]\times {f}^{-1}(\theta_1)\times \Gamma_{\sigma}).
\]
Clearly $\mathcal{R}_{\sigma}({\partial_{\theta_1}^1\mathcal{A}_{\sigma}})\subset \tilde{f}^{-1}(\theta_1)$.

In order to prove (ii) it is useful to consider the projection $\tilde{\mathcal{R}}_{\sigma}$ from $\mathcal{A}_{\sigma}$ onto the second, fourth, fifth and sixth components. Then the image of this map will give the closure of the flow-out of $\Gamma_{\sigma}$ inside $V_{\theta_1}\times M\times B_0\times \bR^j$ and the proof of this fact follows the same lines as the proof of item (1) in Proposition \ref{prop.1}. Then projecting the closure of the flow-out of $\Gamma_{\sigma}$ onto $V_{\theta_1}\times M\times B_0$ equals $\overline{\left(\bigcup_{t\geq 0}\xi_t(\sigma(N))\right)}\cap \tilde{V}_{\theta_1}$ and this takes care of (ii).

The map $\tilde{R}_{\sigma}$ mentioned in the previous paragraph is injective on $\mathcal{A}_{\sigma}\setminus \{t=1\}$ since on this set all points are of type
\[ (t,\psi_{\frac{t}{1-t}}(\sigma_1(z)),\sigma_1(z),\sigma_2(z),\sigma_3(z),z),\qquad z\in N, t\in [0,1)
\]
which get projected to $(\psi_{\frac{t}{1-t}}(\sigma_1(z)),\sigma_2(z),\sigma_3(z),z)$. Clearly all the points in the forward-flowout of $\Gamma_{\sigma}$ that are inside $V_{\theta_1}\times M\times B_0\times \bR^j$ are also in the image of this projection. In order to obtain a set where $\mathcal{R}_{\sigma}$ is injective one needs to take out more points.

Since $\sigma\bigr|_{\partial^2N}$ is completely transverse when restricted to all distinguished boundary pieces $\partial^2_i N$ of $\partial^2N$, it follows that we can define a natural subspace $\mathcal{A}_{\sigma\bigr|_{\partial^2N}}$ of $\mathcal{A}_{\sigma}$ by taking the union of the corresponding sets $ \overline{W_2}\cap Z_{\sigma\bigr|_{\partial^2_iN}}$.  Since this $\mathcal{A}_{\sigma\bigr|_{\partial^2N}}$ will be a union of manifolds with corners of lower dimension it will  have measure zero inside $\mathcal{A}_\sigma$. Then $\mathcal{R}_{\sigma}$ will be injective on $\mathcal{A}_{\sigma}\setminus \left(\{t=1\}\cup \mathcal{A}_{\sigma\bigr|_{\partial^2N}}\right)$  and this takes care of (iii).

In order to prove transversality we can use  Lemma \ref{transvlem} in order to reduce  to the proof of transversality of $\tilde{\mathcal{R}}_{\sigma}\bigr|_{\partial^1_{\theta_1}\mathcal{A}_{\sigma}}$ with $S_p\times M\times  B_0\times \bR^j$. This then follows the same scheme as Lemma \ref{otranslem}.

For injectivity we note that $ \partial^1_0\overline{W_2}\cap\partial^1_1Z_{\sigma}=\emptyset$ and we separate $ N'=\partial^1_{\theta_1}\mathcal{A}_{\sigma}$ into two parts:
\[ N_1:=(\overline{W_2}\setminus \{t=1\})\cap \partial^1_1Z_{\sigma} \;\;\mbox{and}\;\; N_2:=\partial^1_1\overline{W_2}\cap\partial^1_1Z_{\sigma}.
\]
We observe that $\mathcal{R}_{\sigma}$ takes $N_1$ and $N_2$ to two disjoint sets in $\tilde{V}_{\theta}$ distinguished by the fact that the first component belongs  to $\cup_{p\in \Crit(f)\cap f^{-1}(0)} U_p$ (in the case of points in $\mathcal{R}_{\sigma}(N_2)$) or does not belong to the same set (for  $\mathcal{R}_{\sigma}(N_1)$).  We have that $N_2$ is a distinguished boundary of $N'$ while $N_1$ contains the top stratum of $N'$. Hence it is enough to prove the injectivity of $\mathcal{R}_{\sigma}$ separately on $N_1\setminus \partial^2 N_1$ and on $N_2\setminus \partial^1N_2$. 

The points in $N_1\setminus \partial^2 N_1$ are of type $\left(t_{0},\psi_{\frac{t_0}{1-t_0}}(\sigma_1(z)),\sigma_1(z),\sigma_2(z),\sigma_3(z),z\right)$ with $z\in N\setminus \partial^2N$ and they get mapped to $(a,m,b):=\left(\psi_{\frac{t_0}{1-t_0}}(\sigma_1(z)),\sigma_2(z),\sigma_3(z)\right)\in f ^{-1}(\theta_1)\times M\times B_0$. Then $t_0$ is the unique time it takes to flow backwards from point $a\in f ^{-1}(\theta_1)$ to level $f^{-1}(-\theta')$, $\sigma_1(z)$ is the point of intersection with level $f^{-1}(-\theta')$ of the trajectory determined by $a$, and $z\in N\setminus \partial^2N$ is uniquely determined by $\sigma(z)$ due to the hypothesis.

The description $N_2=\displaystyle\bigcup_{{\bf p}\in f^{-1}(0)\cap \Crit(f)}\{1\}\times (U_{\bf{p}}\cap f^{-1}(\theta_1))\times \Gamma_{\sigma\bigr|_{\sigma^{-1}(S_{\bf p}\times M\times B_0)}}$ is useful. For the top stratum $(N_2)^0$ of $N_2$, one restricts $\Gamma_{\sigma}$ to $\sigma^{-1}(S_{\bf p}\times M\times B_0)\cap N^0$. If $p=(u,\sigma_1(z),\sigma_2(z),\sigma_3(z),z)\in (N_2)^0$ then $u$ does not determine  $\sigma_1(z)$ anymore, but property (d) says that $\sigma_1(z)$ is determined by $(\sigma_2(z),\sigma_3(z))$. Since $\sigma$ is also injective on $N^0$ we get that $p$ determines $z$. Hence the map on $(N_2)^0$ that takes $p$ to $(u,\sigma_2(z),\sigma_3(z))$ is injective and this finishes the proof of this issue.

Finally, property (d) itself  holds for $\sigma':N'\ra \tilde{V}_{\theta_1}\cap ({f}^{-1}({\theta_1})\times M\times B_0)$. This follows from the description of $(N')^0=(N_1)^0$ above and the fact that $\sigma$ is injective on $N^0$ and also satisfies property (d).
\end{proof}

\begin{lemma}\label{transvlem} Let $\tilde{\sigma}:\tilde{N}\ra \tilde{M}\times \bR^j$ be a smooth map. Then the (complete) transversality of $\tilde{\sigma}$ with $\tilde{S}\times \bR^j$, for some submanifold $\tilde{S}\subset \tilde{M}$ implies the (complete) transversality of $\tilde{\pi}\circ\sigma$ with $\tilde{S}$ where $\tilde{\pi}:\tilde{M}\times \bR^j\ra \tilde{M}$ is the projection.
\end{lemma}
\begin{proof} Straightforward.
\end{proof}

We now derive an important consequence of Proposition \ref{model.prop} in terms of currents. Let $T_{\sigma}=\mathcal{R}_{\sigma}(\mathcal{A}_{\sigma})$ be the flow-out of $\sigma(N)$ between levels $-\theta'$ and $\theta_1$. It is a rectifiable current, once it is endowed, over the points where $\mathcal{R}_{\sigma}$ is a bijection, with the orientation induced by the direction of the flow and the orientation of $N^0$.

Let $T_{\sigma\bigr|_{\partial^1_j N}}:=\mathcal{R}_{\sigma}(\overline{W_2}\cap \partial^1_{j+1} Z_{\sigma})$ be the flow-out of $\sigma(\partial^1_j N)$ for the $j$-th codimension $1$ boundary of $N$ also between levels $-\theta'$ and $\theta_1$. This is  a rectifiable current of dimension $n$.

\begin{corollary} \label{Impcor}Let $\sigma$ be a smooth map as in Proposition \ref{model.prop}.  The following currential equation holds in the open set $\tilde{V}_{\theta_1}\cap\tilde{f}^{-1}(-\theta',\theta_1)$.
\begin{eqnarray}\quad \label{cordT}dT_{\sigma}=\sum_{{\bf p}\in \Crit(f)\cap f^{-1}(0)} (U_{\bf p}\cap f^{-1}[0,\theta_1))\times (\sigma_2,\sigma_3)(\sigma^{-1}_1(S_{\bf p})) +\sum_{j=1}^{f_N}T_{\sigma\bigr|_{\partial^1_j N}} 
\end{eqnarray}
\end{corollary}
\begin{proof} One uses Lemma \ref{Lem}, Proposition \ref{model.prop} and  Stokes on the manifold with corners $\mathcal{A}_{\sigma}$ pushed-forward via the proper map $\mathcal{R}_{\sigma}$. The injectivity property identifies $(\mathcal{R}_{\sigma})_*(\mathcal{A}_{\sigma})$ with $T_{\sigma}$ and  $(\mathcal{R}_{\sigma})_* (\overline{W_2}\cap \partial^1_{j+1}Z_{\sigma})$ with $T_{\sigma\bigr|_{\partial^1_j N}}$ for all $j$. \end{proof}

\begin{remark} It is important to understand why equation (\ref{cordT}) was not stated as an identity directly in the open set  $\tilde{f}^{-1}(-\theta',\theta_1)$. The conditions in the statement of Proposition \ref{model.prop}  do not exclude the possibility that the image of $\sigma$ oscillates wildly close to the topological boundary of $\tilde{V}_{\theta_1}\cap \tilde{f}^{-1}(-\theta')$ inside $\tilde{f}^{-1}(-\theta')$ so that $\sigma_*(N)$ might not be extendable as a current outside this neighborhood. This is of course not the case for the situation where we will apply Proposition \ref{model.prop}. In the first step of the induction, $\sigma$ is the restriction of $\mathcal{R}$ to $\partial^1_2\mathcal{A}_{\theta}$ and the image of this map "away" from $U_{{\bf p}_1}\times S_{{\bf p}_1}\times B_0$ is simply the flow-out to level $\theta$ of $\xi_0(B_0\setminus B_0')$ where $B_0'\subset B_0$ is a smaller neighborhood around the point of interest $b_0$. Hence close to the topological boundary of $\tilde{V}_{\theta}$, or "away" from $U_{{\bf p}_1}\times S_{{\bf p}_1}\times B_0$ the image of the map is really an embedded submanifold, extendable beyond the topological boundary. For the other steps of the induction we make the following observation.
\end{remark}
\begin{remark}
Proposition \ref{model.prop} is what allows us to cross critical levels at least if the image of the map lands close to the stable manifold(s) of the critical point(s) at level $0$. But when flowing between two consecutive critical levels, nothing guarantees that the flow-out of the image of the resolution at the first critical level will end-up within a neighborhood of type $\tilde{V}_{\theta_1}$ so as to satisfy  the hypothesis of Proposition \ref{model.prop}.

One solution is to do the following. Suppose that in fact $\sigma:N\ra \tilde{f}^{-1}(-\theta_1)$ and for each ${\bf p}\in \Crit(f)\cap f^{-1}(0)$ there are neighborhoods $D_p$ of $(f^{-1}(\theta_1)\cap S_{\bf p})\times M\times B_0$ such that $\sigma\bigr|_{\sigma^{-1}(D_{\bf p})}$ is proper. Then one gets a restriction map $\hat{\sigma}$ of $\sigma$ to an open set of $N$, which is proper and whose image is contained in $\tilde{V}_{\theta_1}$ as in Propostion \ref{model.prop}.  

What about the rest of $N$? Take  $D_{-\theta_1}$ to be \emph{the complement} of $\bigcup_{{\bf p}\in \Crit(f)\cap f^{-1}(0)}(f^{-1}(\theta_1)\cap S_{\bf p})\times M\times B_0$ in $\tilde{f}^{-1}(\theta_1)$ and let $D_N:=\sigma^{-1}(D_{-\theta_1)}$. Look at $\sigma_1:=\sigma\bigr|_{D_N}$. Use the flow diffeomorphism to get from $\sigma_1$ a map $\tilde{\sigma}_1:D_N\ra \tilde{f}^{-1}(\theta_1)$. 

We claim that there exists an open subset of $D_N$ and an open subset of $\partial^1_{\theta_1}\mathcal{A}_{\sigma}$ which are diffeomorphic via a diffeomorphism $\alpha$ such that
\[ \mathcal{R}_{\sigma}\circ\alpha=\tilde{\sigma}_1
\] 
Take $\sigma^{-1}(U)$ where $U:=D_{-\theta_1}\cap \tilde{V}_{\theta_1}$. This is obviously an open set diffeomorphic with a open subset of $D_N$. On the other hand by taking $\overline{W_2}\cap \left(\bR\times f^{-1}(\theta_1)\times \Gamma_{\sigma\bigr|_{\sigma^{-1}(U)}}\right)$ one obtains an open subset of $\partial^1_{\theta_1}\mathcal{A}_{\sigma}$ which is obviously diffeomorphic with $\sigma^{-1}(U)$.

We can then use the diffeomorphism $\alpha$ in order to "glue" $\mathcal{R}_{\sigma}$ and $\sigma_1$ to a smooth map going  from a manifold with corners $\tilde{N}$ to $\tilde{f}^{-1}(\theta_1)$ and flow to the next critical level and apply again Proposition \ref{model.prop} and this Remark and so on. 
\end{remark}

\begin{proof} \emph{of Theorem} \ref{tlimxi}.

  We have already sketched the proof strategy  at the end of Section \ref{Sec3}. We need only explain what are the open sets $M_j\times M\times B_0$ that cover $M\times M\times B_0$ where the currential identity is true. We will discuss only the situations where $B_0$ is a small neighborhood around a point $b_0$ such that the forward trajectory determined by $s(b_0)$ ends at a non-maximal point. The remaining situation was already discussed in Remark \ref{Fmax}.
  
  Let $c_0<\ldots <c_l$ be the consecutive critical levels of $f$  excluding the maximum with $c_0$ the first encountered critical level by the forward trajectory of $s(b_0)$. It might even be a minimal level of $f$, i.e. a level that contains a local minimum if $s(b_0)$ is a local minimum.  Let $c_0<\delta_1<\delta_2<c_1<\delta_3<\delta_4<\ldots <c_{l-1}<\delta_{2l-1}<\delta_{2l}<c_l< \delta_{2l+1}$ be regular level sets such that $\delta_1=c_0+\theta_0$ and for $k\geq 1$,  $\delta_{2k}=c_{k}-\theta_k$ and $\delta_{2k+1}=c_k+\theta_k$ where $\theta_k$ is chosen small enough so that we can chose neighborhood $V_{\theta_k}$ around the critical points of level $c_k$ satisfying the conditions of Proposition \ref{prop.triplo}.
  
  Corollary \ref{coreqfund} takes care of the first step of induction and implies the formula for 
  \[{f}^{-1}(-\infty,\delta_1)\times M\times B_0.\] By property (4) of Proposition \ref{prop.1} and Lemma \ref{Lem} there exists a map  for some $\epsilon_1>0$:
  \[\sigma:\partial^1_2\mathcal{A}_1\ra {f}^{-1}(\delta_1-\epsilon_1)\times M\times B_0\]
whose image contains the closure of the forward-flow of $\xi_0(B_0)$ intersected with level $\delta_1-\epsilon_1$ and is transverse to all $S_{\bf {p}}\times M\times B_0$ for all critical ${\bf p}$. Moreover $\partial^1_2\mathcal{A}_1$ is a manifold with boundary $\partial^2\mathcal{A}_1=U_{{\bf p}_1}\cap {f}^{-1}(\delta_1-\epsilon_1)\times s(s^{-1}(S_{{\bf p}_1})\times B_0)$ and $\sigma$ satisfies the conditions of Proposition \ref{model.prop}.

If we let $T^1_{\sigma}$ to be the current determined by the flow-out of the image of $\sigma$ in the open set ${f}^{-1}(\delta_1-\epsilon_1,\delta_2+\epsilon_2)\times M\times B_0$ for some small $\epsilon_2$ where there are no critical points. Let $T^1_{\partial \sigma}$ be the flow-out of the image of $\sigma\bigr|_{\partial^2\mathcal{A}}$. Both are rectifiable currents with the obvious orientation. The following  identity of currents holds on ${f}^{-1}(\delta_1-\epsilon_1,\delta_2+\epsilon_2)\times M\times B_0$:
\[ dT^1_{\sigma}=T_{\partial \sigma}^1=(U_{{\bf p}_1}\cap {f}^{-1}(\delta_1-\epsilon_1,\delta_2+\epsilon_2))\times s(s^{-1}(S_{{\bf p}_1})\times B_0)
\]
proving thus the theorem on ${f}^{-1}(\delta_1-\epsilon_1,\delta_2+\epsilon_2)\times M\times B_0$.

From this point on we repeatedly apply Proposition \ref{model.prop} and Corollary \ref{Impcor}. We notice that $(\sigma_2,\sigma_3)(\sigma_1^{-1}(S_p))$ can be substituted with $(\sigma_2,\sigma_3)(\sigma_1^{-1}(S_p)\cap N^0)$ and these points are easy to describe as $s(s^{-1}(S_p\times B_0))$ where $p$ is any of the critical points at $i$-th step. It is not any difficult to see that $T_{\sigma\bigr|_{\partial^1_j N}}$ is a sum:
\[  \sum_{p} (U_p\cap f^{-1}(c_i-\theta_i,c_i+\theta_i))\times s(s^{-1}(S_p\times B_0))
\]
where the sum here runs over the critical points that have already been crossed.
\end{proof}
\begin{remark}\label{FinHm} There is a subtlety in the proof of Theorem \ref{tlimxi}, that is easy to miss. In order to prove an equality of currents, these have to exist to begin with. In particular $T$ and $U(F)\times_Fs(s^{-1}(S(F)))$ should be shown to have finite local $n+1$ and respectively $n$-Hausdorff measures. But this is a straightforward corollary of the existence of the flow resolutions we have constructed. As an immediate consequence we get via Fubini that $U(F)$ and $s(s^{-1}(S(F)))$ have finite Hausdorff measures if $B$ is compact.
\end{remark}

\section{Odd Chern-Weil theory}\label{OCW}
We now start a new topic altogether.  Let $E\ra B$ be a hermitian vector bundle over a compact manifold $B$.  Denote by $\mathscr{U}(E)$ the fiber bundle of unitary isomorphisms and let $U\in\Gamma(\mathscr{U}(E))$  be a section of this bundle. Let $\nabla$ be a connection compatible with the metric.  

Given any invariant polynomial $P$, we introduce  odd degree forms $\TP(E,U,\nabla)\in \Omega^*(B)$, called odd Chern-Weil forms, which satisfy the following properties
\begin{itemize}
\item[(a)] $d\TP(E,U,\nabla)=0$;
\item[(b)] $\TP(E,U,\nabla)-\TP(E,U,\nabla')$ is exact for any two metric compatible connections $\nabla,\nabla'$;
\item[(c)] if $U_0,U_1\in \Gamma(\mathscr{U}(E))$ are homotopic then $\TP(E,U_0,\nabla)-\TP(E,U_1,\nabla)$ is exact. 
\item[(d)] if $\varphi:B_1\ra B$ is a smooth map then $\varphi^*\TP(E,U,\nabla)=\TP(\varphi^*E,\varphi^*U,\varphi^*\nabla)$.
\end{itemize}

Recall the fundamental Theorem of Chern-Weil theory. If $P$ is an invariant polynomial and $\nabla_1$ and $\nabla_2$ are compatible connections then one can associate to $\nabla_0$ and $\nabla_1$ two forms $P(E,\nabla_1)$ and $P(E,\nabla_0)$ and  there exists a non-unique form $\TP(\nabla_1,\nabla_0)$ such that
\[P(E,\nabla_1)-P(E,\nabla_0)=d\TP(\nabla_0,\nabla_1).
\]
 The construction of a particular such  form $\TP(\nabla_1,\nabla_0)$ goes as follows. Let $\pi_2^*E\ra [0,1]\times B$ be the pull-back of $E$ with respect to the projection $[0,1]\times B\ra B$. Consider the following connection on $\pi_2^*E$
\[ \tilde{\nabla}:=\frac{d}{dt}+(1-t)\nabla_0+t\nabla_1.
\]
Then the standard homotopy formula implies that
\[ d\left(\int_{[0,1]}P(\pi_2^*E,\tilde{\nabla})\right)=P(E,\nabla_1)-P(E,\nabla_0).
\]
where integration on the left is over the fibers of $\pi_2$. Define $\TP(\nabla_0,\nabla_1)$ to be $\int_{[0,1]}P(\pi_2^*E,\tilde{\nabla})$.

It is not hard to see that
\[\TP(\nabla_0,\nabla_1)=-\TP(\nabla_1,\nabla_0).
\]
This is because the diffeomorphism on $[0,1]\times M\ra  [0,1]\times M$, $(t,m)\ra (1-t,m)$ reverses the orientation of the fiber and fiber integration is sensitive to this. 

\begin{remark} If one takes different paths between the connection $\nabla_0$ and $\nabla_1$, by a result of Simons and Sullivan \cite{SS} the transgression forms obtained for two different paths differ by an exact form.
\end{remark}

In our context, consider the connections $\nabla_0:=\nabla$ and $\nabla_1=U^{-1}\nabla U$ on $E$ and write down the transgression formula from Chern-Weil theory:
\[P(F(U^{-1}\nabla U))- P(F(\nabla))=d\TP(\nabla,U^{-1}\nabla U),
\]

Notice that $F(U^{-1}\nabla U)=U^{-1}F(\nabla)U$. It follows from the fact that $P$ is invariant that $P(F(\nabla))=P(F(U^{-1}\nabla U))$, hence the form $\TP(\nabla,U^{-1}\nabla U)$
 satisfies  property (a) above.

The next lemma helps prove property (b) for $\TP(\nabla,U^{-1}\nabla U)$.
 \begin{lemma} If $\nabla_i$, $i=1,4$ and  are four metric compatible connections then
 \[ \sum_{i}\TP(\nabla_i,\nabla_{i+1})
 \]
 is exact where $\nabla_5:=\nabla_1$.
 \end{lemma}
 \begin{proof} If $H:C\times B\ra M$ is a smooth map,  $C$ is an oriented, compact manifold with corners of dimension $c$ and $\omega$ is a smooth form on $M$ of degree $k\geq c-1$ then:
 \begin{equation}\label{chom} \int_{C}H^*d\omega+(-1)^{c-1}d\int_{C}H^*\omega=\int_{\partial C}H^*\omega.
 \end{equation}
To see (\ref{chom}), apply first Stokes on $C\times B$ to 
 \[d(H^*\omega\wedge \pi_2^*\eta)=dH^*\omega\wedge \pi_2^*\eta+(-1)^{k}H^*\omega\wedge\pi_2^*d\eta\] where $\eta\in\Omega^{n-k+c-1}(B)$ is a smooth test form and $\pi_2:C\times B\ra B$ is the projection and then integrate over the fiber.  On the closed, oriented $B$ one has:
 \begin{equation}\label{StB}0=d\left(\int_CH^*\omega\right)\wedge \eta+ (-1)^{k-c}\left(\int_CH^*\omega\right)\wedge d\eta
 \end{equation}
 Hence, using the orientation of the fiber first convention we get from (\ref{StB})
 \[(-1)^{k}\int_CH^*\omega\wedge\pi_2^*d\eta=(-1)^{k}\left(\int_CH^*\omega\right)\wedge d\eta=(-1)^{c-1}d\left(\int_CH^*\omega\right)\wedge \eta.
 \]
 Take $C=\Delta^2$ be the standard simplex in $\bR^2$ with coordinates $(s,t)$ and on $\pi_2^*E\ra C\times B$ consider the connection that "interpolates" between $\nabla_1$, $\nabla_2$, $\nabla_3$:
 \[ \tilde{\nabla}:=\frac{d}{ds}+\frac{d}{dt} +\nabla_1+s(\nabla_2-\nabla_1)+t(\nabla_3-\nabla_1)
 \]
 where $\frac{d}{ds}+\frac{d}{dt}$ is the differential on $C$. The form $P(\pi_2^*E,\tilde{\nabla})$ on $C\times B$ is closed and 
 \[\int_{\partial C}P(\pi_2^*E,\tilde{\nabla})=TP(\nabla_1,\nabla_2)+\TP(\nabla_2,\nabla_3)+\TP(\nabla_3,\nabla_1)
 \]
 Use now (\ref{chom}) for $H=\id_{C\times B}$ to conclude that the Lemma works for three connections. Using that $\TP(\nabla,\nabla')=-\TP(\nabla',\nabla)$ one can extend by induction to any finite number of connections.
  \end{proof}

 One applies the lemma with $\nabla_1:=\nabla$, $\nabla_2:=U^{-1}\nabla U$, $\nabla_3=U^{-1}\nabla' U$ and $\nabla_4:=\nabla'$ in order to conclude property (b). Indeed one has
 \[\TP(\nabla_2,\nabla_3)=\TP(\nabla,\nabla')=-\TP(\nabla_4,\nabla_1)
 \]
 
 Property (c) for $\TP(\nabla,U^{-1}\nabla U)$ is proved as follows.  Consider the vector bundle $\pi^*E\ra \mathscr{U}(E)$ where $\pi:\mathscr{U}(E)\ra B$ is the natural projection.  Then $\pi^*E$ has a  connection $\pi^*\nabla$ and it also has a tautological unitary isomorphism $U^{\tau}:\mathscr{U}(E)\ra \mathscr{U}(\pi^*E)$. Therefore there exists a natural transgression closed form $\TP(\pi^*\nabla,(U^{\tau})^{-1}(\pi^*\nabla) U^{\tau})\in\Omega^*(\mathscr{U}(E))$. 
 
 Let $U_t$ be a smooth homotopy between $U_0$ and $U_1$. Then $U_t$ is a section of 
 \[\mathscr{U}(\pi_2^*E)\ra [0,1]\times B.\] 
 Then for any closed form $\omega$ on $\mathscr{U}(\pi_2^*E)$ the standard homotopy formula informs that $U_0^*\omega$ and $U_1^*\omega$ differ by an exact form on $B$.
 Take $\omega:= p^*\TP(\pi^*\nabla,(U^{\tau})^{-1}\pi^*\nabla U^{\tau})$ where  $p:\pi_2^*\mathscr{U}(E)\ra \mathscr{U}(E)$ is the natural projection. It is not hard to check that
 \begin{eqnarray} U_0^*\omega=U_0^*\TP(\pi^*\nabla,(U^{\tau})^{-1}\pi^*\nabla U^{\tau})=\TP(\nabla,U^{-1}_0\nabla U_0)\quad \mbox{and} \\ U_1^*\omega=\TP(\nabla,U^{-1}_1\nabla U_1)\qquad
 \end{eqnarray}
 and this finishes the proof of the third property.
 
One checks rather immediately the naturality  of $\TP$.

Hence $\TP(\nabla,U^{-1}\nabla U)$ satisfies the four properties above.  Define then the odd Chern-Weil forms associated to $(P,E,U,\nabla)$ by
  \[\TP(E,U,\nabla):=\TP(\nabla,U^{-1}\nabla U).\]
We will also denote this by $\TP(U,\nabla)$ when the bundle is clear from the context.
\begin{remark} \label{clutQ}
There is an alternative way  of thinking about $\TP(E,U,\nabla)$ that reminds one of the clutching construction. Consider the fiber bundle $\pi_2^*E\ra \bR\times B$ where now $\pi_2:\bR\times B\ra B$.

Then $\bZ$ acts on $\pi^*E\rightarrow \bR\times B$ as follows:
 \[k*(t,b,v)=(t-k,b,U_b^kv).\]
  We get thus a vector bundle $\tilde{E}=T(E,U)=\pi_2^*E/\bZ$. A  smooth section of $\tilde{E}$ is a smooth family of sections $(s_t)_{t\in\bR}\in\Gamma(E)$ satisfying:
  \[s_{t-k}(b)=U_b^ks_{t}(b),\quad\quad \forall b\in B,\; k\in\bZ,\; t\in \bR
  \]
  Suppose $\nabla_t$ is a smooth family of connections on $E\ra B$ satisfying:
  \begin{equation} \label{eqUnab} \nabla_{t+k}=U^{-k}\nabla_t U^k,
  \end{equation}
  Then the connection $T(\nabla_t)=\frac{d}{dt}+\nabla_t$ on $\pi_2^*E$ "descends" to a connection on $\tilde{E}$ as the next computation shows:
  \[ (T(\nabla_t)s)_{t-k}(b)=\frac{\partial s}{\partial t}(t-k,b)~dt+\nabla_{t-k}s_{t-k}(b)=\]\[=U_b^k\frac{\partial s}{\partial t}(t,b)~dt+ U_b^k\nabla_t (U_b^{-k}U_b^ks_t)(b)=U_b^k(T(\nabla_t)s)_t)(b).
  \]

Now, the affine family $(1-t)\nabla+tU^{-1}\nabla U$ defined for $t\in[0,1]$ satisfies (\ref{eqUnab}) for $k=1$ and $t=0$. Unfortunately, one cannot extend it to a \emph{smooth} family satisfying (\ref{eqUnab}) so one cannot say that $\tilde{\nabla}$ is a smooth connection on $\tilde{E}$.  

  However, one can show that $\TP(E,U,\nabla)$ represents the cohomology class
  \[ \int_{S^1}P(\tilde{E})
  \]
  where $P(\tilde{E})$ is the deRham cohomology class of the vector bundle $\tilde{E}\ra S^1\times B$. This is because the transgressions between $\nabla$ and $U^{-1}\nabla U$ determined by the affine path and by a path satisfying (\ref{eqUnab}) differ by an exact form. If one starts with a connection on $\tilde{E}$ of type $T(\nabla_t)$, the integral over $S^1$ of $P(F(T(\nabla_t)))$, is really the integral over $[0,1]$ of the same quantity and is just the transgression between $\nabla$ and $U^{-1}\nabla U$ given by the path of connections $\nabla_t$. This justifies the claim.
  
  All vector bundles over $S^1\times B$ when $B$ is compact  arise up to isomorphism via the clutching construction. This is because the  pull-back of such a bundle to $[0,1]\times B$ is isomorphic with the pull-back of a vector bundle from $B$. The isomorphism of the fiber at $0$ with the fiber at $1$ gives the desired gauge transformation.
\end{remark}

  \begin{example} Consider the trivial vector bundle $\underline{\bC^n}\ra U(n)$. It has a tautological gauge transformation $\widetilde{U}(U):=U$. Let $d$ be the trivial connection.  Then the difference between the $\widetilde U^{-1}d \widetilde{U}$ and  $d$ is just the Maurer-Cartan $1$-form of $U(n)$  usually denoted $g^{-1}(dg)$. Then the  transgression forms $\Tc_k(\tilde{U},d)$ corresponding to the elementary symmetric polynomials $c_k$ in the eigenvalues of a matrix, or if you want to the standard  Chern classes are constant multiples of the forms
  \[ \tr \wedge ^{2k-1}g^{-1}(dg).
  \] 
 We determine these constants now. Let $\omega:=g^{-1}(dg)$. Consider the family of connections over $\underline{\bC^n}$:
 \[ d+tg^{-1}(dg)=d+t\omega
 \]
 Then the induced connection on $\underline{\bC^n}\ra [0,1]\times B$ is $\tilde{\nabla}=d+t\pi_2^*\omega$. Using the Maurer-Cartan identity $d\omega+\omega\wedge\omega=0$ we get that
 \[ F(\tilde{\nabla})=dt\wedge \pi_2^*\omega+(t^2-t)\pi_2^*\omega\wedge \pi_2^*\omega.
 \]
 Then, by definition, the component of degree $k$ of the Chern character is 
 \[\ch_k(\tilde{\nabla})=\left(\frac{i}{2\pi}\right)^{k}\frac{1}{k!}\tr (F(\tilde{\nabla})^k).\]
 But letting $A=dt\wedge \pi_2^*\omega$ and $B=(t^2-t)\omega\wedge \omega$ we see that $A^2=0$ and  $AB=BA$ and so $(A+B)^k= B^k+kB^{k-1}A$. We therefore get
 \[\ch_k(\tilde{\nabla})=\left(\frac{i}{2\pi}\right)^{k}\frac{1}{(k-1)!} (t^2-t)^{k-1}dt\wedge \tr (\wedge^{2k-1}\pi_2^*\omega)+\left(\frac{i}{2\pi}\right)^{k}\frac{1}{k!}\tr(\wedge^{2k} \pi_2^*\omega)
 \]
The last term vanishes before integration because $\wedge^{2k} \pi_2^*\omega=\frac{1}{2}[ \pi_2^*\omega, \pi_2^*\omega]$ and trace vanishes on commutators. Hence 
\begin{equation}\label{chk}\ch_k(\tilde{\nabla})=\left(\frac{i}{2\pi}\right)^{k}\frac{1}{(k-1)!} (t^2-t)^{k-1}dt\wedge \tr (\wedge^{2k-1}\pi_2^*\omega).
\end{equation}
 
 On the other hand, $\int_{[0,1]} (t^2-t)^{k-1}dt=(-1)^{k-1}B(k,k)=(-1)^{k-1}\frac{[(k-1)!]^2}{(2k-1)!}$.
 
 We conclude that
 \begin{equation}\label{chkint} \Tch_k(\tilde{U},d)= \int_{[0,1]}\ch_k(\tilde{\nabla})=(-1)^{k-1}\left(\frac{i}{2\pi}\right)^k\frac{(k-1)!}{(2k-1)!}\tr(\wedge^{2k-1}g^{-1}(dg))
 \end{equation}
We emphasize that this is the same form as $(-2\pi i)^{-k+1/2}\gamma_{2k-1}$ where $\gamma_{2k-1}$ appears in Definition 5.1 of  \cite{Qu}. In fact, Quillen  shows in his Proposition 5.23, by a rather long argument, that the form $(-2\pi i)^{-k+1/2}\gamma_{2k-1}$ represents the integral of $\ch_k(\tilde{E})$ over $S^1$,  where $\tilde{E}$ is the vector bundle on $S^1\times U(n)$ constructed from the trivial bundle via the clutching construction with respect to the tautological gauge transform. Remark \ref{clutQ} gives a more straightforward justification of this fact.
 
 Now, the symmetric polynomials $\ch_k$ and $c_k$ are related via the Newton identities. One has an identity of type:
 \[c_k=(-1)^{k-1}(k-1)! \ch_k+ R\]
 where $R$ stands for a sum of products of $\ch_{i}$ with $i<k$. A quick glance at (\ref{chk}) convinces us that $R(\tilde{\nabla})=0$. Hence after integration over $[0,1]$ we get that
 \begin{equation} \label{tck1} \Tc_k(\tilde{U},d)=\left(\frac{i}{2\pi}\right)^k\frac{[(k-1)!]^{2}}{(2k-1)!}\tr(\wedge^{2k-1}g^{-1}(dg)).
 \end{equation}
 For $k=1$ one gets $Tc_1(\tilde{U},d)=-\frac{1}{2\pi i}\tr (g^{-1}(dg))$.
 
Let us end this example by giving a simple application. 
  \end{example}
  \begin{lemma}\label{eqikn} Let $k\leq n$ and $\iota_{k-1,n}:U(k-1)\ra U(n)$ be the natural inclusion. Then $\iota_{k-1,n}^*\Tc_l(\tilde{U},d)$ is exact for $l\geq k$.
  \end{lemma}
  \begin{proof} It is enough to prove the Lemma for $l=k$, and derive the general property from the inclusions $U(k-1)\hookrightarrow U(l-1)\hookrightarrow U(n)$. We look at the "clutching" bundle determined by $\bC^{n}=\bC^{k-1}\times \bC^{n-k+1}$ over $S^1\times U(k-1)$ where $U(k-1)$ acts trivially on $\bC^{n-{k+1}}$. Denote this bundle by $\tilde{E}_{k,n}$. Clearly $\tilde{E}_{k,n}$ will have $n-k+1$ linearly independent sections, namely the vectors $e_k,\ldots,e_n$ of the canonical basis get glued to themselves and hence determine an trivial rank $n-k+1$ subundle of $\tilde{E}_{k,n}$. It follows that $c_k(\tilde{E}_{k,n})=0$ in cohomology and hence also $\int_{S^1}c_k(\tilde{E}_{k,n})=0$. By Remark \ref{clutQ} and naturality of the odd Chern-Weil forms, this is the same as the class of $\iota_{k-1,n}^*\Tc_k(\tilde{U},d)$. 
      \end{proof}

\begin{remark} The odd Chern-Weil theory associates to a gauge transformation  certain odd cohomology classes. For simplicity we presented it here for structure group $U(n)$,  but it can be easily adapted to principal bundles $P\ra B$ endowed with a gauge transform $g$ and principal connection $1$-form $\omega$. Then one gets a new principal connection $g^{-1}\omega g$ and integrating over $[0,1]$ the invariant polynomial in the curvature entries of $(1-t)\omega+tg^{-1}\omega g$ gives a transgression closed form on the \emph{base space} that satisfies  analogous properties as $\TP(U,\nabla)$.

Recall that Chern-Simons theory  constructs typically \emph{non-closed} odd forms living in the \emph{total space} of the principal bundle as follows. Given a principal bundle $\pi:P\ra B$ with structure group $G$, connection $1$-form $\omega$ and an invariant polynomial $Q$. Then $\pi^*P=P\times_BP\ra P$ is a trivial principal bundle with the trivializing section given by the diagonal embedding. It is endowed with two connections: $\pi^*\omega$ and the pull-back of the Maurer-Cartan connection $1$-form $\pi_2^*g^{-1}(dg)$ via the projection $\pi_2:P\ra G$ induced by the trivialization. Using the affine family of connections between $\pi_2^*g^{-1}(dg)$ and $\pi^*\omega$   one gets the Chern-Simons form $TQ(\omega)$ on $P$. Since the curvature of $g^{-1}(dg)$ is zero it follows that $d TQ(\omega)=\pi^*Q(\omega)$. The form $TQ(\omega)$ "descends" to a form on $B$ only under very special circumstances. 

The odd Chern-Weil theory on the other hand can be constructed in the same spirit, only that one is using the bundle of \emph{gauge transformations} which is the bundle associated to $P\ra B$ via the adjoint action of $G$ onto itself.
\end{remark}

\section{The Chern classes of a gauge transform}

In this section we show how the main transgression identity from Corollary \ref{c.principal} can be applied to simplify the proof and  give an extension to a result of Nicolaescu in \cite{Ni}. This application was previously announced in a pre-print of the first named author and posted on arxiv, using  the \emph{tame} version of the Corollary \ref{c.principal}. However the flow used  did not satisfy the condition of tameness. We revisit this application under a renewed framework. The generalization has to do with the use of the odd Chern-Weil forms introduced in the previous section.

The starting question is  the following. Let  $E\ra B$ be a hermitian vector bundle of rank $n$ endowed with a gauge transform $U\in \Gamma(\mathcal{U}(E))$ and a compatible connection $\nabla$. 

Give a description of the Poincar\'e dual of $\Tc_k(U,\nabla)$, one in terms of pointwise spectral data of $U$. 

In order to achieve this purpose will use a horizontally constant, vertical Morse-Smale vector field on  fiber bundle over $B$ with total space $\mathscr{U}(E)$.

 Let $\pi:\mathscr{U}(E)\ra B$ be the projection and $\pi^*\mathscr{U}(E)=\mathscr U(E)\times_B\mathscr{U}(E)\ra \mathscr{U}(E)$ be the pull-back. It has a tautological section $U^{\tau}$ (the diagonal embedding of $\mathscr{U}(E)$ in $\pi^*\mathscr{U}(E)$) which is obviously a gauge transform of $\pi^*E$. With the help of the connection $\pi^*\nabla$ we can construct $\Tc_k(U^{\tau},\pi^*\nabla)\in\Omega^*(\mathscr{U}(E))$ and the naturality of $\Tc_k$ shows that
\[ U^*\Tc_k(U^{\tau},\pi^*\nabla)=\Tc_k(U,\nabla).
\]

The form $\omega\in \Omega^*(\mathscr{U}(E))$ to be "flown" will be $\omega=\Tc_k(U^{\tau},\pi^*\nabla)$.  What is the flow then? In order to define it we need to fix a complete flag of subbundles 
\[E=W_0\supset \ldots\supset W_n=\{0\}.\]
 For $1\leq i\leq n$  let $E_i:=W_{i-1}/W_{i}$ (better said the orthogonal of $W_i$ in $W_{i-1}$) and 
 \[A:=\bigoplus_{1\leq k\leq n}k\id_{E_i}\] be a self-adjoint endomorphism of $E$ with distinct eigenvalues. Then the function 
 \[ f:\mathscr{U}(E)\ra \bR,\qquad f(U)=\Real\Tr (AU)
 \]
 has a fiberwise or vertical gradient which can be described as $\grad^vf(U)=A-UAU$. The restriction to each fiber $\mathscr{U}(E_b)$ is a Morse-Smale function  on the unitary group, thoroughly explored in \cite{Ni}\footnote{The analysis in \cite{Ni} is on the Grassmannian of hermitian Lagrangians but, by a Theorem of Arnold a clever way of writing the Cayley transform makes this space  diffeomorphic to the unitary group.}  We recall the essential properties. The critical points are in one-to-one correspondence with the invariant subspaces of $A$. More precisely, let $\langle e_1,\ldots, e_n\rangle$ be  a basis of $E_b$ such that $(W_k)_b^{\perp}=\langle e_1\ldots, e_k\rangle$. The critical points are reflections $U_{I}:=-\id_{V}\oplus \id_{V^{\perp}}$ where $V=\langle e_{i_1},\ldots,e_{i_k}\rangle$ for some ordered set $I=\{i_1,\ldots,i_k\}\subset\{1,\ldots,n\}$. There are $2^n$ such critical points with the absolute minimum (of $f\bigr|_{E_b}$) corresponding to $V=E_b$ and the absolute maximum to $V=\{0\}$.
 
The flow is given by the expression \cite{DV}
\[ (t,U)\ra (\sinh{(tA)}+\cosh{(tA)}U)(\cosh{(tA)}+\sinh{(tA)}U)^{-1}
\] 
From this we deduce that the stable and unstable manifolds can be described by the following incidence relations (compare with Corollary 16 in \cite{Ni}):
\begin{eqnarray} \label{DefS} \qquad \;\;S(U_{I})=\{U\in\mathcal{U}(E_b)~|~\dim[{\Ker{(1+U)}\cap W_m}]=k-p,\;\;\forall 0\leq p\leq k,\;\qquad\quad\\\nonumber \forall i_p\leq m< i_{p+1}\}\\
\nonumber U(U_{I})=\{U\in\mathcal{U}(E_b)~|~\dim[{\Ker{(1-U)}\cap W_m}]=n-k-q,\;\;\forall 0\leq q\leq n-k,\;\\ \nonumber\forall j_q\leq m< j_{q+1}\}
\end{eqnarray}
where we set $i_0:=0$, $i_{k+1}:=\infty$ and $\{j_1<\ldots <j_{n-k}\}=I^c$ is the complement of $I$. We note that $\dim{U(U_I)}=\codim{S(U_I)}=\sum_{i\in I}2i-1$.\footnote{We ignore the indication of the point $b$ in the flag so as not to complicate notation. It should be clear from the context whether we refer to the fiber component or the entire fiber bundle.}
\begin{remark} \label{Pseudo} It is shown in \cite{Ni} (Proposition 17 and Corollary 18) that the flow satisfies the Smale property. Moreover, their closures $\overline{U(U_I)}$ and $\overline{S(U_I)}$ are real algebraic sets. One can show that in fact $\overline{S(U_I)}$ has a stratification with no codimension $1$ strata (see the comments after Corollary 5.1 in  \cite{Ci1}) and by reversing the flow (or using the involution $U\ra -U$) the same is true about $\overline{U(U_I)}$. In other words, $\overline{S(U_I)}$ and $\overline{U(U_I)}$ are pseudo-manifolds and $S(U_I)$ and $U(U_I)$ determine \emph{closed} currents once an orientation is chosen.  Their volume is finite by the results of \cite{HL1} (see  Remark \ref{FinHm}). This also implies in particular that the Morse-Witten complex associated to the Morse-Smale flow induced by  $f$ is perfect, i.e. all the differentials are zero. This is not the case for the analogous  flow on the real orthogonal group.
\end{remark}

We will need the following two computational results interesting in their own right. Recall the forms $\Tc_k (\tilde{U},d)$ from (\ref{tck1}). We use the same notation for the analogous forms on $\mathcal{U}(E_b)$. In order to keep the notation simple, for this part, we will forget about the point $b$ and use $E$ and $W$ for $E_b$ and $W_b$, etc. 
\begin{lemma}\label{UI} Let $I\neq \{k\}$. Then
\[ \int_{U(U_I)}\Tc_k (\tilde{U},d)=0
\]
\end{lemma}
\begin{proof} Clearly this is true by definition if $\dim{U(U_I)}\neq \dim{U(U_{\{k\}})}=2k-1$. If $\dim{U(U_I)}=2k-1$ but $U_I\neq U_{\{k\}}$ then we infer that $\iota:=\max{\{i\in I\}}\leq k-1$. We claim that this implies that
\[ \Ker(1-U)\supset W_{k-1},\qquad \forall U\in U(U_{I}).
\]
Indeed, let  $m\geq \iota=\max{I}$. We first estimate $p$ which satisfies $j_p \leq m< j_p+1$,
where $j_p\in I^c$. Let $l:=|I|$. Then, for some $s\geq 0$, we have 
\begin{equation}\label{miota} m=\iota+s<j_{\iota-l+s+1}.
\end{equation} To see this more clearly consider first $s>0$  and notice that in fact $\iota + s = j_{\iota-l+s}$ as there are exactly $\iota-l+s$  
of $j$'s in $I^c$ which are smaller or equal than $\iota+s$, which itself lies in $I^c$. For the case $s=0$ one still has $m=\iota< j_{\iota-l+1}$ since there are exactly $\iota- l$ numbers smaller or equal $\iota$ which are not in $I$.

Then (\ref{miota}) implies that $p\leq \iota-l+s=m-l$. Therefore for $U\in U(U_{I})$ one has:
\[\dim{[\Ker{(1-U)}\cap W_{m}]}=n-l-p\geq n-m=\dim{W_m}
\]
Hence $U\bigr|_{W_m}=\id_{W_m}$ for all $m\geq \max{I}$ and this applies to $m=k-1$. It follows then that the (proper) inclusion map:
\[ \iota_{k-1,n}:\mathcal{U}(W_{k-1}^{\perp})\ra \mathcal{U}(E), \qquad U\ra U\oplus \id_{W_{k-1}}
\]
takes $U(U_{I})\subset \mathcal{U}(W_{k-1}^{\perp})$  diffeomorphically to  $U(U_{I})\subset \mathcal{U}(E)$.\footnote{Of course we abused notation by not making any difference between the unstable manifold corresponding to $U_{I}$ in $\mathcal{U}(E)$ and in $\mathcal{U}(W_{k-1}^{\perp})$, respectively.} Now $U(U_I)$ is a closed current in $ \mathcal{U}(W_{k-1}^{\perp})$ irrespective of the orientation. Then Lemma \ref{eqikn} and Remark \ref{Pseudo} finish the proof.
\end{proof}
In order to compute $\Tc_k(\tilde{U},d)$ over $U(U_{\{k\}})$ we need to fix an orientation. Notice first that $U(U_{\{k\}})=\iota_{k,n}(U(U_{\{k\}}))$, where the latter lies within $\mathcal{U}(W_k^{\perp})$ by the same type of argument that was used in Lemma \ref{UI} for $U(U_{I})$. Moreover by the naturality of $\Tc_k(\tilde{U},d)$ one has $\iota_{k,n}^*\Tc_k(\tilde{U},d)=\Tc_k(\tilde{U},d)$. 

It is only natural then to work on $\mathcal{U}(W^{\perp}_k)$. Notice that we have a natural flag on $W^{\perp}_k$ defined by $W_i':=W_i/W_k$, $0\leq i\leq k$. Then
\[ U(U_{\{k\}})=\{U\in\mathcal{U}(W^{\perp}_k)~|~\dim[\Ker{(1-U)\cap W_m}]=k-1-m,\;\; \forall 0\leq m\leq k-1 \}
\]
In fact $U(U_{\{k\}})$ is an open  dense subset of the following manifold defined by a single incidence relation:
\[ U(U_{\kappa}):=\{U\in\mathcal{U}(W^{\perp}_k)~|~\dim \Ker{(1-U)}=k-1 \}
\]
This is because, generically a hyperplane of $W^{\perp}_k$ (like  $\Ker{(1-U)}$ for $U\in U(U_{\kappa})$ will intersect $W_m'$ in dimension $k-1-m$ for $m\geq 1$. Now
\[\overline{U(U_{\kappa})}=\{U\in\mathcal{U}(W^{\perp}_k)~|~\dim \Ker{(1-U)}\geq k-1 \}=U(U_{\kappa})\cup \{\id_{W_k^{\perp}}\}.
\]

 We use the following map
\[ \phi:S^1\times \bP(W_k^{\perp})\ra \mathcal{U}(W_k^{\perp}),\qquad (\lambda,L)\ra \lambda \id_{L}\oplus \id_{L^{\perp}}
\]
and note that for $\{\lambda\neq 1\}$, this map is a smooth bijection onto ${U(U_{\kappa})}$ while $\{\lambda =1\}$ it collapses $\bP(W_k^{\perp})$ to $\id_{W_k^{\perp}}$. \begin{remark} Notice that the map $\phi$ induces an homeomorphism between $\Sigma\bC\bP^{k-1}$ (the Thom space of the trivial real line bundle over $\bC\bP^{k-1}$) and $\overline{U(U_{\kappa})}$.
\end{remark}
Now $S^1\times \bP(W_k^{\perp})$ has a canonical orientation.  We put the orientation on $U(U_{\kappa})$  (implicitly also on $U(U_{\{k\}})$) that makes $\phi$ \emph{orientation reversing}. The reason is the next result.
\begin{lemma} \[\int_{U(U_{\{k\}})}\Tc_k(\tilde{U},d)=\int_{\overline{U(U_{\kappa})}}\Tc_k(\tilde{U},d)=1\]
\end{lemma}
\begin{proof} Fix $L\in \bP(W_k^{\perp})$. We write $\phi$ in the "chart"  $S^1\times \Hom(L,L^{\perp})$:
\[\phi(\lambda,A)=\left(\begin{array}{cc}\frac{\lambda+A^*A}{1+A^*A} & (\lambda-1)(1+A^*A)^{-1}A^*\\
 (\lambda-1)(1+AA^*)^{-1}A&\frac{1+\lambda AA*}{1+AA^*}\end{array}\right)
\]
where the decomposition of $\phi(\lambda,A)$ on the right is relative $L\oplus L^{\perp}$.
The differential at the point $(\lambda,0)$ in this chart is:
\[d\phi_{\lambda,L}(w,S)=\left(\begin{array}{cc} w& (\lambda-1)S^*\\
 (\lambda-1)S&0\end{array} \right)
\]
Hence
\[\phi^{-1}(\lambda,L)d\phi_{\lambda,L}(w,S)=\left(\begin{array}{cc} \bar{\lambda}w& (1-\bar{\lambda})S^*\\
 (\lambda-1)S&0\end{array} \right)
\]
The write hand side is always a skew-symmetric matrix. This is of course the pull-back of $g^{-1}(dg)$ to $S^1\times \bP(W_{k}^{\perp})$ it should  be looked at as a $1$-form with values in $\mathfrak{u}(k)=\mathfrak{u}(\tau\oplus \tau^{\perp})$ where $\tau$ is the tautological bundle over $\bP(W_{k}^{\perp})$ pulled-back to $S^1\times \bP(W_{k}^{\perp})$. Then we see that
\[\phi^*(g^{-1}dg)(\lambda,L)=\left(\begin{array}{cc} \lambda^{-1}d\lambda& -\overline{\alpha(\lambda)}dS^*_L\\
 \alpha(\lambda)dS_L&0\end{array} \right)\qquad \alpha(\lambda)=\lambda-1
\]  
where $dS$ denotes (the pull-back of) the $1$-form with values in the bundle $\Hom(\tau,\tau^{\perp})\simeq T^{(1,0)}\bP(W_k^{\perp})$, obtained by differentiating  $\id_{\bP(W_{k}^{\perp})}$ and $dS^*$ is the conjugate of $dS$. The important point is that $dS$ is globally defined not just in the chart centered at $L$. 

We need to compute $\wedge^{2k-1}\phi^*(g^{-1}dg)$. Write then
\[\phi^*(g^{-1}dg)=C+B\]
where
\[ C:=\left(\begin{array}{cc} \lambda^{-1}d\lambda &0\\ 0 & 0\end{array}\right)\qquad B:=\left(\begin{array}{cc}0 & -\overline{\alpha(\lambda)}dS^*\\
 \alpha(\lambda)dS& 0\end{array}\right)
\] 
We need also $C_1=\left(\begin{array}{cc} 0 &0\\ 0 & -\lambda^{-1}d\lambda\otimes\id\end{array}\right)$.

The following relations are straightforward
 \[C^2=0,\quad C_1^2=0,\quad  B^2C=CB^2,\quad B^2C_1=C_1B^2,\]
 \[CC_1=C_1C=0,\quad BCB=B^2C_1,\quad C_1BC=0,\quad CBC_1=0.
 \]
 Let $\wedge^0B:=\id$. We prove by induction that for all $j\geq 1$ the following holds.
 \begin{equation}\label{wedge-1}\wedge^{2j-1}\phi^*(g^{-1}dg)=\wedge^{2j-2}B\wedge[jC+(j-1)C_1]+\wedge^{2j-1}B
 \end{equation}
 Indeed the equality is trivially true for $j=1$. Let $\omega:=\phi^*(g^{-1}dg)$ and write $\omega^{j}:=\wedge^j\omega$. Then the following computation finishes the proof of (\ref{wedge-1}):
 \[\omega^{2j-1}=\omega^{2j-3}\wedge\omega^2=[B^{2j-4}((j-1)C+(j-2)C_1)+B^{2j-3}](CB+BC+B^2)=
 \]
 \[=B^{2j-4}((j-1)C + (j-2)C_1)B^2 + B^{2j-3}CB + B^{2j-2}C + B^{2j-1}=
 \]
 \[B^{2j-2}((j-1)C+(j-2)C_1)+B^{2j-2}C_1+B^{2j-2}C+B^{2j-1} = B^{2j-2}(jC+(j-1)C_1)+B^{2j-1}.
 \]
 
 Now $B^{2j-1}$ is block anti-diagonal, hence $\tr B^{2k-1}=0$  and we conclude that:
 \[ \tr\wedge^{2k-1}\omega=\tr B^{2k-2}\wedge (jC+(j-1)C_1)=
 \]
 \[\qquad\qquad\qquad\quad=[(-1)^{k-1}|\alpha(\lambda)|^{2k-2}\lambda^{-1}d\lambda] \cdot\wstr(D^{k-1})
 \]
 where $\wstr\left(\begin{array}{cc} T_1&0\\ 0& T_2 \end{array}\right)=k\tr T_1 -(k-1)T_2$ and $D=\left(\begin{array}{cc} dS^*\wedge dS & 0\\
  0& dS^*\wedge dS \end{array}\right).$
  
  We have thus  written $\tr \wedge^{2k-1}\phi^*g^{-1}(dg)$ as a product of pull-backs of forms from $S^1$ and $\bP(W_{k}^{\perp})$ respectively. Since $\phi$ is orientation reversing we have
    \begin{equation}\label{eqUka}\int_{U(U_{\kappa})}\tr\wedge^{2k-1}g^{-1}(dg)=\int_{S^1}(-1)^{k}|\alpha(\lambda)|^{2k-2}\lambda^{-1}d\lambda\cdot\int_{\bP({W_{k}^{\perp}})}\wstr(D^{k-1})
  \end{equation}
  We use the orientation preserving Cayley transform $t\ra \frac{t-i}{t+i}$ to turn the integral:
  \[\int_{S^1}(-1)^{k}|\alpha(\lambda)|^{2k-2}\lambda^{-1}d\lambda=(-1)^k2^{k-1}\int_{S^1}(1-\Real \lambda)^{k-1}\lambda^{-1}d\lambda
  \] into
  \begin{equation}\label{intS1}2^{2k-1}(-1)^{k}i\int_{\bR}\frac{1}{(1+t^2)^{k}}~dt=(-1)^k2\pi i{2k-2 \choose k-1}
  \end{equation}
  
  In order to compute the integral of $\wstr(D^{k-1})$ we notice first that $\wstr(D^{k-1})$ is a $\mathcal{U}(W_k^{\perp})$ invariant form on $\bP(W_k^{\perp})$ and therefore has to equal a constant multiple times the Fubini-Study volume form of $\bP(W_k^{\perp})$. If we fix a point $L_0\in \bP(W_k^{\perp})$ we can describe  $dS$ and $dS^*$ in terms of a canonical basis of the chart centered at $L_0$ as
  \[ dS_{L_0}=\left(\begin{array}{c} dz_1 \\
   dz_2\\
   \ldots\\
   dz_{k-1}
   \end{array}\right)\qquad dS_{L_0}^*=\left(\begin{array}{cccc} d\bar{z}_1 &   d\bar{z}_2 & \ldots & d\bar{z}_{k-1}
   \end{array}\right)
  \]
  Therefore
  \[ dS_{L_0}^*\wedge dS_{L_0}=\sum_{i=1}^{k-1}d\bar{z}_i\wedge dz_i,\;\;\mbox{and}\;\; (dS_{L_0}\wedge dS_{L_0}^*)_{ij}=dz_i\wedge d\bar{z}_j,\quad \forall\; 1\leq i,j\leq k-1.
  \]
   We get
   \[(S_{L_0}^*\wedge dS_{L_0})^{k-1}=(k-1)!(-1)^{k-1}dz_1\wedge d\bar{z}_1\wedge\ldots\wedge dz_{k-1}\wedge d\bar{z}_{k-1}.
   \]
   One checks rather easily that $(dS_{L_0}\wedge dS_{L_0}^*)^{k-1}$ is diagonal and each diagonal entry is up to a sign equal to $(k-2)!dz_1\wedge d\bar{z}_1\wedge\ldots\wedge dz_{k-1}\wedge d\bar{z}_{k-1}$. In fact
   \[(dS_{L_0}\wedge dS_{L_0}^*)^{k-1}=(-1)^{k-2}(k-2)!dz_1\wedge d\bar{z}_1\wedge\ldots\wedge dz_{k-1}\wedge d\bar{z}_{k-1}\otimes \id
   \]
   and thus
   \[\wstr{D^{k-1}_{L_0}}=(-1)^{k-1}(2k-1)(k-1)!dz_1\wedge d\bar{z}_1\wedge\ldots\wedge dz_{k-1}\wedge d\bar{z}_{k-1}
   \]
 The K\"ahler form at the point $L_0$ on $\bP(W_{k}^{\perp})$ (see \cite{GH}, page 31) is
 \[\eta_{L_0}:=\frac{i}{2\pi}\sum_{j=1}^{k-1}dz_j\wedge d\bar{z}_j
 \]
 and
 \[ \int_{\bP(W_{k}^{\perp})}\wedge^{k-1}\eta=1
 \]
 We compare $\wedge^{k-1}\eta_{L_0}$ and $\wstr{(D^{k-1}_{L_0})}$ and deduce, due to the fact that they are invariant forms that
 \[\wstr{(D^{k-1})}=(-1)^{k-1}(2k-1)\left(\frac{2\pi}{i}\right)^{k-1}\wedge^{k-1}\eta
 \]
 and therefore
 \begin{equation}\label{intCP1}\int_{\bP(W_{k}^{\perp})}\wstr{(D^{k-1})}=(-1)^{k-1}(2k-1)\left(\frac{2\pi}{i}\right)^{k-1}
 \end{equation}
 Putting together (\ref{eqUka}) (\ref{intS1}) and (\ref{intCP1}) we conclude that
 \[\int_{U(U_{\kappa})}\tr\wedge^{2k-1}g^{-1}(dg)=-2\pi i{2k-2\choose k-1}(2k-1)\left(\frac{2\pi}{i}\right)^{k-1}=\left(\frac{2\pi}{i}\right)^{k}\frac{(2k-1)!}{[(k-1)!]^2}.
 \]
   \end{proof}
   \begin{remark} Just as a curiosity note that
    \[\overline{S(U_{\{k\}})}=\{U\in \mathcal{U}(W_k^{\perp})~|~\dim{\Ker(1+U)}\geq 1,\; W_{k-1}/W_{k}\subset {\Ker(1+U)}\}= \mathcal{U}(W_{k-1}^{\perp}),\]
    with $\mathcal{U}(W_{k-1}^{\perp})$ lying inside  $\mathcal{U}(W_k^{\perp})$, via $U\ra -\id_{W_{k-1}/W_k}\oplus U$.
    
 This embedding of $\mathcal{U}(W_{k-1}^{\perp})$ inside $\mathcal{U}(W_k^{\perp})$ is the isotropic space of $(-1,0,\ldots, 0)\in S^{2k-1}$, the unit sphere inside $\mathcal{U}(W_k^{\perp})$. It is no wonder then that the map 
 \[S^1\times \bC\bP^{k-1}\ra S^{2k-1},\qquad(\lambda,L)\ra\phi(\lambda,L)(-1,0,\ldots,0)\]
 is a map of degree $1$.  
   \end{remark}
We will keep the notation $U_{I}$, $S(U_{I})$ and $U(U_{I})$ for the corresponding critical/stable/unstable manifolds in $\mathcal{U}(E)$. 
\begin{theorem}\label{Nico} Let $E\ra B$ be a trivializable hermitian vector bundle of rank $n$ over an oriented manifold with corners $B$ endowed with a compatible connection. Let $g:E\ra E$ be a smooth gauge transform.  Suppose that  a complete flag $E=W_0\supset W_1\supset \ldots \supset W_n=\{0\}$  (equivalently a trivialization of $E$) has been fixed such that $g$ as a section of $\mathcal{U}(E)$ is completely transverse to all the  manifolds $S(U_{I})$ determined by the flag. Then, for each $1\leq k\leq n$ there exists a  flat current $T_k$ such that the following equality of  currents of degree $2k-1$ holds:
\begin{equation}\label{TCkE} \Tc_k(E, g,\nabla)-g^{-1}(S(U_{\{k\}}))=dT_k.
\end{equation}
where 
\begin{eqnarray*}g^{-1}(S(U_{\{k\}}))=\{b\in B~|~\dim{\Ker(1+g_b)}=\dim{\Ker{(1+g_b)\cap (W_{k-1})_b}}=1,\qquad\\
 \dim{\Ker{(1+g_b)}\cap (W_{k})_b}=0\}.\end{eqnarray*}
In particular, when $B$ is compact without boundary, then $\Tc_k(E, g,\nabla)$ and $g^{-1}(S(U_{\{k\}}))$ are Poincar\'e duals to each other.
\end{theorem} 
\begin{proof} We use Corollary \ref{c.principal} for the fiber bundle $\mathcal{U}(E)$ with the flow described in this section and form $\omega=\Tc_k(U^{\tau},\pi^*\nabla)$ where $\pi:\mathcal{U}(E)\ra B$. All the necessary residue computations have been performed. 
\end{proof}
\begin{remark} The current $T_k$ is a spark in the terminology of Harvey and Lawson.
\end{remark}
\begin{remark} \label{transcon} While the transversality condition of $g$ with  the stable manifolds $S(U_{I})$ such that $S(U_{I})\subset \overline{S(U_{\{k\}})}$ is a reasonable requirement for the existence of the current $g^{-1}(S(U_{\{k\}}))$, it seems unnatural that one needs to impose the transversality of $g$ with \emph{all} stable manifolds $S(U_{I})$ in order to obtain (\ref{TCkE}) as one does in  Theorem \ref{Nico}. We conjecture that (\ref{TCkE}) is true under the weaker hypothesis. 
\end{remark}

\section{A Fredholm transgression formula}

In \cite{Qu}, Quillen introduced various smooth differential forms that live on (infinite dimensional) Banach manifolds that are classifying for even and odd $K$-theory. Fix $H$ a complex, separable Hilbert space and let $\mathscr{L}$, $\mathscr{L}^+$, $\mathscr{K}$ be the space of bounded, bounded and self-adjoint, respectively compact operators on $H$. Inside $\mathscr{K}$ there exists a sequence of two-sided ideals, called Schatten spaces, denoted $\Sch^p$. 

The Palais spaces are the spaces of unitary operators $\mathcal{U}^p:=\mathcal{U}(H)\cap (\id_H+\Sch^p)$ and it is well-known that they are smooth Banach manifolds (modelled on $\Sch^p\cap\mathscr{L}^+$) and also classifying for odd $K$-theory, i.e. they have the weak homotopy type of the topological direct limit of spaces $U(\infty):=\lim U(n)$. Quillen defined different families of smooth closed forms $\gamma_{2k-1}^t$, $\gamma_{2k-1,q}^t$, $\Phi_{2k-1}^u$ where $t\in\bC,\; \Real t>0$, $n\in \bN$, $u>0$ which are \emph{well-defined} on the finite dimensional unitary groups $U(n)$ and on certain Palais spaces as follows (see Theorem 5 in op. cit.):
\begin{itemize}
\item[(1)] $\gamma_{2k-1}^t$ on $\mathcal{U}^p$ when $p\leq 2k-1$;
\item[(2)] $\gamma_{2k-1,q}^t$ on $\mathcal{U}^{p}$ when $p\leq 2k-1+2q$;
\item[(3)] $\Phi_{2k-1}^u$ on  $\mathcal{U}^{p}$ for all $p$.
\end{itemize}
The forms $\gamma_{2k-1}^t$, $\gamma_{2k-1,q}^t$ and $\Phi_{2k-1}^1$ are all cohomologous and represent the degree $2k-1$-component of the odd Chern character $\ch_{2k-1}$ of the universal $K^{-1}$-class, i.e. the class induced by the identity map $\id_{U(\infty)}$. We will call each of them a Quillen form.

In \cite{Ci1} we  gave alternative construction to  the pull-backs $\varphi^*{\ch_{2k-1}}$ when $\varphi: B\ra \mathcal{U}^{p}$ when $B$ is a compact oriented manifold and $\varphi$ is smooth. In fact, the theory works for maps $\varphi:B\ra \mathcal{U}^{-}$, where $\mathcal{U}^-$ is the open subset of unitary operators $U$ such that $1+U$ is Fredholm. This is another manifold classifying space for $K^{-1}$ that contains  $\mathcal{U}^p$ for every $p$, however it does not come with any easy to describe smooth differential forms on it. Under a  certain finite set\footnote{hence it applies to "generic" smooth maps} of transversality conditions, the classes $\varphi^*{\ch_{2k-1}}$ were described (up to multiplication by a rational number) via preimages $\varphi^{-1}\overline{Z_{\{k\}}}$ where $\overline{Z_{\{k\}}}$ are stratified subspaces of codimension $2k-1$ in $\mathcal{U}^-$. In fact, the Schubert cell $Z_{\{k\}}$ is defined by the same incidence relations as the stable manifold $S(U_{\{k\}})$ we saw in last section. 

We used local (sheaf) cohomology in order to associate to  a finite codimensional, cooriented stratified space a cohomology class, which behaves well under \emph{transverse} pull-back. We take a different path here and show that in fact under a different but still finite set of transversality condition one can define the current $\varphi^*Z_{\{k\}}$ and this is Poincar\'e dual to $\frac{(-1)^{k-1}}{k-1)!} \varphi^*\Omega_k$ where $\Omega_k$ is a Quillen form. In fact something stronger is true.

We will assume that a complete flag
\[ H\supset W_0\supset\ldots \supset W_k\supset
\]
has been fixed with $\codim{W_k}=k$. For every $I=\{i_1<\ldots<i_k\}$ a $k$-tuple of positive integers, let 
\begin{eqnarray*} Z_{I}^p:=\{U\in \mathcal{U}^p~|~\dim{\Ker(1+U)}=k, \dim{\Ker(1+U)\cap W_{m}}=k-p, \; \forall 0\leq p\leq k,\; \\
 \forall i_p\leq m<i_{p+1}\}
\end{eqnarray*}
where as usual $i_0=0$, $i_{k+1}=\infty$. 

We notice that for every smooth map $\varphi:B\ra \mathcal{U}^p$ and every $p$ from a compact manifold $B$ there exists a subspace $W_N$ of the flag such that ${\Ker(1+\varphi(b))\cap W_N}=\{0\}$ for all $b\in B$. This is because if $U\in  \mathcal{U}^p$ then $1+U$ is Fredholm and this is an open condition. It follows that the collection of transversality condition $\varphi\pitchfork Z_I^p$ is trivially satisfied if there exists $ a>N$ such that $a\in I$ since then $\varphi^{-1}(Z_I^p)=\emptyset$. Hence $\varphi\pitchfork Z_I^p$ for every $I$ is a generic condition. 

\begin{theorem}\label{thm71} Let $\varphi:B\ra \mathcal{U}^p$ be a smooth map from a compact, oriented manifold $B$, possibly with corners such that $\varphi\pitchfork Z_I^p$ for every $I$. Let $\Omega_k$ be a Quillen form of degree $2k-1$ that makes sense on $\mathcal{U}^p$. Then for every such $\Omega_k$, there exists a  flat current $T_k$ such that:
  \begin{equation}\label{lasteq} \varphi^{-1}Z_{\{k\}}-(-1)^{k-1}(k-1)!\varphi^*\Omega_k=dT_k.\end{equation}
  In particular, when $B$ has no boundary, $\frac{(-1)^{k-1}}{(k-1)!}\varphi^{-1}Z_{\{k\}}^p$ represents the Poincar\'e dual of $\ch_{2k-1}([\varphi])$, where $[\varphi]\in K^{-1}(B)$ is the natural odd $K$ theory class determined by $\varphi$.
\end{theorem}
\begin{proof} We use symplectic reduction. For each linear subspace $W\subset H$ of finite codimension there exists a smooth (even real analytic) map $\mathscr{R}^{W}:\mathcal{U}^p_W\ra U(W^{\perp})$, where 
\[ \mathcal{U}^p_W:=\{U\in \mathcal{U}^p~|~\Ker(1+U)\cap W=\{0\}\}
\]
is an open subset of $\mathcal{U}^p$. The expression of $\mathcal{R}^W$ relative to the decomposition $U=\left(\begin{array}{cc} X& Y\\ Z& T\end{array}\right)$ vis-a-vis $H=W\oplus W^{\perp}$ is:
\[\mathcal{R}^W(U)=T-Z(1+X)^{-1}Y.
\]
The map $\mathcal{R}^W$\footnote{It is called symplectic reduction because under Arnold's theorem which identifies  the unitary group $\mathcal{U}^p$ with the  (Hermitian) Lagrangian Grassmannian of Schatten class $p$ it corresponds  to the homonymous process well-known in symplectic topology.} has some nice properties. For example it can be shown that together with the "0-section":
\[\iota:U(W^{\perp})\hookrightarrow \mathcal{U}^p_W,\qquad U\ra -\id_{W}\oplus U
\]
is diffeomorphic to a vector bundle over $U(W^{\perp})$ (see Corollary 4.1 in \cite{Ci1}). Hence by choosing $W=W_N$ a subspace of the flag as mentioned before the proof we get that $\Imag \varphi\subset \mathcal{U}^p_W$ and there exists a smooth homotopy $h:[0,1]\times B\ra \mathcal{U}^p$ between $\psi:=\mathcal{R}^W\circ \varphi$ and $\varphi$. All the Quillen forms $\Omega_k$ have finite dimensional counterparts $\Omega^{W^{\perp}}_k$ such that $\iota^*\Omega_k=\Omega^{W^{\perp}}_k$. Hence there exists a smooth form $\beta(\Omega)$ on $B$ such that
\[ \varphi^*\Omega_k- \psi^*\Omega_k=d(\beta(\Omega_k)).
\]
Another important property  is that $\mathcal{R}^W(Z_I^p)=S(U_{I})$ and in fact $\mathcal{R}^W)^{-1}(S(U_{I}))=Z_I^p$. It follows that $\varphi^{-1}(Z_{\{k\}}^p)=\psi^{-1}(S(U_{\{k\}}))$. Therefore we can use Theorem \label{Nico} to conclude that (\ref{lasteq}) holds for $\Omega_k=\gamma_{2k-1}^1$ which coincides with $\Tch_k$. Since the Quillen forms of degree $2k-1$ are all cohomologous in the finite dimensional case we get the result for such forms.
\end{proof}
\begin{remark} The coorientation (implicitly the orientation) of $\varphi^{-1}Z_{\{k\}}$ is discussed in  detail in \cite{Ci1}.
\end{remark}
\begin{remark} It seems that  there are too many transversality conditions in Theorem \ref{thm71} (see Remark \ref{transcon} and compare with Proposition 7.1 from \cite{Ci1}).
\end{remark}


\begin{thebibliography}{99}
	
\bibitem{CS} J. Cheeger, J. Simons, \emph{Differential characters and geometric invariants}, Lecture Notes in Math., vol. 1167, Springer-Verlag, New York, 1985, pp. 50-80.	
\bibitem{Ci1} D. Cibotaru, \emph{The odd Chern character and index localization formulae}, Comm. An. Geom, {\bf 19}, (2011), 209-276.
\bibitem{Ci2} D. Cibotaru, \emph{Vertical  flows and a general currential homotopy formula}, Indiana U. Math. J, {\bf 65} (2016),  93-169. 
\bibitem{Ci3} D. Cibotaru, \emph{Chern-Gauss-Bonnet and Lefschetz Duality from a currential point of view}, Adv. Math., {\bf 317} (2017), 718-757.
\bibitem{Ci4} D. Cibotaru, \emph{Vertical Morse-Bott-Smale Flows and Characteristic
Forms.} Indiana Univ. Math J., {\bf 65} (2016), 1089-1135.


\bibitem{DV} I.A Dynnikov, A.P. Veselov, \emph{Integrable gradient flows and Morse Theory}, St.Petersburg Math.J. {\bf 8} (1997),  429-446.
\bibitem{Fe} H. Federer, \emph{Geometric measure theory}, Grundlehren der mathematischen Wissenschaften, Springer-Verlag, New
York, 1969.
\bibitem{GH} P. Griffiths, J. Harris, \emph{Principles of Algebraic Geometry}, John Wiley \& Sons, 1978.


\bibitem{HL1} R. Harvey, B.  Lawson Jr.   \emph{Finite Volume Flows and Morse Theory}, Annals of Math. {\bf 153} (2001), no.1, 1-25.
\bibitem{HL2} R. Harvey, B. Lawson Jr.,  \emph{A Theory of Characteristic Currents Associated with a Singular Connection}, Ast\'erisque {\bf 213}, Soc. Math. de France, Montrouge, France, 1993.
\bibitem{HL3} R. Harvey, B. Lawson Jr., \emph{Geometric Residue Theorems}, Amer. J. Math., {\bf 117} (1995), no. 4, 829-873.
\bibitem{HLZ4} R. Harvey, B. Lawson Jr., J. Zweck, \emph{The de Rham-Federer theory of differential characters and character
duality}, Amer. J. Math. {\bf 125} (2003), no. 4, 791-847.
\bibitem{HM} R. Harvey, G. Minervini, \emph{Morse Novikov theory and cohomology with forward supports}, Math. Ann. {\bf 335}, 787-818.
\bibitem{La} J. Latschev, \emph{Gradient flows of Morse-Bott functions}, Math. Ann. {\bf 318} (2000), 731-759.
\bibitem{Le} J. Lee, \emph{Introduction to smooth manifolds}, Sec. Ed., Springer, 2013.
\bibitem{LM} T. Lydman, C. Manolescu, \emph{The equivalence of two Seiberg-Witten Floer homologies}, Ast\'erisque {\bf 399}, 2018.
\bibitem{M} G. Minervini, \emph{A current approach to Morse and Novikov Theories}, Rend. Mat.  {\bf 37} (2015), 95-195.
\bibitem{Ni} L. Nicolaescu, \emph{Schubert calculus on the Grassmannian of hermitian lagrangian spaces}, Adv. Math., {\bf 224} (2010), 2361-2434.
\bibitem{Oli} W. Pereira, \emph{Fluxos n\~ao-tame de correntes e teoria Chern-Weil impar}, PhD Thesis (in portuguese), Universidade Federal do Cear\'a, Fortaleza, 2018.
\bibitem{Qu} D. Quillen, \emph{Superconnection character forms and the Cayley transform}, Topology, {\bf 27}, 211-238, 1988.
\bibitem{Sh}  L. Shilnikov, \emph{Methods of non-qualitative theory in nonlinear dynamics Part I.}
Nonlinear Science, World Scientific, 1998.
\bibitem{SS} J. Simons, D. Sullivan, \emph{Structured Vector Bundles Define Differential $K$-Theory}, Quanta of Maths, Clay Math. Proc. {\bf 11}, AMS, Clay Math. Inst., 2010.
\bibitem{Tr} F. Treves, \emph{Topological Vector Spaces, Distributions and Kernels}, Academic Press, 1967.


\end{thebibliography}
\end{document}